\documentclass{amsart}
\usepackage{amsmath,amsthm}
\usepackage{amsfonts,amssymb}
\usepackage{enumerate}
\usepackage{mathrsfs}
\usepackage{tikz}
\usepackage{float}
\copyrightinfo{2019}{H. De Bie \emph{et al.}}

\newtheorem{theorem}{Theorem}
\newtheorem{proposition}{Proposition}
\newtheorem{lemma}{Lemma}

\newtheorem{corollary}{Corollary}

\theoremstyle{definition}
\newtheorem{definition}{Definition}

\theoremstyle{remark}
\newtheorem{remark}{Remark}

 \newcommand{\R}{\mathbb{R}}
 \newcommand{\C}{\mathbb{C}}
 \newcommand{\N}{\mathbb{N}}
 \newcommand{\Z}{\mathbb{Z}}

\newcommand{\ospq}{\mathfrak{osp}_q(1\vert 2)}

\newcommand{\akk}{a^{(i,k)}}
\newcommand{\bkk}{b^{(i,k)}}
\newcommand{\ckk}{c^{(i,k)}}
\newcommand{\dkk}{d^{(i,k)}}
\newcommand{\vjk}[1]{\vert\mathbf{j}_{#1}^{(i,k)}\vert}
\newcommand{\vj}[2]{\vert\mathbf{j}_{#2}^{(#1)}\vert}
\newcommand{\vjj}[1]{\vert\mathbf{j}_{#1}\vert}
\renewcommand{\vss}[1]{\vert\mathbf{s}_{#1}\vert}
\newcommand{\vst}[1]{\vert\tilde{\mathbf{s}}_{#1}\vert}

\newcommand{\vp}[1]{\vert\mathbf{p}_{#1}\vert}
\newcommand{\bj}{\mathbf{j}}
\newcommand{\bp}{\mathbf{p}}
\newcommand{\bk}{\mathbf{k}}

\newcommand{\bh}{\mathbf{h}}
\newcommand{\vh}[1]{\vert\mathbf{h}_{#1}\vert}
\newcommand{\vk}[1]{\vert\mathbf{k}_{#1}\vert}

\newcommand{\alphakk}{\alpha^{(i,k)}}
\newcommand{\betakk}{\beta^{(i,k)}}
\newcommand{\gammakk}{\gamma^{(i,k)}}
\newcommand{\deltakk}{\delta^{(i,k)}}

\newcommand{\betak}{\widetilde{\beta}_k}
\newcommand{\gammak}{\widetilde{\gamma}_k}
\newcommand{\deltak}{\widetilde{\delta}_k}

\newcommand{\ak}{\widetilde{a}^{(i,r)}_k}
\newcommand{\bkm}{\widetilde{b}^{(i,r)}_k}
\newcommand{\ck}{\widetilde{c}^{(i,r)}_k}
\newcommand{\dk}{\widetilde{d}^{(i,r)}_k}

\newcommand{\nuu}{\boldsymbol{\nu}}
\newcommand{\bb}{\mathfrak{b}}

\begin{document}
\title[The $q$-Bannai--Ito algebra and multivariate $(-q)$-Racah polynomials]{The $q$-Bannai--Ito algebra and multivariate $(-q)$-Racah and Bannai-Ito polynomials}
\author{Hendrik De Bie}
\email{Hendrik.DeBie@UGent.be}
\address{Department of Electronics and Information Systems\\Faculty of Engineering and Architecture\\Ghent University\\Building S8, Krijgslaan 281, 9000 Gent\\ Belgium.}

\author{Hadewijch De Clercq}
\email{Hadewijch.DeClercq@UGent.be}
\address{Department of Electronics and Information Systems\\Faculty of Engineering and Architecture\\Ghent University\\Building S8, Krijgslaan 281, 9000 Gent\\ Belgium.}

\date{\today}
\keywords{\(q\)-Racah polynomials, Askey--Wilson polynomials, Bannai--Ito polynomials, multivariate polynomials, bispectrality, Bannai--Ito algebra, Askey--Wilson algebra, Askey scheme}
\subjclass[2010]{33C50, 33D45, 33D50, 33D80, 39A13, 81R50} 
	
	\begin{abstract}
		The Gasper and Rahman multivariate \((-q)\)-Racah polynomials appear as connection coefficients between bases diagonalizing different abelian subalgebras of the recently defined higher rank \(q\)-Bannai--Ito algebra \(\mathcal{A}_n^q\). Lifting the action of the algebra to the connection coefficients, we find a realization of \(\mathcal{A}_n^q\) by means of difference operators. This provides an algebraic interpretation for the bispectrality of the multivariate \((-q)\)-Racah polynomials, as was established in Iliev, \emph{Trans. Amer. Math. Soc.}, {\bf 363}(3) (2011), 1577--1598.	
		
		Furthermore, we extend the Bannai--Ito orthogonal polynomials to multiple variables and use these to express the connection coefficients for the \(q = 1\) higher rank Bannai--Ito algebra \(\mathcal{A}_n\), thereby proving a conjecture from De Bie et al., \emph{Adv. Math.} {\bf 303} (2016), 390--414. We derive the orthogonality relation of these multivariate Bannai--Ito polynomials and provide a discrete realization for \(\mathcal{A}_n\).
	\end{abstract}
	
	\maketitle
	\tableofcontents
	
	\section{Introduction}
	
	\subsection{The Bannai--Ito algebra and its generalizations}
	The Bannai--Ito algebra \(\mathcal{A}_3\) is the associative algebra over \(\C\) with three generators \(\Gamma_{12},\Gamma_{23},\Gamma_{13}\) and relations
	\begin{align}
	\label{BI alg rank 1}
	\{\Gamma_{12}, \Gamma_{23}\} = \Gamma_{13} + \alpha_{13}, \quad  \{\Gamma_{12}, \Gamma_{13}\} = \Gamma_{23} + \alpha_{23}, \quad  \{\Gamma_{13}, \Gamma_{23}\} = \Gamma_{12} + \alpha_{12},
	\end{align}
	where \(\alpha_{ij}\) are structure constants. It was introduced in \cite{Tsujimoto&Vinet&Zhedanov-2012} as the algebraic structure underlying the Bannai--Ito orthogonal polynomials, which sit atop the so-called \((-1)\)-Askey scheme. It appeared also in \cite{DeBie&Genest&Vinet-2016-2} as the symmetry algebra of the three-dimensional Dirac--Dunkl operator with \(\mathbb{Z}_2^3\) reflection group. A central extension of this algebra can be embedded in the threefold tensor product of the Lie superalgebra \(\mathfrak{osp}(1\vert 2)\) \cite{Genest&Vinet&Zhedanov-2012}. This observation allowed to generalize the Bannai--Ito algebra in several ways.
	
	One possibility is to \(q\)-deform the underlying Lie superalgebra to the quantum superalgebra \(\mathfrak{osp}_q(1\vert 2)\). This led to the construction of a \(q\)-deformed Bannai--Ito algebra \(\mathcal{A}_3^q\) in \cite{Genest&Vinet&Zhedanov-2016}, where its connection with the univariate \((-q)\)-Racah polynomials was uncovered. The algebra \(\mathcal{A}_3^q\) turned out to be isomorphic to the universal Askey--Wilson algebra \cite{Terwilliger-2011}, which is a central extension of the Askey--Wilson or Zhedanov algebra \cite{Zhedanov-1991}, under a transformation \(q\to -q^2\). 
	
	A second generalization was proposed in \cite{DeBie&Genest&Vinet-2016}, where the connection with Dirac--Dunkl operators was used to extend the Bannai--Ito algebra to higher rank. The resulting algebra \(\mathcal{A}_n\) was observed to be the symmetry algebra of a superintegrable system with reflections on the \(n\)-sphere \cite{DeBie&Genest&Lemay&Vinet-2017} and of the \(\mathbb{Z}_2^n\) Dirac--Dunkl equation. Generalizations to other reflection groups were obtained in \cite{Genest&Lapointe&Vinet-2018} and \cite{DeBie&Oste&VanderJeugt-2018}. Other types of higher rank extensions, notably in connection with more general Lie algebras and quantum groups, are still actively investigated. Recently a classical Askey--Wilson algebra associated to the Lie algebra \(\mathfrak{sl}_n\) was constructed in \cite{Baseilhac&Crampe&Pimenta-2018}, see also \cite{Stokman-2018}.
	
	The two mentioned techniques were combined in our recent work \cite{DeBie&DeClercq&vandeVijver-2018}. We introduced the higher rank \(q\)-Bannai--Ito algebra \(\mathcal{A}_n^q\), which is both a higher rank extension of \(\mathcal{A}_3^q\) and a \(q\)-deformation of \(\mathcal{A}_n\), and which will be at the center of our attention throughout this paper. To each subset \(A\subseteq \{1,\dots,n\}\) we associated an algebra generator \(\Gamma_A^q\), yielding the algebraic relations
	\begin{equation}
	\label{q-BI rel.}
	\{\Gamma_A^q,\Gamma_B^q\}_q = \Gamma_{(A\cup B)\setminus(A\cap B)}^q + (q^{1/2}+q^{-1/2})\left(\Gamma_{A\cap B}^q\Gamma_{A\cup B}^q + \Gamma_{A\setminus(A\cap B)}^q\Gamma_{B\setminus(A\cap B)}^q\right),
	\end{equation}
	where \(\{A,B\}_q = q^{1/2}AB+q^{-1/2}BA\) is the \(q\)-anticommutator and \(A,B\) are sets of integers between 1 and \(n\) subject to certain conditions.
	As in the rank one case, this algebra can equally be considered a higher rank universal Askey--Wilson algebra. A different extension method was studied in \cite{Post&Walter-2017} for the rank two case, corresponding to \(n =4\).
	The algebra \(\mathcal{A}_n^q\) is also closely related to superintegrable quantum systems. In \cite{DeBie&DeClercq&vandeVijver-2018} we established a model of \(q\)-difference operators and reflections which has the higher rank \(q\)-Bannai--Ito algebra as its symmetry algebra. We will repeat the definition of this so-called \(\mathbb{Z}_2^n\) \(q\)-Dirac--Dunkl model, as well as our extension algorithm for multifold tensor products in Sections \ref{Paragraph - ospq and the extension process} and \ref{Paragraph q-Dirac-Dunkl model}. 
	
	We will consider a class of abelian subalgebras for \(\mathcal{A}_n^q\) and to each such subalgebra we will attach a set of joint eigenvectors in an \(n\)-fold tensor product of positive discrete series representations of \(\mathfrak{osp}_q(1\vert 2)\). We give an explicit construction for these vectors in Definition \ref{definition bases final} and motivate this definition throughout Sections \ref{Paragraph - Unitary irreducible modules} and \ref{Paragraph Decompositions}. They are to be compared with the canonical bases for finite-dimensional \(U_q(\mathfrak{sl}_2)\)-modules obtained in \cite{Frenkel&Khovanov-1997} and \cite{Khovanov-1997}. The overlap coefficients between two such vectors, each corresponding to a different abelian subalgebra, will be expressed in Theorem \ref{th - Overlap q-case} in terms of a special class of \(q\)-orthogonal polynomials, namely the Gasper and Rahman multivariate \((-q)\)-Racah polynomials \cite{Gasper&Rahman-2007}. As such, these polynomials arise as \(3nj\)-symbols for \(\ospq\), in the sense of \cite{Kirillov&Reshetikhin-1989} and \cite{Rosengren-2007}. 
	
	\subsection{Orthogonal polynomials in multiple variables}
	The \(q\)-Racah polynomials are amongst the most general of all \(q\)-hypergeometric orthogonal polynomials. They can be obtained upon reparametrization and truncation from the Askey--Wilson polynomials, which are at the top of the \(q\)-Askey scheme and hence have all other \(q\)-orthogonal polynomials in the scheme as limiting cases. Two different methods have been proposed to extend the Askey--Wilson and \(q\)-Racah polynomials to multiple variables. 
	
	One possibility stems from the theory of symmetric functions and was exploited by Koornwinder \cite{Koornwinder-1992} based on previous work by Macdonald \cite{Macdonald-1988}. Restricting the orthogonality measure to a discrete support, one obtains a class of multivariable \(q\)-Racah polynomials \cite{Stokman&vanDiejen-1998}. These Macdonald-Koornwinder polynomials and their connection to affine Hecke algebras have been intensively studied, see for example \cite{Stokman-2000}.
	
	A second type of extension, by means of coupled products of univariate Askey--Wilson polynomials, was introduced by Gasper and Rahman \cite{Gasper&Rahman-2005}, based on a construction by Tratnik \cite{Tratnik-1991} for multivariable \((q = 1)\) Racah polynomials. Suitable truncation led to a class of multivariate \(q\)-Racah polynomials \cite{Gasper&Rahman-2007}, which will be of interest in this paper. 
	
	The univariate Askey--Wilson and \(q\)-Racah polynomials are well-known to be bispectral in the sense of Duistermaat and Gr\"{u}nbaum \cite{Duistermaat&Grunbaum-1986}. Writing them as \(p_n(x)\), these polynomials can be defined both through a recursion relation in \(x\) and one in \(n\). 
	Iliev extended these bispectrality properties to the multivariate case in \cite{Iliev-2011}. He constructed two commutative algebras of difference and \(q\)-difference operators diagonalized by these polynomials and thereby discovered a duality relation between the variables \(x_i\) and the degrees \(n_i\).
	
	So far, this bispectrality has received little algebraic foundation. In this paper we will connect the multivariate \((-q)\)-Racah polynomials to the higher rank \(q\)-Bannai--Ito algebra \(\mathcal{A}_n^q\) and thereby cast Iliev's difference operators in a larger algebraic framework. More precisely, we will show in Theorem \ref{th discrete realization} how these operators give rise to a realization of \(\mathcal{A}_n^q\), which is hence generated by Iliev's aforementioned commutative algebras. 
	
	\subsection{Algebraic underpinning for the Askey--Wilson polynomials}
	In previous literature, profound connections were established between quantum groups and algebras on the one hand and Askey--Wilson polynomials on the other. We refer to \cite{Noumi&Mimachi-1990}, \cite{Koelink-1996} and \cite{Rosengren-2000} for an overview in the univariate case. More recently, similar relations were unveiled for their multivariate counterparts.
	
	Genest, Iliev and Vinet studied in \cite{Genest&Iliev&Vinet-2018} the coupling coefficients for multifold tensor products of \(\mathfrak{su}_q(1,1)\). These were identified as multivariate \(q\)-Racah polynomials through their appearance as connection coefficients between bases of multivariate \(q\)-Hahn and \(q\)-Jacobi polynomials. 
	
	A related problem was studied for the \(q\)-Onsager algebra \cite{Baseilhac-2006}, which is known to have the Askey--Wilson algebra as a homomorphic image \cite{Terwilliger-2011}. In \cite{Baseilhac&Vinet&Zhedanov-2017}, an iteration of coproducts of the quantum affine algebra \(U_q(\widehat{\mathfrak{sl}_2})\) allowed to identify two sets of \(N\) mutually commuting \(q\)-difference operators, which together generate several copies of the \(q\)-Onsager algebra. The polynomial eigenbases of the two sets of \(q\)-difference operators were considered and their overlap coefficients were expressed as entangled products of univariate \(q\)-Racah polynomials. However, these did not coincide with the Gasper and Rahman multivariable \(q\)-Racah polynomials. Altering the difference operators nevertheless allowed to obtain the Gasper and Rahman polynomials as generators for a family of infinite-dimensional modules of the rank one \(q\)-Onsager algebra in \cite{Baseilhac&Martin-2015}.
	
	In \cite{Groenevelt-2018} an iteration of coproducts of twisted primitive elements was used to relate the quantum group \(U_q(\mathfrak{su}(1,1))\) to the multivariate Askey--Wilson and Al-Salam-Chihara polynomials. The obtained difference equations coincide with the relations in \cite{Iliev-2011}.
	
	A similar recoupling for the higher rank Racah algebra was considered in \cite{DeBie&Genest&vandeVijver&Vinet-2016}. Its overlap coefficients were obtained in \cite{DeBie&vandeVijver-2018} in terms of Tratniks multivariate \(q = 1\) Racah polynomials and the higher rank Racah algebra was realized by means of the difference operators for Racah polynomials defined by Geronimo and Iliev \cite{Geronimo&Iliev-2010}.
	
	Also in the classical setting, bispectral orthogonal polynomials have been given profound algebraic underpinnings.
	In \cite{Iliev-2012} Iliev studies Krawtchouk polynomials in \(n\) variables and relates their spectral properties to representations of the Lie algebra \(\mathfrak{sl}_{n+1}(\C)\). Iliev and Xu established the bispectrality of a class of multivariate Hahn polynomials in \cite{Iliev&Xu-2017}. The corresponding difference operators arise as symmetries of the discrete generic quantum superintegrable system and were shown to generate a realization of the Kohno-Drinfeld Lie algebra.
	
	All of these references, except for the latter two, focus solely on (\(q\)-)shifts in the variables of the polynomials, thus neglecting the duality with the polynomial degrees. However, we will allow for discrete shifts in the degrees as well, thereby exploiting the full bispectrality. Another novelty in our approach is that we obtain explicit algebraic relations connecting both types of difference operators, thereby recovering the \(\mathcal{A}_n^q\) identities (\ref{q-BI rel.}). Moreover, none of the mentioned works goes beyond the iteration of coproducts as a technique for extension to higher rank. We will instead use the novel extension algorithm that we introduced in \cite{DeBie&DeClercq&vandeVijver-2018}, to provide a complete description of the considered algebras. Referring to the notation in (\ref{q-BI rel.}), we will index our algebra generators by subsets of \(\{1,\dots,n\}\). In previous work only sets of consecutive integers could be considered, whereas our method makes it possible to obtain explicit expressions for any possible set.
	
	We emphasize that all results obtained in this paper are easily transferred to the higher rank Askey--Wilson algebra under the transformation \(q\to -q^2\), by the isomorphism with \(\mathcal{A}_n^q\) established in \cite{DeBie&DeClercq&vandeVijver-2018}. The considered overlap coefficients then become multivariate \((q^2)\)-Racah polynomials and generate infinite-dimensional modules for the higher rank Askey--Wilson algebra.
	
	\subsection{Specialization for $q\to 1$}
	In the limit \(q\to 1\), the univariate \((-q)\)-Racah polynomials reduce to the Bannai--Ito polynomials, a class of orthogonal polynomials first considered in \cite{Bannai&Ito-1984}. They satisfy both a three-term recurrence and three-term difference relation, thus giving rise to an associated Leonard pair \cite{Terwilliger-2001}. They appeared in \cite{Genest&Vinet&Zhedanov-2012} as connection coefficients for the original \(q = 1\) Bannai--Ito algebra (\ref{BI alg rank 1}). In \cite{DeBie&Genest&Vinet-2016} a construction by means of Dunkl operators was proposed for the higher rank Bannai--Ito algebra, of which \(\mathcal{A}_n^q\) is a \(q\)-deformation. Its connection coefficients were conjecturally stated to be multivariate extensions of Bannai--Ito polynomials, which had however not been defined at the time. In Section \ref{Section - limit q to 1} we will affirm and prove this statement: we propose a definition for multivariate Bannai--Ito polynomials following the methods of Tratnik and Gasper and Rahman. We will obtain explicit expressions for the overlap coefficients in terms of these polynomials in Theorem \ref{th - Overlap q = 1}. Subsequently we derive their orthogonality and bispectrality properties. We also show that our definitions coincide with the Bannai--Ito polynomials recently introduced in \cite{Lemay&Vinet-2018} for the bivariate case, up to a modification of two parameters. 
	
	\subsection{Outline of the contents}
	The paper is organized as follows. In Section \ref{Section - Higher rank q-BI}, we recall the necessary prerequisites on the higher rank \(q\)-Bannai--Ito algebra and the discrete series representation of \(\mathfrak{osp}_q(1\vert 2)\). We repeat the definition of the \(\mathbb{Z}_2^n\) \(q\)-Dirac--Dunkl model in Section \ref{Paragraph q-Dirac-Dunkl model}. To each maximal abelian subalgebra of \(\mathcal{A}_n^q\) we associate a class of joint eigenvectors in a coupled \(\ospq\)-module in Section \ref{Paragraph - Unitary irreducible modules}. Consequently, we obtain a decomposition of this module in Section \ref{Paragraph Decompositions} as a motivation for our solution to the spectral problem in Definition \ref{definition bases final}. The connection coefficients for \(\mathcal{A}_n^q\) will be derived in Section \ref{Paragraph - Connection coefficients}. We first consider the univariate case before going to multiple variables. In Section \ref{Section - Discrete realization} we translate some of the results from \cite{Iliev-2011} to \((-q)\)-Racah polynomials and use these to construct a discrete realization for \(\mathcal{A}_n^q\). Finally, we introduce the multivariate Bannai--Ito polynomials in Section \ref{Section - limit q to 1}. We motivate their definition and obtain their orthogonality and bispectrality properties. The results from previous sections are subjected to a limit \(q \to 1\), leading to the connection coefficients and their interpretation as generators for infinite-dimensional modules of the higher rank \(q = 1\) Bannai--Ito algebra. We end with some conclusions and an outlook.
	
\section{The higher rank $q$-Bannai--Ito algebra}
\label{Section - Higher rank q-BI}

We will first recall some prerequisites about the \(q\)-Bannai--Ito algebra and \(\mathfrak{osp}_q(1\vert 2)\). 

\subsection{The quantum algebra $\mathfrak{osp}_q(1\vert 2)$ and the extension process}
\label{Paragraph - ospq and the extension process}
Let \(q\) be a complex number which is not a root of unity. For any complex number or any operator \(A\) we will denote by \([A]_q\) the \(q\)-number
\[
[A]_q = \frac{q^A-q^{-A}}{q-q^{-1}}.
\]
We will write \([i;j]\) for the set of natural numbers \(\{i,i+1,\dots,j\}\).
The \(q\)-anticommutator of two operators \(A\) and \(B\) is defined as
\[
\{A,B\}_q = q^{1/2}AB+q^{-1/2}BA.
\]

The quantum superalgebra \(\ospq\) is the \(\mathbb{Z}_2\)-graded unital associative algebra with generators \(A_+,A_-, K, K^{-1}\) and the grade involution \(P\), satisfying the commutation relations
\begin{align}
\label{def - ospq with K}
\begin{split}
KA_{+}K^{-1} = q^{1/2} A_{+}, \quad KA_{-}K^{-1} = q^{-1/2} A_{-}, \quad \{A_+,A_-\} = \frac{K^2-K^{-2}}{q^{1/2}-q^{-1/2}}, \\ \quad \{P,A_{\pm}\} = 0, \quad 
\ [P,K] = 0, \quad [P,K^{-1}] = 0, \quad KK^{-1} = K^{-1}K = 1, \quad P^2 = 1.
\end{split}
\end{align}
The Casimir operator has the expression
\begin{equation}
\label{Gamma^q}
\Gamma^q = \left(-A_+A_-+\frac{q^{-1/2}K^2-q^{1/2}K^{-2}}{q-q^{-1}} \right)P.
\end{equation}
It is easily checked that \(\Gamma^q\) commutes with all algebra elements.

One can endow \(\ospq\) with a coproduct \(\Delta: \ospq \rightarrow \ospq^{\otimes 2} \) acting on the generators as \cite{Genest&Vinet&Zhedanov-2016}
\begin{equation}
\label{Coproduct}
\Delta(A_{\pm}) = A_{\pm}\otimes KP + K^{-1}\otimes A_{\pm}, \quad \Delta(K) = K\otimes K, \quad \Delta(P) = P\otimes P.
\end{equation}
We will consider the multifold tensor product algebra \(\ospq^{\otimes n}\) with its standard product law as in \cite{Genest&Vinet&Zhedanov-2016}:
\begin{equation}
\label{Standard product law}
(a_1\otimes a_2\otimes\dots\otimes a_n)(b_1\otimes b_2\otimes\dots\otimes b_n) = a_1b_1\otimes a_2b_2\otimes\dots\otimes a_nb_n.
\end{equation}

The rank \(n-2\) \(q\)-Bannai--Ito algebra was constructed inside \(\ospq^{\otimes n}\) in \cite{DeBie&DeClercq&vandeVijver-2018} by means of an extension process, allowing to lift the Casimir operator \(\Gamma^q\) to \(n\)-fold tensor products. Let us denote by \(\mathcal{I}\) the subalgebra of \(\ospq\) generated by the elements \(A_-K,A_+K,K^2P\) and \(\Gamma^q\), and let \(\tau:\mathcal{I}\to \ospq\otimes\mathcal{I}\) be the algebra morphism defined by 
\begin{align}
\begin{split}
\label{def:tau isomorphism}
\tau(A_-K) = & \ K^2P\otimes A_-K, \\
\tau(A_+K) = & \ (K^{-2}P\otimes A_+K)+q^{-1/2}(q-q^{-1})(A_+^2P\otimes A_-K) \\&+ q^{-1/2}(q^{1/2}-q^{-1/2})(A_+K^{-1}P\otimes K^2P)\\& +q^{-1/2}(q-q^{-1})(A_+K^{-1}P\otimes \Gamma^q), \\
\tau(K^2P) = & \ 1\otimes K^2P - (q-q^{-1})(A_+K\otimes A_-K),\\
\tau(\Gamma^q) = & \ 1\otimes \Gamma^q.
\end{split}
\end{align}
One can easily check that \(\mathcal{I}\) is a left coideal subalgebra of \(\ospq\), as well as a left \(\ospq\)-comodule with coaction \(\tau\).

To each set \(A\subseteq[1;n]\) one can now associate an element \(\Gamma_A^q\in\ospq^{\otimes n}\), constructed as
\begin{equation}
\label{Gamma_A^q}
\Gamma_A^q = 1^{\otimes(\min(A)-1)}\otimes \left(\tau_{\max(A)-1,\max(A)}^{A}\dots\tau_{\min(A),\min(A)+1}^{A}\right)
(\Gamma^q)\otimes 1^{\otimes (n-\max(A))},
\end{equation}
where
\begin{equation}
\label{extension process}
\tau_{k-1,k}^{A}=\left\{
\arraycolsep=1.4pt\def\arraystretch{1.5}\begin{array}{lll}
1^{\otimes (k-\min(A)-1)}\otimes\Delta & \textrm{if}\ k-1\in A \ \textrm{and} \ k\in A, \\
(1^{\otimes (k-\min(A))}\otimes\tau)(1^{\otimes (k-\min(A)-1)}\otimes\Delta)\quad & \textrm{if}\ k-1\in A \ \mathrm{and} \ k\notin A,\\
1^{\otimes(k-\min(A)-1)}\otimes\Delta\otimes 1
&\textrm{if}\ k-1\notin A \ \mathrm{and}\ k\notin A,\\
1^{\otimes (k-\min(A)+1)} & \textrm{if}\ k-1\notin A \ \textrm{and} \ k\in A.
\end{array}
\right.
\end{equation}
For sets \(A = [i;j]\) of consecutive numbers, this process reduces to the well-known iteration of coproducts
\[
\Gamma_{[i;j]}^q = 1^{\otimes(i-1)}\otimes \Delta^{(j-i+1)}(\Gamma^q)\otimes 1^{\otimes(n-j)},
\]
with
\begin{equation}
\label{iterating Delta}
\Delta^{(d)} = \left(1^{\otimes(d-2)}\otimes\Delta\right)\Delta^{(d-1)}, \quad \Delta^{(1)} = \mathrm{id}.
\end{equation}
\begin{definition}
	The subalgebra of \(\ospq^{\otimes n}\) spanned by the elements \(\Gamma_A^q\), \(A\subseteq[1;n]\), is said to be the \(q\)-Bannai--Ito algebra of rank \(n-2\). We will denote it by \(\mathcal{A}_n^q\).
\end{definition}
This terminology goes back to \cite{Genest&Vinet&Zhedanov-2016}, where the rank 1 \(q\)-Bannai--Ito algebra was introduced, and to \cite{DeBie&Genest&Vinet-2016}, which established a construction for the \(q = 1\) rank \(n-2\) Bannai--Ito algebra using Dunkl operators.
The algebra relations satisfied by the elements \(\Gamma_A^q\) can be summarized as follows.
\begin{proposition}
	\label{prop q-anticommutation Gamma}
	Let \(A_1\), \(A_2\), \(A_3\), \(A_4\) be subsets of \([1;n]\) such that for any \(i\in\{1,2,3\}\) one has
	\[
	\max(A_i) < \min(A_{i+1}) \ \mathrm{or}\ A_i = \emptyset\ \mathrm{or}\ A_{i+1}=\emptyset. 
	\]
	Let \(A, B \subseteq[1;n]\) be defined through one of the following relations:
	\begin{alignat}{3}\label{most general sets 1}
	& A= A_1 \cup A_2 \cup A_4, \qquad && B = A_2 \cup A_3,& \\
	\label{most general sets 2}
	& A= A_2 \cup A_3, \qquad && B = A_1 \cup A_3 \cup A_4, &\\
	\label{most general sets 3}
	& A= A_1 \cup A_3 \cup A_4, \qquad && B = A_1 \cup A_2 \cup A_4.&
	\end{alignat}
	Let \(C = (A\cup B) \setminus(A\cap B)\).
	Then the elements \(\Gamma_A^q\), \(\Gamma_B^q\) and \(\Gamma_C^q\) satisfy the relations
	\begin{align}
	\label{q-anticommutation Gamma}
	\begin{split}
	\{\Gamma_A^q,\Gamma_B^q\}_q  & = \Gamma_{C}^q + (q^{1/2}+q^{-1/2})\left(\Gamma_{A\cap B}^q\Gamma_{A\cup B}^q + \Gamma_{A\setminus( A\cap B)}^q\Gamma_{B\setminus (A\cap B)}^q\right), \\
	\{\Gamma_B^q,\Gamma_C^q\}_q  & = \Gamma_{A}^q + (q^{1/2}+q^{-1/2})\left(\Gamma_{B\cap C}^q\Gamma_{B\cup C}^q + \Gamma_{B\setminus (B\cap C)}^q\Gamma_{C\setminus (B\cap C)}^q\right), \\
	\{\Gamma_C^q,\Gamma_A^q\}_q  & = \Gamma_{B}^q + (q^{1/2}+q^{-1/2})\left(\Gamma_{C\cap A}^q\Gamma_{C\cup A}^q + \Gamma_{C\setminus (C\cap A)}^q\Gamma_{A\setminus (C\cap A)}^q\right).
	\end{split}
	\end{align}
\end{proposition}
For the proof of these relations we refer to \cite{DeClercq-2019}.

The iteration of coproducts (\ref{iterating Delta}) can equally be applied to the \(\ospq\)-generators, one thus obtains from (\ref{Coproduct})
\begin{align}
\label{iteration A_pm, K, P}
\begin{split}
\left(A_{\pm}\right)_{[1;n]} & = \Delta^{(n)}\left(A_{\pm}\right) = \sum_{\ell = 1}^{n} \left(K^{-1}\right)^{\otimes (\ell-1)}\otimes A_\pm\otimes (KP)^{\otimes (n-\ell)}, \\  \left(K^{\pm 1}\right)_{[1;n]} & = \Delta^{(n)}\left(K^{\pm 1}\right) = \left(K^{\pm 1}\right)^{\otimes n}, \qquad P_{[1;n]} = \Delta^{(n)}(P) = P^{\otimes n}
\end{split}
\end{align}
and hence by (\ref{Gamma^q}) we have
\begin{equation}
\label{Gamma_[1;n]^q}
\Gamma_{[1;n]}^q = \left(-\left(A_+\right)_{[1;n]}\left(A_-\right)_{[1;n]}+\frac{q^{-1/2}}{q-q^{-1}}K_{[1;n]}^2 - \frac{q^{1/2}}{q-q^{-1}}\left(K^{-1}\right)_{[1;n]}^2\right)P_{[1;n]}.
\end{equation}

Some other useful properties of the algebra generators, proven in \cite{DeBie&DeClercq&vandeVijver-2018}, are the following.

\begin{corollary}
	\label{cor Generating set}
	The set of operators \(\Gamma_{[i;j]}^q\), with \(1\leq i\leq j\leq n\), is a generating set for \(\mathcal{A}_n^q\).
\end{corollary}

\begin{proposition}
	\label{prop A contained in B}
	For \(A,B\subseteq [1;n]\) sets of consecutive integers such that \(A\subseteq B\) or \(A\cap B = \emptyset\), one has
	\[
	[\Gamma_A^q,\Gamma_B^q] = 0.
	\]
\end{proposition}

This allows us to state the following definition.
\begin{definition}
	\label{definition Y_i,r}
	For \(i\in[1;n-1], r\in[1;n-i]\), we denote by \(\mathcal{Y}_{i,r}\) the subalgebra of \(\mathcal{A}_n^q\) with generators 
	\begin{equation}
	\label{subalgebra Y_i,r}
	\mathcal{Y}_{i,r} = \langle \Gamma_{[1;2]}^q,\dots,\Gamma_{[1;i]}^q,\Gamma_{[i+1;i+2]}^q,\dots,\Gamma_{[i+1;i+r]}^q,\Gamma_{[1;i+r]}^q,\dots,\Gamma_{[1;n]}^q\rangle.
	\end{equation}
	It is abelian by Proposition \ref{prop A contained in B}.
\end{definition}
\begin{remark}
	\label{Remark Y_i,r}
	In case \(i = 1\), the first sequence \(\Gamma_{[1;2]}^q,\dots,\Gamma_{[1;i]}^q\) of generators is of course empty. Hence our notation \(\mathcal{Y}_{1,r}\) will refer to the algebra
	\[
	\mathcal{Y}_{1,r} = \langle \Gamma_{[2;3]}^q,\dots,\Gamma_{[2;r+1]}^q,\Gamma_{[1;r+1]}^q,\dots,\Gamma_{[1;n]}^q\rangle.
	\]
	For \(r = 1\), the middle sequence \(\Gamma_{[i+1;i+2]}^q,\dots,\Gamma_{[i+1;i+r]}^q\) will be empty, so by \(\mathcal{Y}_{i,1}\) we mean
	\[
	\mathcal{Y}_{i,1} = \langle \Gamma_{[1;2]}^q,\dots,\Gamma_{[1;i]}^q,\Gamma_{[1;i+1]}^q,\dots,\Gamma_{[1;n]}^q\rangle = \langle \Gamma_{[1;k]}^q: k\in[2;n]\rangle,
	\]
	which yields the same algebra for every \(i\), i.e.
	\[
	\mathcal{Y}_{i,1} = \mathcal{Y}_{1,1}, \qquad \forall i \in [2;n-1].
	\]
\end{remark}
\begin{remark}
	It is clear that no generators can be added to the algebras \(\mathcal{Y}_{i,r}\) without losing the property of commutativity. The number of non-central generators of these maximal abelian subalgebras, \(n-2\) in this case, stands as the rank of the algebra \(\mathcal{A}_n^q\).
\end{remark}

The generators of these subalgebras mutually commute and they will also turn out to be simultaneously diagonalizable in Section \ref{Paragraph - Unitary irreducible modules}. Our main purpose in Section \ref{Paragraph - Connection coefficients} will be the computation of the overlap coefficients between the eigenbases of two such subalgebras \(\mathcal{Y}_{i,r}\).

\subsection{The $\mathbb{Z}_2^n$ $q$-Dirac--Dunkl model}
\label{Paragraph q-Dirac-Dunkl model}

An explicit realization of the higher rank \(q\)-Bannai--Ito algebra was proposed in \cite{DeBie&DeClercq&vandeVijver-2018}. We recall here its basic features as a motivation for the forthcoming Definition \ref{definition bases final}. Let \(x_1,\dots,x_n\) be arbitrary real variables and \(\gamma_1,\dots,\gamma_n>\frac12\) be real parameters. Let \(r_i\) be the reflection with respect to the hyperplane \(x_i = 0\), i.e. \(r_if(x_1,\dots,x_n) = f(x_1,\dots,-x_i,\dots,x_n)\), and let \(T_{q,i}\) be the \(q\)-shift operator \(T_{q,i}f(x_1,\dots,x_n) = f(x_1,\dots,qx_i,\dots,x_n)\). Introduce the \(\mathbb{Z}_2^n\) \(q\)-Dunkl operator as
\[
D_i^q = \frac{q^{\gamma_i-\frac12}}{q-q^{-1}}\left(\frac{T_{q,i}-r_i}{x_i} \right)-\frac{q^{-(\gamma_i-\frac12)}}{q-q^{-1}}\left(\frac{T_{q,i}^{-1}-r_i}{x_i}\right).
\]
A realization of \(\ospq\), acting on functions of the variable \(x_i\), can be obtained by taking \cite{Genest&Vinet&Zhedanov-2016}
\begin{equation}
\label{ospq-correspondence}
A_+ \rightarrow x_i, \quad A_- \rightarrow D_i^q, \quad K \rightarrow q^{\gamma_i/2}T_{q,i}^{1/2}, \quad P\rightarrow r_i.
\end{equation}
Explicit expressions were obtained in \cite{DeBie&DeClercq&vandeVijver-2018} for the operators \(\Gamma_{[i+1;i+j]}^q\) inside this realization. The model was completed by the definition of the operators
\[
D_{[n]}^q = \sum_{i=1}^n D_i^qR_{[n],i}^q, \qquad X_{[n]}^q = \sum_{i=1}^n x_iR_{[n],i}^q, \qquad T_{q,[n]} = \prod_{i = 1}^nT_{q,i},
\]
where
\[
R_{[n],i}^q = \left(\prod_{j = 1}^{i-1} q^{-\frac{\gamma_j}{2}}\left(T_{q,j}\right)^{-1/2}\right)\left(\prod_{j = i+1}^{n}q^{\frac{\gamma_j}{2}}\left(T_{q,j}\right)^{1/2}r_j\right),
\]
and the Cauchy--Kowalewska extension
\[
\mathbf{CK}_{x_j}^{\gamma_j} = \sum_{m = 0}^{\infty} \frac{(-1)^{\frac{m(m+1)}{2}}q^{\frac{m}{2}(\gamma_{[1;j]}-1)}}{\sigma_{m}^{(\gamma_j)}\sigma_{m-1}^{(\gamma_j)}\dots\sigma_{1}^{(\gamma_j)}}x_j^{m}\left(D_{[j-1]}^q\right)^{m}\left(T_{q,[j-1]}\right)^{\frac{m}{2}},
\]
where by \(\sigma_m^{(\gamma)}\) we denote the number
\begin{equation}
\label{def sigma_m^gamma}
\sigma_m^{(\gamma)} = \left[m+\gamma-\frac12\right]_q-(-1)^m\left[\gamma-\frac12\right]_q
\end{equation}
and where we write \(\gamma_A = \sum_{i\in A}\gamma_i \) for any \(A\subseteq [1;n]\).
A construction was given for the polynomial null-solutions of \(D_{[n]}^q\) that diagonalize the subalgebra \(\mathcal{Y}_{i,1}\). For any vector \(\mathbf{j}\in \N^{n-1}\) we will denote by \(\vert\mathbf{j}_m\vert\) the sum \(j_1+\dots+j_{m}\) and by \(\vert\bj\vert\) the full sum \(j_1+\dots+j_{n-1}\). Then the functions
\begin{align}
	\label{psi_j def}
	\begin{split}
	& \psi_{\mathbf{j}}(x) = \psi_{(j_1,\dots,j_{n-1})}(x_1,\dots,x_n) \\ =& \  \mathbf{CK}_{x_{n}}^{\gamma_{n}}\left[\left(X_{[n-1]}^q\right)^{j_{n-1}}\mathbf{CK}_{x_{n-1}}^{\gamma_{n-1}}\left[\left(X_{[n-2]}^q\right)^{j_{n-2}}\dots \left[\left(X_{[2]}^q\right)^{j_{2}}\mathbf{CK}_{x_2}^{\gamma_2}(x_1^{j_1})\right]\dots\right]\right]
	\end{split}
\end{align}
are homogeneous polynomials in \(x_1,\dots,x_n\) of degree \(\vjj{n-1}\) and were shown to satisfy the eigenvalue equations
\[
\Gamma_{[1;\ell]}^q\psi_{\mathbf{j}}(x) = (-1)^{\vert \mathbf{j}_{\ell-1}\vert} \left[\vert \mathbf{j}_{\ell-1}\vert + \gamma_{[1;\ell]} - \frac12 \right]_q\psi_{\mathbf{j}}(x),
\]
for all \(\ell\in\{2,\dots,n\}\) and
\[
D_{[n]}^q\psi_{\mathbf{j}}(x) = 0,
\]
and moreover turned out to be orthogonal with respect to the \(q\)-Dunkl-Fischer inner product
\begin{equation}
	\label{q-Dunkl-Fischer inner product}
	\langle \psi(x), \varphi(x)\rangle = \big(\overline{\psi(D_1^q,\dots,D_n^q)}\varphi(x_1,\dots,x_n)\big)\big\vert_{\mathbf{x} = 0}.
\end{equation}
Referring to the notation of the forthcoming Definition \ref{definition bases final}, they will coincide with the abstract vectors \(\vert j^{(1,1)}_{n-1},\dots,j^{(1,1)}_1;0\rangle \), up to normalization. Similarly, the basis vectors \(\vert j^{(1,n-1)}_{n-1},\dots,j^{(1,n-1)}_1; 0\rangle\) diagonalizing the subalgebra
\[
\mathcal{Y}_{1,n-1} = \langle \Gamma_{[2;3]}^q,\Gamma_{[2;4]}^q,\dots,\Gamma_{[2;n]}^q,\Gamma_{[1;n]}^q\rangle
\]
are realized inside the \(q\)-Dirac--Dunkl model by the functions
\begin{align*}
	\varphi_{\mathbf{j}}(x_1,\dots,x_n) = \pi\left(\psi_{\mathbf{j}}(x_1,\dots,x_n)\right),
\end{align*}
where \(\pi\) is the permutation \(1\mapsto2, 2\mapsto 3,\dots, n-1\mapsto n, n\mapsto 1\), applied simultaneously to the index \(i\) of the variables \(x_i\), the \(q\)-Dunkl operators \(D_i^q\), the reflections \(r_i\), the \(q\)-shift operators \(T_{q,i}\) and the parameters \(\gamma_i\). The overlap coefficients
\[
\langle \varphi_{\mathbf{j}}(x), \psi_{\mathbf{s}}(x)\rangle
\]
with respect to the inner product (\ref{q-Dunkl-Fischer inner product}) will hence follow as special cases from the results in Section \ref{Paragraph - Connection coefficients}.

\subsection{Unitary irreducible modules and the spectral problem}
\label{Paragraph - Unitary irreducible modules}
Let us recall some notation from \cite{Genest&Vinet&Zhedanov-2016}. Let \(\gamma>\frac{1}{2}\) be a real number and \(m\in\N\). Denote by \(\vert m\rangle_{\gamma}\) the orthonormal vectors satisfying
\[
\,_{\gamma}\langle m\vert m'\rangle_{\gamma} = \delta_{m,m'},
\]
endowed with the following \(\ospq\)-action
\begin{equation}
\label{ospq-action}
\begin{alignedat}{3}
A_+\vert m\rangle_{\gamma}& = \sqrt{\sigma_{m+1}^{(\gamma)}}\vert m+1\rangle_{\gamma}, \quad & A_-\vert m\rangle_{\gamma} &= \sqrt{\sigma_m^{(\gamma)}}\vert m-1\rangle_{\gamma}, \\
K\vert m\rangle_{\gamma}& = q^{\frac12(m+\gamma)}\vert m\rangle_{\gamma},\quad & P\vert m\rangle_{\gamma}&  = (-1)^m\vert m\rangle_{\gamma},
\end{alignedat}
\end{equation} 
with \(\sigma_m^{(\gamma)}\) as in (\ref{def sigma_m^gamma}).
As follows from (\ref{Gamma^q}), the Casimir element \(\Gamma^q\) acts on these vectors as a multiple of the identity:
\[
\Gamma^q\vert m\rangle_{\gamma} = \left[\gamma-\frac12\right]_q\vert m\rangle_{\gamma}.
\]
Let \(W^{(\gamma)}\) be the infinite-dimensional vector space spanned by the vectors \(\vert m\rangle_{\gamma}\), \(m\in \N\). It follows immediately from (\ref{ospq-action}) that
\begin{equation}
\label{hermitian conjugate generators}
A_\pm^{\dagger} = A_\mp, \quad K^{\dagger}= K, \quad P^{\dagger} = P,
\end{equation}
hence also
\begin{equation}
\label{hermitian conjugate Gamma}
{\Gamma^q}^{\dagger} = \Gamma^q.
\end{equation}
Then by the preceding observations \(W^{(\gamma)}\) is a unitary \(\ospq\)-module, and moreover it is irreducible if one requires \(\vert q\vert \neq 1\), since then one has \(\sigma_m^{(\gamma)}\neq 0\) for all \(m\in\N\setminus\{0\}\). 

\begin{remark}
Note that the explicit \(\ospq\)-realization (\ref{ospq-correspondence}) agrees with (\ref{ospq-action}) under the identification
\[
\vert m\rangle_{\gamma} \to \frac{x^m}{\sqrt{\sigma_1^{(\gamma)}\sigma_2^{(\gamma)}\dots\sigma_m^{(\gamma)}}}.
\]
\end{remark}

Now fix natural numbers \(i, r, n\in \N\setminus\{0\}\) such that \(n \geq i+r\). Like before we will let \(\gamma_1, \dots, \gamma_n>\frac12\) be a set of real numbers. In what follows, we will consider the action of \(\mathcal{A}_n^q\) on the coupled module \(W^{(\gamma_1)}\otimes\dots\otimes W^{(\gamma_n)}\). Note that for any vector \(\vert\mathbf{v}\rangle \) in \(W^{(\gamma_1)}\otimes\dots\otimes W^{(\gamma_n)}\) one has
\begin{equation}
\label{Remark singletons}
\Gamma_{\{i\}}^q\vert\mathbf{v}\rangle = \left[\gamma_i-\frac12\right]_q\vert\mathbf{v}\rangle.
\end{equation}

It will also turn out useful to make the following observation.

\begin{lemma}
	\label{lemma self-adjoint}
	Each of the operators \(\Gamma_{[k;\ell]}^q\), \(1\leq k\leq \ell\leq n\), as well as the operator \(K_{[1;n]}\), are self-adjoint.
\end{lemma}
\begin{proof}
	Using the definition of the coproduct (\ref{Coproduct}) and the hermitian conjugate (\ref{hermitian conjugate generators}), one can easily show by induction that
	\[
	\Delta^{(d)}(A_{\pm})^{\dagger} = \Delta^{(d)}(A_\pm^{\dagger}), \quad \Delta^{(d)}(K)^{\dagger} = \Delta^{(d)}(K^{\dagger}), \quad \Delta^{(d)}(P)^{\dagger} = \Delta^{(d)}(P^{\dagger}),
	\]
	for any \(d\in\N\), referring to the notation (\ref{iterating Delta}). By the multiplicativity of the coproduct and (\ref{Gamma^q}) one hence also finds
	\[
	\Delta^{(d)}(\Gamma^q)^{\dagger} = \Delta^{(d)}(\Gamma^q),
	\]
	which concludes the proof.
\end{proof}

Let us write \((A_{\pm})_{[1;i]}\), \(K_{[1;i]}\) and \(P_{[1;i]}\) for the operators obtained from (\ref{iteration A_pm, K, P}) by replacing \(n\) by \(i\). Note that these have a natural action on the coupled module \(W^{(\gamma_1)}\otimes\dots\otimes W^{(\gamma_i)}\). Similarly, write \((A_{\pm})_{[i+1;i+r]}\), \(K_{[i+1;i+r]}\) and \(P_{[i+1;i+r]}\) for the corresponding operators acting on \(W^{(\gamma_{i+1})}\otimes\dots\otimes W^{(\gamma_{i+r})}\), and \((A_{\pm})_{[1;i+r]}\), \(K_{[1;i+r]}\) and \(P_{[1;i+r]}\) for their counterparts acting on \(W^{(\gamma_{1})}\otimes\dots\otimes W^{(\gamma_{i+r})}\).

Our main purpose in this section will be to derive several explicit bases for the coupled module \(W^{(\gamma_{1})}\otimes\dots\otimes W^{(\gamma_{n})}\). To that end, we state the following definition.

\begin{definition}
	\label{definition psi bases}
	Let \(\mathbf{h}=(h_1,\dots,h_{i-1})\in\N^{i-1}\) and \(p\in \N\). The vectors \(\psi^{(i)}(\bh,p)\in W^{(\gamma_1)}\otimes \dots\otimes W^{(\gamma_i)}\) are recursively defined as 
	\[
	\psi^{(i)}(h_1,\dots,h_{i-1},p) = (A_+)_{[1;i]}^p\left(\sum_{k = 0}^{h_{i-1}}c_{k,\mathbf{h}}^{(i)}\psi^{(i-1)}(h_1,\dots,h_{i-2},k)\otimes\vert h_{i-1}-k\rangle_{\gamma_i}\right)
	\]
	with coefficients
	\[
	c_{k,\mathbf{h}}^{(i)} = (-1)^{k(h_{i-1}+\frac{k-1}{2})}q^{-\frac{k}{2}(\vh{i-1}+\gamma_{[1;i]}-1)}\prod_{\ell = 0}^{k-1}\frac{\sqrt{\sigma_{h_{i-1}-\ell}^{(\gamma_i)}}}{\sigma_{\ell+1}^{(\vh{i-2}+\gamma_{[1;i-1]})}}
	\]
	and initial conditions
	\begin{equation}
	\label{psi i initial conditions}
	\psi^{(1)}(p) = \left(\prod_{\ell = 1}^p\sqrt{\sigma_{\ell}^{(\gamma_1)}} \right)\vert m\rangle_{\gamma_1}.
	\end{equation}
\end{definition}

The vectors \(\psi^{(i)}(\bh,p)\) are easily shown to be eigenvectors for certain operators in \(\ospq^{\otimes i}\).

\begin{lemma}
	\label{lemma psi}
	The vectors \(\psi^{(i)}(\bh,p)\) satisfy the equations
	\begin{align}
	\label{eigenvalue Gamma psi}
	\Gamma_{[1;m]}^q\psi^{(i)}(h_1,\dots,h_{i-1},p) & = (-1)^{\vh{m-1}}\left[\vh{m-1}+\gamma_{[1;m]}-\frac12\right]_q \psi^{(i)}(h_1,\dots,h_{i-1},p),
	\end{align}
	for any \(m\in \{2,\dots,i\}\),
	\begin{align}
	\label{eigenvalue K psi}
	K_{[1;i]}\psi^{(i)}(h_1,\dots,h_{i-1},p) & = q^{\frac12(\vert\bh\vert+p+\gamma_{[1;i]})}\psi^{(i)}(h_1,\dots,h_{i-1},p), \\
	\label{eigenvalue P psi}
	P_{[1;i]}\psi^{(i)}(h_1,\dots,h_{i-1},p) & = (-1)^{\vert\bh\vert+p}\psi^{(i)}(h_1,\dots,h_{i-1},p),
	\end{align}
	and
	\begin{equation}
	\label{eigenvalue A_- psi}
	\left(A_-\right)_{[1;i]}\psi^{(i)}(h_1,\dots,h_{i-1},p) = \sigma_{p}^{(\vert\bh\vert+\gamma_{[1;i]})} \psi^{(i)}(h_1,\dots,h_{i-1},p-1),
	\end{equation}
	where we set \(\psi^{(i)}(h_1,\dots,h_{i-1},-1) = 0\).
\end{lemma}
\begin{proof}
	We will prove the claims by induction on \(i\), where the case \(i = 1\) follows immediately from (\ref{ospq-action}) and (\ref{psi i initial conditions}). 
	
	It is immediate from (\ref{def - ospq with K}) that
	\begin{align}
	\label{relations with powers}
	\begin{split}
	A_-A_+^{k} & = A_+^{k-1}\left(\frac{q^{-1/2}(q^{k}-(-1)^{k})}{q-q^{-1}}K^2 - \frac{q^{1/2}(q^{-k}-(-1)^{k})}{q-q^{-1}}K^{-2}\right) + (-1)^{k} A_+^{k} A_-, \\
	KA_+^{k} & = q^{\frac{k}{2}} A_+^{k}K, \\
	PA_+^{k} & = (-1)^{k} A_+^{k}P,
	\end{split}
	\end{align}
	and of course the same identities hold after applying \(\Delta^{(i)}\) to both sides. The induction hypothesis hence leads straightforwardly to (\ref{eigenvalue K psi}) and (\ref{eigenvalue P psi}). Moreover, \(\Gamma_{[1;m]}^q\) commutes with \((A_+)_{[1;i]}\) for any \(m\in \{2,\dots,i-1\}\), as one sees by applying \(\Delta^{(m)}\otimes 1^{\otimes (i-m)}\) to the identity
	\[
	\left[\Gamma^q\otimes 1^{\otimes (i-m)},\left(A_+\right)_{[1;i-m+1]}\right] = 0.
	\]
	Hence the relations (\ref{eigenvalue Gamma psi}) for \(m<i\) follow immediately from the induction hypothesis. 
	
	Furthermore, (\ref{relations with powers}) asserts that
	\begin{align*}
	(A_-)_{[1;i]}\psi^{(i)}(\bh,p) =\ & (A_-)_{[1;i]}(A_+)_{[1;i]}^p\psi^{(i)}(\bh,0)\\ =\  & \sigma_p^{(\vert\bh\vert + \gamma_{[1;i]})}\psi^{(i)}(\bh,p-1) + (-1)^p(A_+)_{[1;i]}^{p}(A_-)_{[1;i]}\psi^{(i)}(\bh,0),
	\end{align*}
	where we have used (\ref{eigenvalue K psi}). It follows from (\ref{iteration A_pm, K, P}) that
	\[
	(A_-)_{[1;i]} = (A_-)_{[1;i-1]}\otimes KP + K_{[1;i-1]}^{-1}\otimes A_-,
	\]
	and so
	\begin{align*}
	(A_-)_{[1;i]}\psi^{(i)}(\bh,0)=\ & \sum_{k = 0}^{h_{i-1}}c_{k,\bh}^{(i)} (A_-)_{[1;i-1]}\psi^{(i-1)}(\bh',k)\otimes (KP)\vert h_{i-1}-k\rangle_{\gamma_i} \\ & + \sum_{k = 0}^{h_{i-1}}c_{k,\bh}^{(i)} K_{[1;i-1]}^{-1}\psi^{(i-1)}(\bh',k)\otimes (A_-)\vert h_{i-1}-k\rangle_{\gamma_i}.
	\end{align*}
	where \(\bh' = (h_1,\dots,h_{i-2})\). By (\ref{ospq-action}) and the induction hypothesis, this yields
	\[
	\sum_{k = 0}^{h_{i-1}-1}\chi_{k,\bh}^{(i)}\psi^{(i-1)}(\bh',k)\otimes\vert h_{i-1}-k-1\rangle_{\gamma_i},
	\]
	where
	\begin{align*}
	\chi_{k,\bh}^{(i)} =\ & (-1)^{h_{i-1}-k-1}q^{\frac12(h_{i-1}-k-1+\gamma_i)}\sigma_{k+1}^{(\vert\bh_{i-2}\vert+\gamma_{[1;i-1]})}c_{k+1,\bh}^{(i)}\\& + q^{-\frac12(\vert\bh_{i-2}\vert+k+\gamma_{[1;i-1]})}\sqrt{\sigma_{h_{i-1}-k}^{(\gamma_i)}}c_{k,\bh}^{(i)},
	\end{align*}
	which vanishes by definition of \(c_{k,\bh}^{(i)}\). This proves (\ref{eigenvalue A_- psi}).
	
	Finally, one easily verifies the equation (\ref{eigenvalue Gamma psi}) for \(m = i\) by (\ref{Gamma_[1;n]^q}) with \(i\) instead of \(n\), and the equations (\ref{eigenvalue K psi}), (\ref{eigenvalue P psi}) and (\ref{eigenvalue A_- psi}).
\end{proof}

Upon replacing \(i\) by \(r\), \((\gamma_1,\dots,\gamma_i)\) by \((\gamma_{i+1},\dots,\gamma_{i+r})\) and \((A_-)_{[1;i]}\) by \((A_-)_{[i+1;i+r]}\) in Definition \ref{definition psi bases}, one obtains the following analogous construction.

\begin{definition}
	\label{definition phi bases}
	Let \(\bk=(k_1,\dots,k_{r-1})\in \N^{r-1}\) and \(p'\in\N\). The vectors \(\varphi^{(r)}(\bk,p')\in W^{(\gamma_{i+1})}\otimes\dots\otimes W^{(\gamma_{i+r})}\) are recursively defined as
	\[
	\varphi^{(r)}(k_1,\dots,k_{r-1},p') = (A_+)_{[i+1;i+r]}^{p'}\left(\sum_{t = 0}^{k_{r-1}}d_{t,\mathbf{k}}^{(r)}\varphi^{(r-1)}(k_1,\dots,k_{r-2},t)\otimes\vert k_{r-1}-t\rangle_{\gamma_{i+r}}\right)
	\]
	with coefficients
	\[
	d_{t,\mathbf{k}}^{(r)} = (-1)^{t(k_{r-1}+\frac{t-1}{2})}q^{-\frac{t}{2}(\vert\mathbf{k}_{r-1}\vert+\gamma_{[i+1;i+r]}-1)}\prod_{\ell = 0}^{t-1}\frac{\sqrt{\sigma_{k_{r-1}-\ell}^{(\gamma_{i+r})}}}{\sigma_{\ell+1}^{(\vert\bk_{r-2}\vert+\gamma_{[i+1;i+r-1]})}}
	\]
	and initial conditions
	\[
	\varphi^{(1)}(p') = \left(\prod_{\ell = 1}^{p'}\sqrt{\sigma_{\ell}^{(\gamma_{i+1})}} \right)\vert m\rangle_{\gamma_{i+1}}.
	\]
\end{definition}

The vectors \(\varphi^{(r)}(\bk,p')\) satisfy suitable analogs of the equations (\ref{eigenvalue Gamma psi})--(\ref{eigenvalue A_- psi}). The proof goes along the same lines as Lemma \ref{lemma psi}.

\begin{lemma}
	\label{lemma phi}
	The vectors \(\varphi^{(r)}(\bk,p')\) satisfy the equations
	\begin{align}
	\label{eigenvalue Gamma phi}
	\begin{split}
	&\Gamma_{[i+1;i+m]}^q\varphi^{(r)}(k_1,\dots,k_{r-1},p') \\ =\ & (-1)^{\vk{m-1}}\left[\vk{m-1}+\gamma_{[i+1;i+m]}-\frac12\right]_q \varphi^{(r)}(k_1,\dots,k_{r-1},p'),
	\end{split}
	\end{align}
	for any \(m\in \{2,\dots,r\}\),
	\begin{align}
	\label{eigenvalue K phi}
	K_{[i+1;i+r]}\varphi^{(r)}(k_1,\dots,k_{r-1},p') & = q^{\frac12(\vert\bk\vert+p'+\gamma_{[i+1;i+r]})}\varphi^{(r)}(k_1,\dots,k_{r-1},p'), \\
	\label{eigenvalue P phi}
	P_{[i+1;i+r]}\varphi^{(r)}(k_1,\dots,k_{r-1},p') & = (-1)^{\vert\bk\vert+p'}\varphi^{(r)}(k_1,\dots,k_{r-1},p'),
	\end{align}
	and
	\begin{equation}
	\label{eigenvalue A_- phi}
	\left(A_-\right)_{[i+1;i+r]}\varphi^{(r)}(k_1,\dots,k_{r-1},p') = \sigma_{p'}^{(\vert\bk\vert + \gamma_{[i+1;i+r]})} \varphi^{(r)}(k_1,\dots,k_{r-1},p'-1),
	\end{equation}
	where we set \(\varphi^{(r)}(k_1,\dots,k_{r-1},-1)=0\).
\end{lemma}

The vectors \(\psi^{(i)}(\bh,p)\) and \(\varphi^{(r)}(\bk,p')\) can be coupled in the following fashion.

\begin{definition}
	Let \(\mathbf{h}=(h_1,\dots,h_{i-1})\in\N^{i-1}\), \(\bk=(k_1,\dots,k_{r-1})\in \N^{r-1}\) and \(s,t\in\N\). The vectors \(\Psi^{(i,r)}(\bh,\bk,s,t) = \Psi^{(i,r)}(h_1,\dots,h_{i-1},k_1,\dots,k_{r-1},s,t)\) are defined by
	\[
	\Psi^{(i,r)}(\bh,\bk,s,t) = (A_+)_{[1;i+r]}^t\sum_{p = 0}^s \epsilon_p \psi^{(i)}(h_1,\dots,h_{i-1},p)\otimes\varphi^{(r)}(k_1,\dots,k_{r-1},s-p),
	\]
	where
	\[
	\epsilon_p = (-1)^p\frac{\beta_{s-p+1}\beta_{s-p+2}\dots \beta_s}{\alpha_p\alpha_{p-1}\dots\alpha_1}
	\]
	with
	\begin{align*}
	\alpha_p =\ & (-1)^{\vert\bk\vert + s-p} q^{\frac12(\gamma_{[i+1;i+r]}+s-p+\vert\bk\vert)} \sigma_{p}^{(\vert\bh\vert + \gamma_{[1;i]})}, \\
	\beta_{p'}=\ & q^{-\frac12(\gamma_{[1;i]}+s-p'+\vert\bh\vert)} \sigma_{p'}^{(\vert\bk\vert + \gamma_{[i+1;i+r]})}.
	\end{align*}
\end{definition}

Relying on Lemmas \ref{lemma psi} and \ref{lemma phi}, one may identify some interesting equations satisfied by the vectors \(\Psi^{(i,r)}(\bh,\bk,s,t)\).

\begin{lemma}
	\label{lemma Psi}
	The vectors \(\Psi^{(i,r)}(\bh,\bk,s,t)\) satisfy the equations
	\begin{align}
	\label{eigenvalue Gamma Psi 1}
	\Gamma_{[1;m]}^q\Psi^{(i,r)}(\bh,\bk,s,t) & = (-1)^{\vh{m-1}}\left[\vh{m-1}+\gamma_{[1;m]}-\frac12\right]_q \Psi^{(i,r)}(\bh,\bk,s,t),
	\end{align}
	for any \(m\in \{2,\dots,i\}\),
	\begin{align}
	\label{eigenvalue Gamma Psi 2}
	\Gamma_{[i+1;i+m]}^q\Psi^{(i,r)}(\bh,\bk,s,t) & = (-1)^{\vk{m-1}}\left[\vk{m-1}+\gamma_{[i+1;i+m]}-\frac12\right]_q \Psi^{(i,r)}(\bh,\bk,s,t),
	\end{align}
	for any \(m\in \{2,\dots,r\}\), 
	\begin{align}
	\label{eigenvalue Gamma Psi 3}
	\Gamma_{[1;i+r]}^q\Psi^{(i,r)}(\bh,\bk,s,t) & = (-1)^{\vert\bh\vert+\vert\bk\vert+s}\left[\vert\bh\vert+\vert\bk\vert+s+\gamma_{[1;i+r]}-\frac12\right]_q \Psi^{(i,r)}(\bh,\bk,s,t),
	\end{align}
	and
	\begin{align}
	\label{eigenvalue K Psi}
	K_{[1;i+r]}\Psi^{(i,r)}(\bh,\bk,s,t) & = q^{\frac12(\vert\bh\vert+\vert\bk\vert+s+t+\gamma_{[1;i+r]})}\Psi^{(i,r)}(\bh,\bk,s,t), \\
	\label{eigenvalue P Psi}
	P_{[1;i+r]}\Psi^{(i,r)}(\bh,\bk,s,t) & = (-1)^{\vert\bh\vert+\vert\bk\vert+s+t}\Psi^{(i,r)}(\bh,\bk,s,t),
	\end{align}
	and moreover 
	\begin{equation}
	\label{eigenvalue A_- Psi}
	\left(A_-\right)_{[1;i+r]}\Psi^{(i,r)}(\bh,\bk,s,t) = \sigma_t^{(\vert\bh\vert + \vert\bk\vert + s + \gamma_{[1;i+r]})} \Psi^{(i,r)}(\bh,\bk,s,t-1),
	\end{equation}
	where we set \(\Psi^{(i,r)}(\bh,\bk,s,-1)=0\). 
\end{lemma}
\begin{proof}
	The equations (\ref{eigenvalue Gamma Psi 1}) and (\ref{eigenvalue Gamma Psi 2}) follow immediately from Lemmas \ref{lemma psi} and \ref{lemma phi} respectively. For the equations (\ref{eigenvalue K Psi}) and (\ref{eigenvalue P Psi}), observe that
	\begin{align*}
	K_{[1;i+r]}(A_+)_{[1;i+r]}^t =\ & q^{\frac{t}{2}}(A_+)_{[1;i+r]}^tK_{[1;i+r]}, \\ P_{[1;i+r]}(A_+)_{[1;i+r]}^t =\ & (-1)^{t}(A_+)_{[1;i+r]}^tP_{[1;i+r]}
	\end{align*}
	and
	\[
	K_{[1;i+r]} = K_{[1;i]}\otimes K_{[i+1;i+r]}, \qquad P_{[1;i+r]} = P_{[1;i]}\otimes P_{[i+1;i+r]}.
	\]
	Hence (\ref{eigenvalue K Psi}) and (\ref{eigenvalue P Psi}) follow from (\ref{eigenvalue K psi}), (\ref{eigenvalue P psi}), (\ref{eigenvalue K phi}) and (\ref{eigenvalue P phi}). 
	
	It follows from (\ref{relations with powers}) and (\ref{eigenvalue K Psi}) that
	\begin{align*}
	&\left(A_-\right)_{[1;i+r]}\Psi^{(i,r)}(\bh,\bk,s,t) = \left(A_-\right)_{[1;i+r]}\left(A_+\right)_{[1;i+r]}^t\Psi^{(i,r)}(\bh,\bk,s,0) \\
	=\ & \sigma_t^{(\vert\bh\vert+\vert\bk\vert+s+\gamma_{[1;i+r]})} \Psi^{(i,r)}(\bh,\bk,s,t-1) + (-1)^t (A_+)_{[1;i+r]}^t(A_-)_{[1;i+r]}\Psi^{(i,r)}(\bh,\bk,s,0)
	\end{align*}
	where we have used (\ref{eigenvalue K Psi}) with \(t = 0\). Observe from (\ref{iteration A_pm, K, P}) that
	\[
	(A_-)_{[1;i+r]} = (A_-)_{[1;i]}\otimes K_{[i+1;i+r]}P_{[i+1;i+r]} + K^{-1}_{[1;i]}\otimes (A_-)_{[i+1;i+r]}
	\]
	and so it follows from Lemmas \ref{lemma psi} and \ref{lemma phi} that
	\begin{align*}
	&(A_-)_{[1;i+r]}\Psi^{(i,r)}(\bh,\bk,s,0) \\ 
	=\ & \sum_{p = 0}^{s-1}\left(\epsilon_{p+1}\alpha_{p+1}+\epsilon_p\beta_{s-p}\right)\psi^{(i)}(h_1,\dots,h_{i-1},p)\otimes\varphi^{(r)}(k_1,\dots,k_{r-1},s-p-1).
	\end{align*}
	One now immediately checks that \(\epsilon_{p+1}\alpha_{p+1}+\epsilon_p\beta_{s-p}=0\) for all \(p\) and hence the equation (\ref{eigenvalue A_- Psi}) is satisfied.
	
	The equation (\ref{eigenvalue Gamma Psi 3}) follows again from (\ref{Gamma_[1;n]^q}), (\ref{eigenvalue K Psi}), (\ref{eigenvalue P Psi}) and (\ref{eigenvalue A_- Psi}).
\end{proof}

The recursive extension process outlined in Definition \ref{definition psi bases} can now be used to expand the vectors \(\Psi^{(i,r)}\) even further.

\begin{definition}
	\label{definition bases Theta}
	Let \(\mathbf{j}\in\N^{n-1}\) and \(N\in\N\). The vectors \[\Theta^{(i,r)}(\bj,N) = \Theta^{(i,r,n-i-r+1)}(j_1,\dots,j_{n-1},N)\in W^{(\gamma_1)}\otimes\dots\otimes W^{(\gamma_n)}\] are recursively defined by
	\begin{align*}
	&\Theta^{(i,r,n-i-r+1)}(j_1,\dots,j_{n-1},N) \\=\ &  \left(A_+\right)_{[1;n]}^{N}\left(\sum_{k = 0}^{j_{n-1}} e_{k,\mathbf{j}}^{(n)}\Theta^{(i,r,n-i-r)}(j_1,\dots,j_{n-2},k)\otimes\vert j_{n-1}-k\rangle_{\gamma_n}\right),
	\end{align*}
	with coefficients
	\[
	e_{k,\mathbf{j}}^{(n)} = (-1)^{k(j_{n-1}+\frac{k-1}{2})}q^{-\frac{k}{2}(\vert\bj_{n-1}\vert+\gamma_{[1;n]}-1)}\prod_{\ell = 0}^{k-1}\frac{\sqrt{\sigma_{j_{n-1}-\ell}^{(\gamma_n)}}}{\sigma_{\ell+1}^{(\vert\bj_{n-2}\vert+\gamma_{[1;n-1]})}}
	\]
	and initial conditions
	\[
	\Theta^{(i,r,1)}(j_1,\dots,j_{i+r-1},p) = \Psi^{(i,r)}(j_1,\dots,j_{i+r-1},p).
	\]
\end{definition}

The vectors \(\Theta^{(i,r)}(\bj,N)\) can also be characterized through certain eigenvalue equations.

\begin{lemma}
	\label{lemma Theta}
	The vectors \(\Theta^{(i,r)}(j_1,\dots,j_{n-1},N)\) satisfy the equations
	\begin{equation}
	\label{eigenvalue equation Theta 1}
	\Gamma_{[1;m]}^q\Theta^{(i,r)}(j_1,\dots,j_{n-1},N) = (-1)^{\vert\mathbf{j}_{m-1}\vert}\left[\vert\mathbf{j}_{m-1}\vert+\gamma_{[1;m]}-\frac12\right]_q\Theta^{(i,r)}(j_1,\dots,j_{n-1},N)
	\end{equation}
	for all \(m\in \{2,\dots,i\}\cup\{i+r,\dots,n\}\),
	\begin{align}
	\label{eigenvalue equation Theta 2}
	\begin{split}
	& \Gamma_{[i+1;i+m]}^q\Theta^{(i,r)}(j_1,\dots,j_{n-1},N) \\ = &\  (-1)^{\vert\mathbf{j}_{i+m-2}\vert-\vert\mathbf{j}_{i-1}\vert}\left[\vert\mathbf{j}_{i+m-2}\vert-\vert\mathbf{j}_{i-1}\vert+\gamma_{[i+1;i+m]}-\frac12\right]_q\Theta^{(i,r)}(j_1,\dots,j_{n-1},N),
	\end{split}
	\end{align}
	for all \(m\in \{2,\dots,r\}\),
	\begin{align}
	\label{eigenvalue equation Theta K}
	K_{[1;n]}\Theta^{(i,r)}(j_1,\dots,j_{n-1},N) & = q^{\frac12(\vert\bj\vert+N+\gamma_{[1;n]})}\Theta^{(i,r)}(j_1,\dots,j_{n-1},N), \\
	\label{eigenvalue P Theta}
	P_{[1;n]}\Theta^{(i,r)}(j_1,\dots,j_{n-1},N) & = (-1)^{\vert\bj\vert+N}\Theta^{(i,r)}(j_1,\dots,j_{n-1},N),
	\end{align}
	and
	\begin{equation}
	\label{eigenvalue A_- Theta}
	\left(A_-\right)_{[1;n]}\Theta^{(i,r)}(j_1,\dots,j_{n-1},N) = \sigma_{N}^{(\vert\bj\vert + \gamma_{[1;n]})} \Theta^{(i,r)}(j_1,\dots,j_{n-1},N-1),
	\end{equation}
	where we set \(\Theta^{(i,r)}(j_1,\dots,j_{n-1},-1)=0\).
\end{lemma}
\begin{proof}
	Assuming \(i\) and \(r\) are fixed, we in fact have to prove the claim for every \(n\in \N\) with \(n\geq i+r\). We will do this by induction on \(n\).
	
	For \(n = i+r\) we have
	\[
	\Theta^{(i,r)}(j_1,\dots,j_{n-1},N) = \Psi^{(i,r)}(j_1,\dots,j_{n-1},N)
	\]
	and so the claim follows immediately from Lemma \ref{lemma Psi}. 
	
	Suppose hence the statement is true for \(n-1\geq i+r\). Since
	\[
	K_{[1;n]} = K_{[1;n-1]}\otimes K, \qquad P_{[1;n]} = P_{[1;n-1]}\otimes P,
	\]
	it follows immediately from the induction hypothesis that the equations (\ref{eigenvalue equation Theta 2}), (\ref{eigenvalue equation Theta K}) and (\ref{eigenvalue P Theta}) are satisfied, as well as the equations (\ref{eigenvalue equation Theta 1}) for \(m\in \{2,\dots,i\}\cup\{i+r,\dots,n-1\}\). As before, by (\ref{relations with powers}) we have
	\begin{align*}
	&\left(A_-\right)_{[1;n]}\Theta^{(i,r)}(j_1,\dots,j_{n-1},N) \\=\ & \sigma_{N}^{(\vert\bj\vert+\gamma_{[1;n]})} \Theta^{(i,r)}(j_1,\dots,j_{n-1},N-1) + (-1)^{N} (A_+)_{[1;n]}^{N} (A_-)_{[1;n]}\Theta^{(i,r)}(j_1,\dots,j_{n-1},0).
	\end{align*}
	Now recall that
	\[
	(A_-)_{[1;n]} = (A_-)_{[1;n-1]}\otimes KP + K_{[1;n-1]}^{-1}\otimes A_-
	\]
	and so
	\begin{align*}
	(A_-)_{[1;n]}\Theta^{(i,r)}(j_1,\dots,j_{n-1},0) = & \sum_{k = 0}^{j_{n-1}}\chi_{\bj,k}^{(n)} \Theta^{(i,r)}(j_1,\dots,j_{n-2},k)\otimes \vert j_{n-1}-k-1\rangle_{\gamma_n}, 
	\end{align*}
	with
	\begin{align*}
	\chi_{\bj,k}^{(n)} =\ & (-1)^{j_{n-1}-k-1}q^{\frac12(j_{n-1}-k-1+\gamma_n)}\sigma_{k+1}^{(\vert\bj_{n-2}\vert+\gamma_{[1;n-1]})}e_{k+1,\bj}^{(n)}\\& + q^{-\frac12(\vert\bj_{n-2}\vert+k+\gamma_{[1;n-1]})}\sqrt{\sigma_{j_{n-1}-k}^{(\gamma_n)}}e_{k,\bj}^{(n)},
	\end{align*}
	which vanishes by definition of \(e_{k,\bj}^{(n)}\).
	Hence the equation (\ref{eigenvalue A_- Theta}) holds.
	
	Finally, the equation (\ref{eigenvalue equation Theta 1}) with \(m = n\) follows from (\ref{Gamma_[1;n]^q}), (\ref{eigenvalue equation Theta K}), (\ref{eigenvalue P Theta}) and (\ref{eigenvalue A_- Theta}).
\end{proof}

Finally, we will introduce a new notation for the normalized counterparts of the vectors \(\Theta^{(i,r)}\).

\begin{definition}
	\label{definition bases final}
	Let \(\mathbf{j}^{(i,r)}\in\N^{n-1}\) and \(N\in\N\). We denote by \[\vert\mathbf{j}^{(i,r)};N\rangle = \vert j^{(i,r)}_{n-1},j^{(i,r)}_{n-2},\dots,j^{(i,r)}_{1};N\rangle \in W^{(\gamma_1)}\otimes\dots\otimes W^{(\gamma_n)}\]
	the vectors defined by
	\[
	\vert j^{(i,r)}_{n-1},j^{(i,r)}_{n-2},\dots,j^{(i,r)}_{1};N\rangle = \eta_{\mathbf{j}^{(i,r)},N}\Theta^{(i,r,n-i-r+1)}(j_1^{(i,r)},\dots,j_{n-1}^{(i,r)},N),
	\]
	where \(\eta_{\mathbf{j}^{(i,r)},N}\) is a normalization constant, chosen such that
	\begin{equation}
	\label{normalization bases}
	\langle j^{(i,r)}_{n-1},\dots,j^{(i,r)}_1;N \vert {j}_{n-1}^{(i,r)},\dots,{j}_1^{(i,r)};N\rangle = 1.
	\end{equation}
\end{definition}

It follows immediately from Lemma \ref{lemma Theta} that these vectors satisfy the equations
\begin{equation}
\label{eigenvalue equation 1}
\Gamma_{[1;m]}^q\vert j^{(i,r)}_{n-1},\dots,j^{(i,r)}_{1};N\rangle = (-1)^{\vert\mathbf{j}^{(i,r)}_{m-1}\vert}\left[\vert\mathbf{j}^{(i,r)}_{m-1}\vert+\gamma_{[1;m]}-\frac12\right]_q\vert j^{(i,r)}_{n-1},\dots,j^{(i,r)}_{1};N\rangle
\end{equation}
for any \(m\in \{2,\dots,i\}\cup\{i+r,\dots,n\}\),
\begin{align}
\label{eigenvalue equation 2}
\begin{split}
& \Gamma_{[i+1;i+m]}^q\vert j^{(i,r)}_{n-1},\dots,j^{(i,r)}_{1};N\rangle \\ = &\  (-1)^{\vert\mathbf{j}^{(i,r)}_{i+m-2}\vert-\vert\mathbf{j}^{(i,r)}_{i-1}\vert}\left[\vert\mathbf{j}^{(i,r)}_{i+m-2}\vert-\vert\mathbf{j}^{(i,r)}_{i-1}\vert+\gamma_{[i+1;i+m]}-\frac12\right]_q\vert j^{(i,r)}_{n-1},\dots,j^{(i,r)}_{1};N\rangle,
\end{split}
\end{align}
for all \(m\in \{2,\dots,r\}\) and
\begin{equation}
\label{eigenvalue equation K}
K_{[1;n]}\vert j^{(i,r)}_{n-1},\dots,j^{(i,r)}_{1};N\rangle = q^{\frac12(\vj{i,r}{n-1}+N+\gamma_{[1;n]})}\vert j^{(i,r)}_{n-1},\dots,j^{(i,r)}_{1};N\rangle.
\end{equation}
Moreover, if \((\bj^{(i,r)},N)\neq ({\bj'}^{(i,r)},N')\in\N^n\), then \(\vert \bj^{(i,r)};N\rangle\) and \(\vert{\bj'}^{(i,r)};N'\rangle\) will correspond to different eigenvalues for at least one of the operators
\[
\Gamma_{[1;m]}^q \ \mathrm{with}\ m\in \{2,\dots,i\}\cup\{i+r,\dots,n\}, \quad \Gamma_{[i+1;i+m]}^q \ \mathrm{with}\ m\in \{2,\dots,r\},  \quad K_{[1;n]},
\]
and hence by Lemma \ref{lemma self-adjoint} they will be orthogonal, and by (\ref{normalization bases}) even orthonormal.

\begin{remark}
	\label{Remark equality i and 1}
	It follows from Definitions \ref{definition psi bases}, \ref{definition phi bases} and \ref{definition bases Theta} that for any \(i\in \{2,\dots,n-1\}\) one has
	\begin{equation}
	\label{(i,1) = (1,1)}
	\vert\mathbf{j}^{(i,1)};N\rangle =  \vert\mathbf{j}^{(1,1)};N\rangle = \psi^{(n)}(j_1^{(1,1)},\dots,j_{n-1}^{(1,1)};N),
	\end{equation}
	in agreement with Remark \ref{Remark Y_i,r}.
\end{remark}

\subsection{Decompositions and orthonormal bases for the coupled module}
\label{Paragraph Decompositions}

Let us introduce the notation
\begin{equation}
\label{U_N,i def}
U_{N,m} = \langle \vert n_1\rangle_{\gamma_1}\otimes \vert n_2\rangle_{\gamma_2}\otimes\dots\otimes\vert n_m\rangle_{\gamma_m}: n_1+n_2+\dots+n_m = N\rangle,
\end{equation}
i.e. \(U_{N,m}\) is the eigenspace of \(K^{\otimes m}\) with eigenvalue \(q^{\frac12(N+\gamma_{[1;m]})}\).

We will also need the following subspaces of \(W^{(\gamma_1)}\otimes\dots\otimes W^{(\gamma_n)}\).

\begin{definition}
	\label{definition V_p,N}
	Let \(\mathbf{p}= (p_1,\dots,p_{n-i-r})\in\N^{n-i-r}\) and \(N\in\N\).
	We denote by \(V_{N,\mathbf{p}}^{(i,r)} = V_{N,p_1,\dots,p_{n-i-r}}^{(i,r)}\) the joint eigenspace of the operators
	\[
	\Gamma_{[1;i+r+m]}^q \quad \mathrm{with\ eigenvalue} \quad (-1)^{\vert\mathbf{p}_{m+1}\vert}\left[\vert\mathbf{p}_{m+1}\vert+\gamma_{[1;i+r+m]}-\frac12\right]_q,
	\]
	with \(m\in \{0,\dots,n-i-r-1\}\) and
	\[
	K_{[1;n]} \quad \mathrm{with\ eigenvalue} \quad q^{\frac12(N+\vert\mathbf{p}_{n-i-r}\vert+\gamma_{[1;n]})}.
	\]
\end{definition}
Note that for \(n = i+r\) only the last eigenvalue equation remains, hence in that case \(V_{N,\mathbf{p}}^{(i,r)}\) reduces to \(U_{N,n}\). A basis for this space is obtained in the following lemma. 

\begin{lemma}
	\label{lemma U_N,2}
	A basis for the space \(U_{N,n}\) is given by
	\begin{equation}
	\label{U_N,n basis}
	\left\{ \Theta^{(i,r)}\left(j_1,\dots,j_{n-1},N-\vert\bj_{n-1}\vert\right): \bj\in\N^{n-1}\ \mathrm{with}\ \vert\bj_{n-1}\vert\leq N \right\}.
	\end{equation}
\end{lemma}
\begin{proof}
	Each of the vectors \(\Theta^{(i,r)}\left(\bj,N-\vert\bj_{n-1}\vert\right)\), with \(\bj\in\N^{n-1}, \vert\bj_{n-1}\vert\leq N\), lies in \(U_{N,n}\) by Lemma \ref{lemma Theta}, and moreover they are linearly independent, since they correspond to different eigenvalues for the generators of \(\mathcal{Y}_{i,r}\). The claim now follows from the fact that \(\dim\left(U_{N,n}\right) = \binom{N+n-1}{N}\), which is precisely the cardinality of (\ref{U_N,n basis}). 
\end{proof}

This simple observation will lead us to a decomposition for the coupled \(\ospq\)-module. The proof will be omitted, since it essentially follows by induction from the results in \cite[Theorems 2.1 and 5.1]{Shibukawa-1992}.

\begin{proposition}
	The coupled module \(W^{(\gamma_1)}\otimes\dots\otimes W^{(\gamma_n)}\) can be decomposed in irreducible components as
	\begin{equation}
	W^{(\gamma_1)}\otimes\dots\otimes W^{(\gamma_n)} \cong \bigoplus_{N = 0}^{\infty} m_N W^{(\gamma_{[1;n]}+N)},
	\end{equation}
	where the multiplicity is given by
	\[
	m_N = \binom{N+n-2}{N}.
	\]
\end{proposition}

For any \(\mathbf{p}\in\N^{n-i-r}\) and \(N\in\N\) we will denote by \(\mathcal{B}_{\mathbf{p},N}^{(i,r)}\) the set which is given by
\begin{align*}
&\left\{
\Theta^{(i,r)}\left(j_1,\dots,j_{i+r-2},p_1-\vert\bj_{i+r-2}\vert,p_2,\dots,p_{n-i-r},\ell,N-\ell \right):\right. \\
&\left.\bj\in\N^{i+r-2}, \ell\in\N\ \mathrm{with}\ \vert\bj_{i+r-2}\vert \leq p_1, \ell\leq N
\right\}
\end{align*}
if \(n>i+r\), and by
\[
\left\{
\Theta^{(i,r)}\left(j_1,\dots,j_{n-1},N-\vert\bj_{n-1}\vert \right):
\bj\in\N^{n-1}\ \mathrm{with}\ \vert\bj_{n-1}\vert \leq N
\right\}
\]
if \(n = i+r\). Starting from the spaces \(V_{N,\mathbf{p}}^{(i,r)}\) we may also obtain a different decomposition and even an explicit basis for the coupled module. This is the subject of the following proposition.

\begin{proposition}
	\label{prop decomposition}
	An orthonormal basis for the coupled module \(W^{(\gamma_1)}\otimes\dots\otimes W^{(\gamma_n)}\) is given by 
	\begin{equation}
	\label{basis for coupled module}
	\{\vert j_{n-1}^{(i,r)}, \dots, j_1^{(i,r)}; N\rangle: j_1^{(i,r)},\dots,j_{n-1}^{(i,r)}\in\N, N\in\N\}.
	\end{equation}
	More precisely, one has the decomposition
	\begin{equation}
	\label{decomposition coupled module}
	W^{(\gamma_1)}\otimes\dots\otimes W^{(\gamma_n)} = \bigoplus_{\mathbf{p}\in\N^{n-i-r}}\bigoplus_{N\in \N}V_{N,\mathbf{p}}^{(i,r)}
	\end{equation}
	and for any \(\mathbf{p}\in\N^{n-i-r}\) and \(N\in\N\), the set \(\mathcal{B}_{\mathbf{p},N}^{(i,r)}\) forms an orthogonal basis for \(V_{N,\mathbf{p}}^{(i,r)}\).
\end{proposition}
\begin{proof}
	We prove this by induction on \(n\), for fixed \(i\) and \(r\). The case \(n = i+r\) follows immediately from Lemma \ref{lemma U_N,2} and the fact that one has
	\[
	W^{(\gamma_1)}\otimes\dots\otimes W^{(\gamma_n)} = \bigoplus_{N = 0}^{\infty}U_{N,n},
	\]
	as is immediate from (\ref{U_N,i def}).
	
	Now suppose \(n > i+r\) and the claims have been proven for \(n-1\). Referring to the notation (\ref{U_N,i def}) we have
	\begin{align*}
	W^{(\gamma_1)}\otimes\dots\otimes W^{(\gamma_n)} &= \bigoplus_{M = 0}^{\infty} \left(W^{(\gamma_1)}\otimes\dots\otimes W^{(\gamma_n)}\right)\cap U_{M,n} \\
	& = \bigoplus_{M = 0}^{\infty}\bigoplus_{k = 0}^M\left(\left(W^{(\gamma_1)}\otimes\dots\otimes W^{(\gamma_{n-1})}\right)\cap U_{k,n-1}\right)\otimes \Big\langle \vert M-k\rangle_{\gamma_n} \Big\rangle.
	\end{align*}
	By the induction hypothesis, we have
	\[
	W^{(\gamma_1)}\otimes\dots\otimes W^{(\gamma_{n-1})} = \bigoplus_{\mathbf{p}\in\N^{n-i-r-1}}\bigoplus_{N\in \N}V_{N,\mathbf{p}}^{(i,r)},
	\]
	and moreover
	\[
	V_{N,\mathbf{p}}^{(i,r)}\cap U_{k,n-1} = \left\{
	\begin{array}{ll}
	V_{N,\mathbf{p}}^{(i,r)} \qquad & \mathrm{if}\ k = N + \vert\mathbf{p}\vert, \\
	\{0\} & \mathrm{otherwise}.
	\end{array}
	\right.
	\]
	Hence we find
	\begin{align*}
	W^{(\gamma_1)}\otimes\dots\otimes W^{(\gamma_n)} & = \bigoplus_{M = 0}^{\infty}\bigoplus_{k = 0}^M\bigoplus_{\substack{\mathbf{p}\in\N^{n-i-r-1},\\\vert\mathbf{p}\vert\leq k}} V_{k-\vert\mathbf{p}\vert,\mathbf{p}}^{(i,r)}\otimes \Big\langle \vert M-k\rangle_{\gamma_n} \Big\rangle \\ & = \bigoplus_{M = 0}^{\infty}\bigoplus_{\substack{\mathbf{p}\in\N^{n-i-r-1},\\\vert\mathbf{p}\vert\leq M}}\bigoplus_{k = \vert\mathbf{p}\vert}^M V_{k-\vert\mathbf{p}\vert,\mathbf{p}}^{(i,r)}\otimes \Big\langle \vert M-k\rangle_{\gamma_n} \Big\rangle \\& =
	\bigoplus_{M = 0}^{\infty}\bigoplus_{\substack{\mathbf{p}\in\N^{n-i-r-1},\\\vert\mathbf{p}\vert\leq M}}\bigoplus_{k = 0}^{M-\vert\mathbf{p}\vert} V_{k,\mathbf{p}}^{(i,r)}\otimes \Big\langle \vert M-k-\vert\mathbf{p}\vert\rangle_{\gamma_n} \Big\rangle,
	\end{align*}
	where we have switched the order of summation in the second line and renamed \(k\) in the third. 
	
	For \(\ell\leq k\) we will denote by \(\mathcal{Z}_{\bp,k,\ell}^{(i,r)}\) the set given by
	\begin{align*}
	& \left\{
	\Theta^{(i,r)}\left(j_1,\dots,j_{i+r-2},p_1-\vert\bj_{i+r-2}\vert,p_2,\dots,p_{n-i-r-1},\ell,k-\ell \right):\bj\in\N^{i+r-2}\ \mathrm{with}\ \vert\bj_{i+r-2}\vert \leq p_1
	\right\},
	\end{align*}
	if \(n-1>i+r\), and by
	\[
	\left\{ 
	\Theta^{(i,r)}\left(j_1,\dots,j_{n-3},\ell-\vert\bj_{n-3}\vert,k-\ell \right): \bj\in\N^{n-3}\ \mathrm{with}\ \vert\bj_{n-3}\vert\leq\ell
	\right\}
	\]
	if \(n-1=i+r\).
	By the induction hypothesis, a basis for \(V_{k,\mathbf{p}}^{(i,r)}\) is given by
	\[
	\mathcal{B}_{\bp,k}^{(i,r)} = \left\{\mathcal{Z}_{\bp,k,\ell}^{(i,r)}: \ell \in \{0,\dots,k\}\right\}.
	\]
	Hence we have
	\begin{align}
	\label{third line}
	\begin{split}
	& W^{(\gamma_1)}\otimes\dots\otimes W^{(\gamma_n)} \\ = &\ \bigoplus_{M = 0}^{\infty}\bigoplus_{\substack{\mathbf{p}\in\N^{n-i-r-1},\\\vp{}\leq M}}\bigoplus_{k = 0}^{M-\vp{}}\bigoplus_{\ell = 0}^{k}
	\Big\langle \mathcal{Z}_{\bp,k,\ell}^{(i,r)}\Big\rangle \otimes \vert M-k-\vp{}\rangle_{\gamma_n} \\
	= &\ \bigoplus_{M = 0}^{\infty}\bigoplus_{\substack{\mathbf{p}\in\N^{n-i-r-1},\\\vp{}\leq M}}\bigoplus_{\ell = 0}^{M-\vp{}}\bigoplus_{k = \ell}^{M-\vp{}}
	\Big\langle \mathcal{Z}_{\bp,k,\ell}^{(i,r)}\Big\rangle \otimes \vert M-k-\vp{}\rangle_{\gamma_n} \\
	= &\ \bigoplus_{M = 0}^{\infty}\bigoplus_{\substack{\mathbf{p}\in\N^{n-i-r},\\\vp{}\leq M}}
	\bigg\langle\Big\langle \mathcal{Z}_{\bp', k'+p_{n-i-r},p_{n-i-r}}^{(i,r)}\Big\rangle \otimes \vert M-k'-\vp{}\rangle_{\gamma_n}: k'\in \{0,\dots,M-\vert\bp\vert\}\bigg\rangle,
	\end{split}
	\end{align}
	where we have once more switched the order of summation in the third line and concatenated the second and third summation in the fourth line, after setting \(k' = k-\ell\) and renaming \(\ell\) to \(p_{n-i-r}\), such that in the last line we have \(\bp = (p_1,\dots,p_{n-i-r})\in\N^{n-i-r}\) and \(\bp'\) denotes \((p_1,\dots,p_{n-i-r-1})\in\N^{n-i-r-1}\).
	
	Let us now introduce the set \(\mathcal{C}_{\mathbf{p},M}^{(i,r)}\) given by
	\begin{align*}
	& \left\{\Theta^{(i,r)}\left(j_1,\dots,j_{i+r-2},p_1-\vert\bj_{i+r-2}\vert,p_2,\dots,p_{n-i-r},k' \right)
	\otimes \vert M-k'-\vp{}\rangle_{\gamma_n}:\right.\\&\left. \bj\in\N^{i+r-2}\ \mathrm{with}\ \vert\bj_{i+r-2}\vert\leq p_1, k'\in \{0,\dots,M-\vert\bp\vert\} \Big\rangle, \right\}.
	\end{align*}
	Observe from Lemma \ref{lemma Theta} that each of its elements belongs to \(V_{M-\vert\bp\vert,\mathbf{p}}^{(i,r)}\), and so
	\begin{equation}
	 \label{C_i,N}
	 \Big\langle \mathcal{C}_{\mathbf{p},M}^{(i,r)}\Big\rangle \subseteq V_{M-\vp{},\bp}^{(i,r)},
	\end{equation}
	and moreover it is clear that
	\[
	\Big\langle \mathcal{C}_{\mathbf{p},M}^{(i,r)}\Big\rangle = \bigg\langle\Big\langle \mathcal{Z}_{\bp', k'+p_{n-i-r},p_{n-i-r}}^{(i,r)}\Big\rangle \otimes \vert M-k'-\vp{}\rangle_{\gamma_n}: k'\in \{0,\dots,M-\vert\bp\vert\}\bigg\rangle.
	\]
	In combination with (\ref{third line}), this leads to 
	\begin{align*}
	W^{(\gamma_1)}\otimes\dots\otimes W^{(\gamma_n)} & \subseteq \bigoplus_{M = 0}^{\infty}\bigoplus_{\substack{\mathbf{p}\in\N^{n-i-r}\\\vp{}\leq M}}V_{M-\vp{},\bp}^{(i,r)},
	\end{align*}
	which after switching summation order and renaming summation indices again, becomes
	\begin{align}
	\label{decomposition final}
	W^{(\gamma_1)}\otimes\dots\otimes W^{(\gamma_n)} & \subseteq \bigoplus_{\mathbf{p}\in\N^{n-i-r}}\bigoplus_{N = 0}^{\infty}V_{N,\bp}^{(i,r)}.
	\end{align}
	But obviously the opposite inclusion also holds, since the spaces \(V_{N,\mathbf{p}}^{(i,r)}\) intersect trivially. Hence we get equality in (\ref{decomposition final}) and so (\ref{decomposition coupled module}) is proven. 
	
	Looking back at the fourth line in (\ref{third line}), we see that the inclusion in (\ref{C_i,N}) must also be an equality. This implies
	\begin{equation}
	\label{dimension smaller}
	\dim(V_{N,\bp}^{(i,r)})\leq \big\vert\mathcal{C}_{\bp,N+\vert\bp\vert}^{(i,r)}\big\vert.
	\end{equation}
	On the other hand, the set \(\mathcal{B}_{\mathbf{p},N}^{(i,r)}\) is contained in \(V_{N,\mathbf{p}}^{(i,r)}\) by Lemma \ref{lemma Theta}, and its elements are linearly independent, since they correspond to different eigenvalues for the generators of \(\mathcal{Y}_{i,r}\). Moreover, it is clear that
	\[
	\big\vert\mathcal{C}_{\bp,N+\vert\bp\vert}^{(i,r)}\big\vert = \big\vert\mathcal{B}_{\bp,N}^{(i,r)}\big\vert.
	\] 
	As a consequence, the inequality in (\ref{dimension smaller}) must be an equality as well, and both \(\mathcal{B}_{\mathbf{p},N}^{(i,r)}\) and \(\mathcal{C}_{\mathbf{p},N+\vp{}}^{(i,r)}\) must be bases for \(V_{N,\mathbf{p}}^{(i,r)}\). This concludes the induction.
	
	Finally, it follows from (\ref{decomposition coupled module}) that
	\[
	\bigcup_{\mathbf{p}\in\N^{n-i-r}}\bigcup_{N\in\N}\mathcal{B}_{\mathbf{p},N}^{(i,r)} = \left\{\Theta^{(i,r)}(j_1,\dots,j_{n-1},N'): \bj\in\N^{n-1}, N'\in\N \right\},
	\]
	is a basis for the coupled module \(W^{(\gamma_1)}\otimes \dots\otimes W^{(\gamma_n)}\), and hence so is the set (\ref{basis for coupled module}) by Definition \ref{definition bases final}. This concludes the proof.
\end{proof}

To conclude, let us further refine the spaces \(V_{N,\bp}^{(i,r)}\).

\begin{definition}
For \(\mathbf{j}\in\N^{n-1}\) and \(N\in\N\), we will write \(Z_{N,\bj}^{(i,r)} = Z_{N,j_1,\dots,j_{n-1}}^{(i,r)}\) for the joint eigenspace of the operators
\[
\Gamma_{[1;m]}^q \quad \mathrm{with\ eigenvalue} \quad (-1)^{\vert\mathbf{j}_{m-1}\vert}\left[\vert\mathbf{j}_{m-1}\vert+\gamma_{[1;m]}-\frac12\right]_q,
\]
with \(m\in \{2,\dots,i\}\cup\{i+r,\dots,n\}\),
\[
\Gamma_{[i+1;i+m]}^q \quad \mathrm{with\ eigenvalue} \quad (-1)^{\vert\mathbf{j}_{i+m-2}\vert-\vert\mathbf{j}_{i-1}\vert}\left[\vert\mathbf{j}_{i+m-2}\vert-\vert\mathbf{j}_{i-1}\vert+\gamma_{[i+1;i+m]}-\frac12\right]_q,
\]
with \(m\in \{2,\dots,r\}\), and
\[
K_{[1;n]} \quad \mathrm{with\ eigenvalue} \quad q^{\frac12(N+\vert\mathbf{j}_{n-1}\vert+\gamma_{[1;n]})}.
\]
\end{definition}

As an immediate consequence of Proposition \ref{prop decomposition}, we can state the following corollary.

\begin{corollary}
	\label{cor characterized by equations}
	The coupled module \(W^{(\gamma_1)}\otimes\dots\otimes W^{(\gamma_n)}\) can be decomposed as
	\begin{equation}
	\label{decomposition one-dimensional}
	W^{(\gamma_1)}\otimes\dots\otimes W^{(\gamma_n)} = \bigoplus_{\bj^{(i,r)}\in\N^{n-1}}\bigoplus_{N\in\N} Z_{N,\bj^{(i,r)}}^{(i,r)},
	\end{equation}
	where each of the spaces \(Z_{N,\bj^{(i,r)}}^{(i,r)}\) is one-dimensional and spanned by the vector \(\vert\bj^{(i,r)};N\rangle\). The vectors \(\vert\bj^{(i,r)};N\rangle\) are hence uniquely determined by the equations (\ref{eigenvalue equation 1}), (\ref{eigenvalue equation 2}) and (\ref{eigenvalue equation K}).
\end{corollary}

As a conclusion, we have associated to each maximal abelian subalgebra \(\mathcal{Y}_{i,r}\) of \(\mathcal{A}_n^q\) an orthonormal basis for the coupled module \(W^{(\gamma_1)}\otimes\dots\otimes W^{(\gamma_n)}\), which corresponds to the decomposition (\ref{decomposition one-dimensional}). In Section \ref{Paragraph - Connection coefficients} we will obtain explicit expressions for the Clebsch-Gordan coefficients between two such orthonormal bases.

\section{Connection coefficients}
\label{Paragraph - Connection coefficients}

In this section we will compute an explicit expression for the overlap coefficients
\begin{equation}
\label{coefficients to compute}
\langle \mathbf{j}^{(i,r)};m\vert\mathbf{j}^{(1,1)};m\rangle = \langle j_{n-1}^{(i,r)},\dots,j_1^{(i,r)};m \vert j_{n-1}^{(1,1)},\dots,j_1^{(1,1)};m\rangle
\end{equation}
related to the abelian subalgebras \(\mathcal{Y}_{i,r}\) and \(\mathcal{Y}_{1,1}\). Note that by Definition \ref{definition bases final} and (\ref{hermitian conjugate generators}) one has
\[
\langle \mathbf{j}^{(i,r)};m\vert\mathbf{j}^{(1,1)};m\rangle = \langle \mathbf{j}^{(i,r)};0\vert (A_-)_{[1;n]}^m(A_+)_{[1;n]}^m\vert\mathbf{j}^{(1,1)};0\rangle.
\]
By (\ref{relations with powers}) we may write
\[
(A_-)_{[1;n]}^m(A_+)_{[1;n]}^m = \prod_{\ell = 1}^m\mathcal{T}_{\ell}(K) + \Upsilon (A_-)_{[1;n]},
\]
for some \(\Upsilon\in \ospq^{\otimes n}\), and where
\[
\mathcal{T}_{\ell}(K) = q^{-1/2}\frac{q^{\ell}-(-1)^{\ell}}{q-q^{-1}}K^2_{[1;n]} - q^{1/2}\frac{q^{-\ell}-(-1)^{\ell}}{q-q^{-1}}K^{-2}_{[1;n]}.
\]
It hence follows from (\ref{eigenvalue equation Theta K}) and (\ref{eigenvalue A_- Theta}) that
\begin{equation}
\label{effect of m not zero}
\langle \mathbf{j}^{(i,r)};m\vert\mathbf{j}^{(1,1)};m\rangle = \left(\prod_{\ell = 1}^m\sigma_{\ell}^{(\vert\bj^{(1,1)}_{n-1}\vert + \gamma_{[1;n]})}\right)\langle \mathbf{j}^{(i,r)};0\vert\mathbf{j}^{(1,1)};0\rangle.
\end{equation}
Consequently, it will suffice to compute the coefficients (\ref{coefficients to compute}) for \(m = 0\), and we may as well write \(\vert\mathbf{j}^{(i,r)}\rangle\) for \(\vert\mathbf{j}^{(i,r)};0\rangle\).
We will first consider several intermediate bases in Section \ref{Paragraph univariate} before tackling the general case in Section \ref{Paragraph multivariate}.

\subsection{Univariate $(-q)$-Racah polynomials}
\label{Paragraph univariate}

The \(q\)-Racah polynomials depend on four parameters \(a,b,c\in\R\) and \(N\in\N\) and are defined as \cite{Askey&Wilson-1985}
\begin{equation}
\label{q-Racah univariate}
r_n\left(x;a,b,c,N\vert q\right) = (aq,bcq,q^{-N};q)_n\left(\frac{q^N}{c}\right)^{\frac{n}{2}}  \ _4\phi_3\left(\left.
\begin{array}{c}
q^{-n}, abq^{n+1}, q^{-x}, cq^{x-N} \\
aq, bcq, q^{-N}
\end{array}
\right| q,q 
\right),
\end{equation}
where we have used the standard notations for the \(q\)-Pochhammer symbol
\begin{equation}
\label{q-Pochhammer symbol}
(a;q)_n = \prod_{k = 0}^{n-1}(1-aq^k), \quad (a_1,\dots,a_s;q)_n = \prod_{k = 1}^s(a_k;q)_n
\end{equation}
and the basic hypergeometric series \cite{Koekoek&Lesky&Swarttouw-2010}
\begin{equation}
\label{basic hypergeometric series def}
\,_4\phi_3\left(\left.
\begin{array}{c}
a_1, a_2, a_3, a_4 \\
b_1, b_2, b_3
\end{array}
\right| q,z
\right) = \sum_{n = 0}^{\infty}\frac{(a_1,a_2,a_3,a_4;q)_n}{(b_1,b_2,b_3,q;q)_n}z^n.
\end{equation}
However, the polynomials we will need, have \(-q\) as a deformation parameter instead of \(q\), hence we will refer to the polynomials (\ref{q-Racah univariate}) with all \(q\) changed to \(-q\) as \((-q)\)-Racah polynomials. In this section we will show that the connection coefficients
\[
\langle\mathbf{j}^{(i,k+1)}\vert\mathbf{j}^{(i,k)}\rangle = \langle j^{(i,k+1)}_{n-1},\dots,j^{(i,k+1)}_1\vert j^{(i,k)}_{n-1},\dots,j^{(i,k)}_1\rangle
\]
are proportional to these \((-q)\)-Racah polynomials. A useful way to write these is the following.

\begin{lemma}
	\label{lemma sequence of deltas} The connection coefficients between the eigenbases of the intermediate subalgebras \(\mathcal{Y}_{i,k}\) and \(\mathcal{Y}_{i,k+1}\) can be expressed as
	\begin{align*}
	\langle\mathbf{j}^{(i,k+1)}\vert\mathbf{j}^{(i,k)}\rangle = &\  \delta_{j_{i+k-1}^{(i,k)}+j_{i+k}^{(i,k)},j_{i+k-1}^{(i,k+1)}+j_{i+k}^{(i,k+1)}}\prod_{\ell = 1}^{i+k-2}\delta_{j_{\ell}^{(i,k)},j_{\ell}^{(i,k+1)}}
	\prod_{\ell = i+k+1}^{n-1}\delta_{j_{\ell}^{(i,k)},j_{\ell}^{(i,k+1)}}\\ &\ \times\omega_{\mathbf{j}^{(i,k+1)}}G_{j_{i+k-1}^{(i,k)}}(\mathbf{j}^{(i,k+1)}),
	\end{align*}
	where 
	\[
	\omega_{\mathbf{j}^{(i,k+1)}} = \langle \mathbf{j}^{(i,k+1)}\vert j_{n-1}^{(i,k+1)},\dots,j_{i+k+1}^{(i,k+1)},j_{i+k}^{(i,k+1)}+j_{i+k-1}^{(i,k+1)},0,j_{i+k-2}^{(i,k+1)},\dots,j_{1}^{(i,k+1)}\rangle
	\]
	is a normalization factor, chosen such that
	\begin{equation}
	\label{normalization G0}
	G_0(\mathbf{j}^{(i,k+1)}) = 1.
	\end{equation}
\end{lemma}
\begin{proof} Writing \(\lambda_{[1;m+1]}^{(i,\ell)}(\mathbf{j}^{(i,\ell)})\) for the eigenvalue
	\begin{equation}
	\label{eigenvalue 1}
	(-1)^{\vj{i,\ell}{m}}\left[\vj{i,\ell}{m}+\gamma_{[1;m+1]}-\frac12\right]_q
	\end{equation}
	and \(\lambda_{[i+1;i+m]}^{(i,\ell)}(\mathbf{j}^{(i,\ell)})\) for
	\begin{equation}
	\label{eigenvalue 2}
	(-1)^{\vj{i,\ell}{i+m-2}-\vj{i,\ell}{i-1}}\left[\vj{i,\ell}{i+m-2}-\vj{i,\ell}{i-1}+\gamma_{[i+1;i+m]}-\frac12\right]_q,
	\end{equation}
	we find from Definition \ref{definition bases final} and Lemma \ref{lemma self-adjoint}
	\[
	\langle \mathbf{j}^{(i,k+1)} \vert \Gamma_A^q\vert \mathbf{j}^{(i,k)}\rangle = \lambda_A^{(i,k)}(\mathbf{j}^{(i,k)})\langle \mathbf{j}^{(i,k+1)}\vert \mathbf{j}^{(i,k)}\rangle = \lambda_A^{(i,k+1)}(\mathbf{j}^{(i,k+1)})\langle \mathbf{j}^{(i,k+1)}\vert \mathbf{j}^{(i,k)}\rangle
	\]
	where \(A\) is any of the sets \([1;2],\dots,[1;i],[i+1;i+2],\dots,[i+1;i+k],[1;i+k+1],\dots,[1;n]\). Hence we find that the connection coefficient will vanish unless
	\begin{equation}
	\label{condition 1}
	j_{\ell}^{(i,k)} = j_{\ell}^{(i,k+1)}
	\end{equation}
	for any \(\ell\in[1;i+k-2]\cup[i+k+1;n-1]\) and
	\begin{equation}
	\label{condition 2}
	j^{(i,k)}_{i+k-1}+j^{(i,k)}_{i+k} = j^{(i,k+1)}_{i+k-1}+j^{(i,k+1)}_{i+k}.
	\end{equation}
	This results in the formulation
	\[
	\langle\mathbf{j}^{(i,k+1)}\vert\mathbf{j}^{(i,k)}\rangle =  g(\mathbf{j}^{(i,k+1)},j^{(i,k)}_{i+k-1}) \ \delta_{\vj{i,k}{i+k},\vj{i,k+1}{i+k}}\prod_{\ell = 1}^{i+k-2}\delta_{j_{\ell}^{(i,k)},j_{\ell}^{(i,k+1)}}
	\prod_{\ell = i+k+1}^{n-1}\delta_{j_{\ell}^{(i,k)},j_{\ell}^{(i,k+1)}},
	\]
	with
	\[
	g(\mathbf{j}^{(i,k+1)},j^{(i,k)}_{i+k-1}) = \langle\mathbf{j}^{(i,k+1)}\vert j_{n-1}^{(i,k+1)},\dots,j_{i+k}^{(i,k+1)}+j_{i+k-1}^{(i,k+1)}-j^{(i,k)}_{i+k-1},j^{(i,k)}_{i+k-1},\dots,j_{1}^{(i,k+1)}\rangle.
	\]
	It is clear that \(\omega_{\mathbf{j}^{(i,k+1)}} = g(\mathbf{j}^{(i,k+1)},0)\). Defining now 
	\[
	G_{j_{i+k-1}^{(i,k)}}(\mathbf{j}^{(i,k+1)}) = \frac{g(\mathbf{j}^{(i,k+1)},j_{i+k-1}^{(i,k)})}{g(\mathbf{j}^{(i,k+1)},0)},
	\]
	the result follows.
\end{proof}

The only two operators that act differently on \(\vert\mathbf{j}^{(i,k)}\rangle\) and \(\vert\mathbf{j}^{(i,k+1)}\rangle\) are \(\Gamma_{[i+1;i+k+1]}^q\), which diagonalizes \(\vert\mathbf{j}^{(i,k+1)}\rangle\) but not \(\vert\mathbf{j}^{(i,k)}\rangle\), and \(\Gamma_{[1;i+k]}^q\), where we have the opposite. In the next proposition, we will show that \(\Gamma_{[i+1;i+k+1]}^q\) will act on \(\vert\mathbf{j}^{(i,k)}\rangle\) in a tridiagonal fashion. Let us first introduce some notation. Let
\begin{align}
\label{akk, bkk}
\begin{split}
\akk & = (-q)^{\vjk{i+k-2}-\vjk{i-1}}q^{\gamma_{[i+1;i+k+1]}-\frac12} \\
\bkk & = -(-q)^{-\vjk{i+k}+\vjk{i-1}}q^{-\gamma_{[i+1;i+k+1]}+\frac12} \\
\ckk & = -(-q)^{\vjk{i+k}+\vjk{i-1}}q^{2\gamma_{[1;i]}+\gamma_{[i+1;i+k+1]}-\frac12} \\
\dkk & = (-q)^{\vjk{i+k-2}-\vjk{i-1}}q^{\gamma_{[i+1;i+k]}-\gamma_{i+k+1}+\frac12},
\end{split}
\end{align}
and
\begin{align}
\label{Expressions A_sk, C_sk}
\begin{split}
A_{s}^{(i,k)} = &\, -\frac{(1+(-q)^s\akk\bkk)(1-(-q)^s\akk\ckk)(1-(-q)^s\akk\dkk)}{\akk(1-(-q)^{2s-1}\akk\bkk\ckk\dkk)(1-(-q)^{2s}\akk\bkk\ckk\dkk)}\\&\times(1-(-q)^{s-1}\akk\bkk\ckk\dkk), \\
C_s^{(i,k)} = &\, \frac{\akk(1-(-q)^s)(1-(-q)^{s-1}\bkk\ckk)(1-(-q)^{s-1}\bkk\dkk)}{(1-(-q)^{2s-2}\akk\bkk\ckk\dkk)(1-(-q)^{2s-1}\akk\bkk\ckk\dkk)}\\&\times(1+(-q)^{s-1}\ckk\dkk).
\end{split}
\end{align}
Define also
\begin{align}
\label{tridiagonal coefficients}
\begin{split}
V_{\mathbf{j}^{(i,k)}} & = \frac{\akk-\left(\akk\right)^{-1}-A_{j_{i+k-1}^{(i,k)}}^{(i,k)}-C_{j_{i+k-1}^{(i,k)}}^{(i,k)}}{q-q^{-1}}, \\
U_{\mathbf{j}^{(i,k)}} & = \frac{\sqrt{A_{j_{i+k-1}^{(i,k)}-1}^{(i,k)}C_{j_{i+k-1}^{(i,k)}}^{(i,k)}}}{q-q^{-1}}.
\end{split}
\end{align}
Then one can show the following proposition.
\begin{proposition}
	\label{prop tridiagonal action}
	The operator \(\Gamma_{[i+1;i+k+1]}^q\) acts tridiagonally on the basis functions \(\vert \mathbf{j}^{(i,k)}\rangle\):
	\begin{align*}
	\Gamma_{[i+1;i+k+1]}^q\vert \mathbf{j}^{(i,k)}\rangle = &\  V_{\mathbf{j}^{(i,k)}}\vert \mathbf{j}^{(i,k)}\rangle+ U_{\mathbf{j}^{(i,k)}} \vert \mathbf{j}^{(i,k)}- \mathbf{h}_{i+k-1}\rangle+ U_{\mathbf{j}^{(i,k)}+\mathbf{h}_{i+k-1}} \vert\mathbf{j}^{(i,k)} +\mathbf{h}_{i+k-1} \rangle,
	\end{align*}
	where \(\mathbf{h}_{i+k-1} = (\underbrace{0,\dots,0}_{i+k-2\ \mathrm{times}},1,-1,\underbrace{0,\dots,0}_{n-i-k-1\ \mathrm{times}}) \).
\end{proposition}
\begin{proof}
	The proof follows the same steps as Theorem 2 in \cite{DeBie&DeClercq&vandeVijver-2018}. The eigenvalue equations (\ref{eigenvalue equation 1}), (\ref{eigenvalue equation 2}) and the algebra relations (\ref{q-anticommutation Gamma}) allow to obtain expressions for the coefficients \(V_{\mathbf{j}^{(i,k)}}\) and \(U_{\mathbf{j}^{(i,k)}}\), whose factorizations (\ref{tridiagonal coefficients}) can be checked using any computer algebra package.
\end{proof}

Assuming that the conditions (\ref{condition 1}) and (\ref{condition 2}) are fulfilled and using Proposition \ref{prop tridiagonal action} and the notation (\ref{eigenvalue 2}), we obtain a three-term recurrence relation for the functions \(G_{j_{i+k-1}^{(i,k)}}\left(\mathbf{j}^{(i,k+1)}\right)\), similar to \cite[formula (3.33)]{Genest&Vinet&Zhedanov-2016}, namely
\begin{align}
\label{recursion G hat}
\begin{split}
& \langle \mathbf{j}^{(i,k+1)}\vert\Gamma_{[i+1;i+k+1]}^q\vert \mathbf{j}^{(i,k)}\rangle \\ = &\  \lambda_{[i+1;i+k+1]}^{(i,k+1)}(\mathbf{j}^{(i,k+1)})\, \omega_{\mathbf{j}^{(i,k+1)}} G_{j_{i+k-1}^{(i,k)}}\left(\mathbf{j}^{(i,k+1)}\right)\\
= &\ \omega_{\mathbf{j}^{(i,k+1)}}V_{\mathbf{j}^{(i,k)}}G_{j_{i+k-1}^{(i,k)}}\left(\mathbf{j}^{(i,k+1)}\right)\\&+ \omega_{\mathbf{j}^{(i,k+1)}}\left(U_{\mathbf{j}^{(i,k)}}G_{j_{i+k-1}^{(i,k)}-1}\left(\mathbf{j}^{(i,k+1)}\right)+U_{\mathbf{j}^{(i,k)}+\mathbf{h}_{i+k-1}}G_{j_{i+k-1}^{(i,k)}+1}\left(\mathbf{j}^{(i,k+1)}\right) \right).
\end{split}
\end{align}
We will also renormalize as in \cite{Genest&Vinet&Zhedanov-2016}: define
\begin{equation}
\label{renormalization}
\widehat{G}_{j_{i+k-1}^{(i,k)}}(\mathbf{j}^{(i,k+1)}) = \left(-\akk\right)^{j_{i+k-1}^{(i,k)}}\left(\prod_{\ell = 1}^{j_{i+k-1}^{(i,k)}}\sqrt{A_{\ell-1}^{(i,k)}C_{\ell}^{(i,k)}}\right) G_{j_{i+k-1}^{(i,k)}}(\mathbf{j}^{(i,k+1)}).
\end{equation}
These functions turn out to be proportional to \((-q)\)-Racah polynomials.
\begin{proposition}
	\label{prop G hat}
	The functions \(\widehat{G}_{j_{i+k-1}^{(i,k)}}(\mathbf{j}^{(i,k+1)})\) can be expressed as
	\begin{align}
	\label{G hat}
	\begin{split}
	\widehat{G}_{j_{i+k-1}^{(i,k)}}(\mathbf{j}^{(i,k+1)}) = F_{i,k} r_{j_{i+k-1}^{(i,k+1)}}\left(\left.j_{i+k-1}^{(i,k)}; \gammakk, \deltakk, \betakk,\widetilde{N}^{(i,k)}\right| -q\right),
	\end{split}
	\end{align}
	where \(r_n(x;a,b,c,N\vert -q)\) is the \((-q)\)-Racah polynomial (\ref{q-Racah univariate}), where \(\widetilde{N}^{(i,k)} = \vj{i,k}{i+k}-\vj{i,k}{i+k-2}\) and
	\begin{equation}
	\label{def alphakk}
	\begin{alignedat}{3}
	\alphakk & = q^{-1}\akk\bkk, \qquad &\betakk & = q^{-1}\ckk\dkk, \\
	\gammakk & = -q^{-1}\akk\dkk, \qquad &\deltakk & = -\frac{\akk}{\dkk},
	\end{alignedat}
	\end{equation}
	and where the proportionality coefficient is given by
	\[
	F_{i,k} = \frac{\left(\betakk(-q)^{-\widetilde{N}^{(i,k)}}\right)^{j_{i+k-1}^{(i,k+1)}/2}}{\left(\betakk(-q)^{-j_{i+k}^{(i,k)}};-q\right)_{j_{i+k-1}^{(i,k)}}}\frac{(-\gammakk q,-\betakk\deltakk q, (-q)^{-\widetilde{N}^{(i,k)}}; -q)_{j_{i+k-1}^{(i,k)}}}{(-\gammakk q,-\betakk\deltakk q, (-q)^{-\widetilde{N}^{(i,k)}}; -q)_{j_{i+k-1}^{(i,k+1)}}}.
	\]
\end{proposition}
\begin{proof}
	Let us introduce some more notation. Define
	\[
	\mu(j_{i+k-1}^{(i,k+1)}) = (-q)^{-j_{i+k-1}^{(i,k+1)}}+\gammakk\deltakk(-q)^{j_{i+k-1}^{(i,k+1)}+1},
	\]
	with \(\gammakk,\deltakk\) as in (\ref{def alphakk}) and let
	\[
	\check{A}_s^{(i,k)} = -\akk A_s^{(i,k)}, \quad \check{C}_s^{(i,k)} = -\akk C_s^{(i,k)},
	\]
	with \(A_s^{(i,k)},C_s^{(i,k)}\) as in (\ref{Expressions A_sk, C_sk}). Observe that
	\begin{align*}
	\check{A}_s^{(i,k)} & = \frac{(1-(-q)^{s+1}\alphakk)(1-(-q)^{s+1}\alphakk\betakk)(1-(-q)^{s+1}\betakk\deltakk)}{(1-(-q)^{2s+1}\alphakk\betakk)(1-(-q)^{2s+2}\alphakk\betakk)}\\&\times (1-(-q)^{s+1}\gammakk), \\
	\check{C}_s^{(i,k)} & = \frac{(-q)(1-(-q)^{s})(1-(-q)^s\betakk)(\gammakk-(-q)^s\alphakk\betakk)}{(1-(-q)^{2s}\alphakk\betakk)(1-(-q)^{2s+1}\alphakk\betakk)}\\&\times (\deltakk-(-q)^s\alphakk).
	\end{align*}
	Using these notations and (\ref{renormalization}), the recursion relation (\ref{recursion G hat}) reduces to
	\begin{align*}
	& \mu(j_{i+k-1}^{(i,k+1)})\widehat{G}_{j_{i+k-1}^{(i,k)}}(\mathbf{j}^{(i,k+1)}) \\ = &\ \widehat{G}_{j_{i+k-1}^{(i,k)}+1}(\mathbf{j}^{(i,k+1)}) + \left(1-q\gammakk\deltakk-\check{A}_{j_{i+k-1}^{(i,k)}}^{(i,k)}- \check{C}_{j_{i+k-1}^{(i,k)}}^{(i,k)}\right)\widehat{G}_{j_{i+k-1}^{(i,k)}}(\mathbf{j}^{(i,k+1)}) \\ & + \check{A}_{j_{i+k-1}^{(i,k)}-1}^{(i,k)}\check{C}_{j_{i+k-1}^{(i,k)}}^{(i,k)}\widehat{G}_{j_{i+k-1}^{(i,k)}-1}(\mathbf{j}^{(i,k+1)}).
	\end{align*}
	This is the well-known recursion relation for \((-q)\)-Racah polynomials \cite{Koekoek&Lesky&Swarttouw-2010}. Taking into account the initial condition (\ref{normalization G0}), we can write its solution as
	\begin{align}
	\label{recurrence relation solution}
	\begin{split}
	\widehat{G}_{j_{i+k-1}^{(i,k)}}(\mathbf{j}^{(i,k+1)}) = & \ \frac{(-\alphakk q,-\betakk\deltakk q, -\gammakk q; -q)_{j_{i+k-1}^{(i,k)}}}{(\alphakk\betakk(-q)^{j_{i+k-1}^{(i,k)}+1};-q)_{j_{i+k-1}^{(i,k)}}}\\&\times f_{j_{i+k-1}^{(i,k)}}\left(\left.\mu(j_{i+k-1}^{(i,k+1)}); \alphakk,\betakk,\gammakk,\deltakk\right| -q\right).
	\end{split}
	\end{align}
	where
	\begin{align*}
	f_{n}\left(\mu(x); \alpha,\beta,\gamma,\delta\vert -q\right) & = \,_4\phi_3\left(\left.
	\begin{array}{c}
	(-q)^{-n}, \alpha\beta(-q)^{n+1}, (-q)^{-x}, \gamma\delta(-q)^{x+1} \\
	-\alpha q, -\beta\delta q, -\gamma q
	\end{array}
	\right| -q,-q 
	\right),
	\end{align*}
	and
	\[
	\mu(x) = (-q)^{-x}+\gamma\delta(-q)^{x+1}.
	\]
	Upon permuting the arguments of the basic hypergeometric function, it is seen that these functions relate to the \((-q)\)-Racah polynomials as
	\begin{align}
	\label{Gasper Rahman notation}
	\begin{split}
	r_n\left(x;a,b,c,N\vert-q\right)
	= & \, (-aq,-bcq,(-q)^{-N};-q)_n\left(\frac{(-q)^N}{c}\right)^{\frac{n}{2}}\\&\times f_x\left(\mu(n); (-q)^{-N-1}, c, a, b\vert -q\right).
	\end{split}
	\end{align}
	Observing that \(\alphakk = (-q)^{-\widetilde{N}^{(i,k)}-1}\), we can combine (\ref{recurrence relation solution}) and (\ref{Gasper Rahman notation}) to find the expression (\ref{G hat}).
\end{proof}

\subsection{Multivariate $(-q)$-Racah polynomials}
\label{Paragraph multivariate}
Gasper and Rahman introduced their multivariable \(q\)-Racah polynomials in \cite{Gasper&Rahman-2007} as \(q\)-analogs of Tratniks \(q = 1\) Racah polynomials \cite{Tratnik-1991}. They were defined as entangled products of univariate \(q\)-Racah polynomials, depending on \(s\) discrete indices \(n_i\), \(s\) continuous variables \(x_i\), and parameters \(a_1,\dots,a_{s+1},b\in \R\) and \(N\in\N\): 
\begin{align}
\label{q-Racah multivariate def}
\begin{split}
& R_{(n_1,\dots,n_s)}(\left.(x_1,\dots,x_s);a_1,\dots,a_{s+1},b,N\right|q)  \\
= & \, \prod_{k = 1}^s r_{n_k}\left(\left.x_k-N_{k-1}; \frac{bA_k}{a_1}q^{2N_{k-1}}, \frac{a_{k+1}}{q}, A_kq^{x_{k+1}+N_{k-1}}, x_{k+1}-N_{k-1}\right|q\right),
\end{split}
\end{align}
where
\begin{equation}
\label{notation Nk, Ak}
N_k = \sum_{i = 1}^kn_i, \quad A_k = \prod_{i = 1}^ka_i, \quad x_{s+1} = N.
\end{equation}
They are homogeneous polynomials of total degree \(N_s\) in the variables 
\[
q^{-x_k}+A_kq^{x_k}, \qquad k \in \{1,\dots,s\}.
\]
They turn out to be orthogonal on the simplex 
\[
\{(x_1,\dots,x_s)\in\N^s: 0\leq x_1\leq x_2\leq\dots\leq x_s\leq N\},
\]
in the sense that
\begin{equation}
\label{multivariate q-racah orthogonality}
\sum_{x_s = 0}^N\sum_{x_{s-1} = 0}^{x_s}\dots\sum_{x_1 = 0}^{x_2}\rho(\mathbf{x})R_{\mathbf{n}}\left(\left.\mathbf{x}; \mathbf{a}, b, N \right| q\right)R_{\mathbf{n'}}\left(\left.\mathbf{x}; \mathbf{a}, b, N \right| q\right) = h_{\mathbf{n}}\delta_{\mathbf{n},\mathbf{n'}},
\end{equation}
for a certain weight function \(\rho(\mathbf{x})\) and normalization coefficient \(h_{\mathbf{n}}\). We refer to \cite[(2.14) and (2.16)]{Gasper&Rahman-2007} for explicit expressions.

As before, the polynomials that we will make use of, are multivariate \((-q)\)-Racah polynomials, i.e. the polynomials (\ref{q-Racah multivariate def}) where we change all \(q\) to \(-q\).

In Section \ref{Paragraph univariate} we found expressions for the overlap coefficients between the eigenbases of the algebras \(\mathcal{Y}_{i,k}\) and \(\mathcal{Y}_{i,k+1}\). Gluing these intermediate steps together, we can compute the overlaps
\[
\langle \mathbf{j}^{(i,r)}\vert\mathbf{j}^{(1,1)}\rangle = \langle j_{n-1}^{(i,r)},\dots,j_1^{(i,r)} \vert j_{n-1}^{(1,1)},\dots,j_1^{(1,1)}\rangle
\]
corresponding to the subalgebras \(\mathcal{Y}_{1,1} = \mathcal{Y}_{i,1}\) and \(\mathcal{Y}_{i,r}\).

\begin{lemma}
	\label{lemma products of G}
	The connection coefficients between the eigenbases of the subalgebras \(\mathcal{Y}_{1,1}\) and \(\mathcal{Y}_{i,r}\) have the expression
	\begin{align*}
	\langle \mathbf{j}^{(i,r)}\vert\mathbf{j}^{(1,1)}\rangle = &\ \delta_{\vj{1,1}{i+r-1},\vj{i,r}{i+r-1}}\prod_{\ell = 1}^{i-1} \delta_{j_\ell^{(1,1)},j_{\ell}^{(i,r)}}\prod_{\ell = i+r}^{n-1}\delta_{j_{\ell}^{(1,1)},j_{\ell}^{(i,r)}}\prod_{\ell= i}^{i+r-2}\omega_{\,\widehat{\mathbf{j}}^{(i,r,\ell)}}G_{\vj{1,1}{\ell}-\vj{i,r}{\ell-1}}\left(\widehat{\mathbf{j}}^{(i,r,\ell)}\right)
	\end{align*}
	with
	\[
	\widehat{\mathbf{j}}^{(i,r,\ell)} = (j_1^{(i,r)}\dots,j_{\ell}^{(i,r)},\vj{1,1}{\ell+1}-\vj{i,r}{\ell},j^{(1,1)}_{\ell+2},\dots,j^{(1,1)}_{n-1}).
	\]
\end{lemma}
\begin{proof}
	By induction on \(r\). The case \(r = 2\) follows from Lemma \ref{lemma sequence of deltas} with \(k = 1\), upon observing that \(\vert\mathbf{j}^{(i,1)}\rangle = \vert\mathbf{j}^{(1,1)}\rangle\) by (\ref{(i,1) = (1,1)}). Assuming the statement holds for \(r\), one proceeds by writing
	\[
	\langle\mathbf{j}^{(i,r+1)}\vert\mathbf{j}^{(1,1)}\rangle = \sum_{\mathbf{j}^{(i,r)}}\langle\mathbf{j}^{(i,r+1)}\vert\mathbf{j}^{(i,r)}\rangle\langle\mathbf{j}^{(i,r)}\vert\mathbf{j}^{(1,1)}\rangle
	\]
	and using the induction hypothesis and Lemma \ref{lemma sequence of deltas}. The products of Kronecker deltas reveal that each term in the sum above vanishes unless
	\begin{align}
	\label{fixing indices}
	\begin{split}
	j^{(i,r)}_{\ell} & = j^{(i,r+1)}_{\ell}, \qquad \qquad\qquad \ell\in[1;i+r-2] \\
	j^{(i,r)}_{i+r-1} & = \vj{1,1}{i+r-1}-\vj{i,r+1}{i+r-2}\\
	j^{(i,r)}_{\ell} & = j^{(1,1)}_{\ell}, \qquad\qquad\qquad\quad \ell\in[i+r;n-1].
	\end{split}
	\end{align}
	This concludes the induction.
\end{proof}
	
Substituting the expressions for \(G_{\vj{1,1}{\ell}-\vj{i,r}{\ell-1}}\left(\widehat{\mathbf{j}}^{(i,r,\ell)}\right)\) as obtained in Proposition \ref{prop G hat}, we find an explicit expression for the overlap coefficients. In order to assure that these coefficients are real-valued, we will from now on impose the supplementary requirement that \(q\) be a real number between 0 and 1.

\begin{theorem}
	\label{th - Overlap q-case}
	The overlap coefficients between the eigenbases of \(\mathcal{Y}_{1,1}\) and \(\mathcal{Y}_{i,r}\) can be expressed as
	\begin{align}
	\label{connection coeff without proportionality constant}
	\begin{split}
	& \langle \mathbf{j}^{(i,r)}\vert \mathbf{j}^{(1,1)}\rangle\\ = & \ \delta_{\vj{1,1}{i+r-1},\vj{i,r}{i+r-1}}\prod_{\ell = 1}^{i-1} \delta_{j_\ell^{(1,1)},j_{\ell}^{(i,r)}}\prod_{\ell = i+r}^{n-1}\delta_{j_{\ell}^{(1,1)},j_{\ell}^{(i,r)}} \sqrt{\frac{\rho^{(i,r)}(\mathbf{j}^{(1,1)})}{h_{\mathbf{j}^{(i,r)}}}} \\& R_{\mathbf{n}}\left(\left.\mathbf{x}; (-q)^{2\vj{1,1}{i-1}-1}q^{2\gamma_{[1;i+1]}}, q^{2\gamma_{i+2}},\dots, q^{2\gamma_{i+r}}, -q^{2\gamma_{i+1}-1}, \vj{1,1}{i+r-1}-\vj{1,1}{i-1}\right| -q\right),
	\end{split}
	\end{align}
	where \(\mathbf{n}\) and \(\mathbf{x}\) stand for
	\begin{equation}
	\label{n_k, x_k}
	\mathbf{n} = (n_1,\dots,n_{r-1}), \quad n_k = j_{i+k-1}^{(i,r)}, \quad \mathbf{x} = (x_1,\dots,x_{r-1}), \quad x_k = \vj{1,1}{i+k-1}-\vj{1,1}{i-1}.
	\end{equation}
	The expressions for the functions \(\rho^{(i,r)}(\mathbf{j}^{(1,1)})\) and \(h_{\mathbf{j}^{(i,r)}}\) can be found in Appendix A.
\end{theorem}
\begin{proof}
	Combining Lemma \ref{lemma products of G}, Proposition \ref{prop G hat} and the expression (\ref{renormalization}), we find
	\begin{align*}
	& \langle \mathbf{j}^{(i,r)}\vert \mathbf{j}^{(1,1)}\rangle\\ = & \ \delta_{\vj{1,1}{i+r-1},\vj{i,r}{i+r-1}}\prod_{\ell = 1}^{i-1} \delta_{j_\ell^{(1,1)},j_{\ell}^{(i,r)}}\prod_{\ell = i+r}^{n-1}\delta_{j_{\ell}^{(1,1)},j_{\ell}^{(i,r)}}  \\& C_{\mathbf{j}^{(1,1)},\mathbf{j}^{(i,r)}}\prod_{k = i}^{i+r-2}r_{j^{(i,r)}_k}\left(\left.\vj{1,1}{k}-\vj{i,r}{k-1}; \gammak, \deltak, \betak,\vj{1,1}{k+1}-\vj{i,r}{k-1}\right| -q\right),
	\end{align*}
	for some proportionality factor \(C_{\mathbf{j}^{(1,1)},\mathbf{j}^{(i,r)}}\) depending on \(\mathbf{j}^{(1,1)}\) and \(\mathbf{j}^{(i,r)}\), and with
	\begin{align*}
	\betak & = q^{-1}\ck\dk, \qquad & \gammak & = -q^{-1}\ak\dk, \qquad &\deltak & = -\frac{\ak}{\dk},
	\end{align*}
	where
	\begin{align*}
	\ak & = (-q)^{\vj{i,r}{k-1}-\vj{i,r}{i-1}}q^{\gamma_{[i+1;k+2]}-\frac12} \\
	\bkm & = -(-q)^{-\vj{1,1}{k+1}+\vj{i,r}{i-1}}q^{-\gamma_{[i+1;k+2]}+\frac12} \\
	\ck & = -(-q)^{\vj{1,1}{k+1}+\vj{i,r}{i-1}}q^{2\gamma_{[1;i]}+\gamma_{[i+1;k+2]}-\frac12} \\
	\dk & = (-q)^{\vj{i,r}{k-1}-\vj{i,r}{i-1}}q^{\gamma_{[i+1;k+1]}-\gamma_{k+2}+\frac12},
	\end{align*}
	are obtained from (\ref{akk, bkk}) upon consecutively fixing the intermediary indices \(j_{\ell}^{(i,k)}\) by (\ref{fixing indices}) with \(k\) instead of \(r\). Comparison with (\ref{q-Racah multivariate def}), with all \(q\) changed to \(-q\), tells us that the product of univariate polynomials above is in fact a multivariate \((-q)\)-Racah polynomial with parameters
	\begin{equation}
	\label{parameters in Gasper-Rahman notation}
	\begin{alignedat}{5}
	s & = r-1, \qquad & a_1 & = (-q)^{2\vj{1,1}{i-1}-1}q^{2\gamma_{[1;i+1]}}, \qquad&  a_k & = q^{2\gamma_{i+k}}, k\in \{2,\dots,r\}, \\ b & = -q^{2\gamma_{i+1}-1}, \qquad & N & = \vj{1,1}{i+r-1}-\vj{1,1}{i-1} &&
	\end{alignedat}
	\end{equation}
	and with \(n_k\) and \(x_k\) as in (\ref{n_k, x_k}). 
	
	The factor \(C_{\mathbf{j}^{(1,1)},\mathbf{j}^{(i,r)}}\) can be obtained from the orthogonality relation for the multivariate \((-q)\)-Racah polynomials. Since the vectors \(\vert\mathbf{j}^{(i,k)}\rangle\) are only determined up to a phase factor, we may assume the \(C_{\mathbf{j}^{(1,1)},\mathbf{j}^{(i,r)}}\) to be real and positive. As we have chosen the functions \(\vert \mathbf{j}^{(i,r)}\rangle\) to be mutually orthonormal, we have
	\begin{align*}
	\langle\mathbf{j'}^{(i,r)}\vert \mathbf{j}^{(i,r)}\rangle = \delta_{\mathbf{j}^{(i,r)},\mathbf{j'}^{(i,r)}},
	\end{align*}
	whereas on the other hand we have
	\begin{align*}
	\langle\mathbf{j'}^{(i,r)}\vert \mathbf{j}^{(i,r)}\rangle & = \sum_{\mathbf{j}^{(1,1)}} \langle\mathbf{j'}^{(i,r)}\vert \mathbf{j}^{(1,1)}\rangle\langle\mathbf{j}^{(1,1)}\vert \mathbf{j}^{(i,r)}\rangle \\
	& = \delta_{\vj{i,r}{i+r-1},\vert\mathbf{j'}^{(i,r)}_{i+r-1}\vert}\prod_{\ell = 1}^{i-1} \delta_{j_\ell^{(i,r)},{j'}_{\ell}^{(i,r)}}\prod_{\ell = i+r}^{n-1}\delta_{j_{\ell}^{(i,r)},{j'}_{\ell}^{(i,r)}}\\&\quad \sum_{\mathbf{j}^{(1,1)}} C_{\mathbf{j}^{(1,1)},\mathbf{j}^{(i,r)}}C_{\mathbf{j}^{(1,1)},\mathbf{j'}^{(i,r)}}R_{\mathbf{n}}\left(\mathbf{x};\mathbf{a},b,N\vert -q\right)R_{\mathbf{n'}}\left(\mathbf{x};\mathbf{a},b,N\vert -q\right),
	\end{align*}
	with \(\mathbf{a}=(a_1,\dots,a_{r})\), \(b\) and \(N\) as in (\ref{parameters in Gasper-Rahman notation}), \(\mathbf{x}\) as in (\ref{n_k, x_k}) and
	\[
	\mathbf{n} = (j_{i}^{(i,r)},\dots,j_{i+r-2}^{(i,r)}), \quad \mathbf{n'} = ({j'}_{i}^{(i,r)},\dots,{j'}_{i+r-2}^{(i,r)}).
	\]
	The sum is over all vectors \(\mathbf{j}^{(1,1)}\) with
	\[
	j_{\ell}^{(1,1)} = j_{\ell}^{(i,r)}, \quad \ell \in [1;i-1]\cup[i+r;n-1], \qquad \vj{1,1}{i+r-1}=\vj{i,r}{i+r-1},
	\]
	hence we are in fact summing over all \((j_i^{(1,1)},j_{i+1}^{(1,1)},\dots,j_{i+r-2}^{(1,1)})\in\N^{r-1}\) satisfying 
	\[
	0 \leq j_i^{(1,1)} \leq j_i^{(1,1)}+j_{i+1}^{(1,1)} \leq \dots \leq \vj{1,1}{i+r-2}-\vj{1,1}{i-1} \leq N,
	\]
	precisely as required in the orthogonality relation (\ref{multivariate q-racah orthogonality}) of the multivariate \((-q)\)-Racah polynomials. We conclude that the factor \(C_{\mathbf{j}^{(1,1)},\mathbf{j}^{(i,r)}}C_{\mathbf{j}^{(1,1)},\mathbf{j'}^{(i,r)}}\) coincides with 
	\[
	\frac{\rho(\mathbf{x})}{h_{\mathbf{n}}},
	\]
	with the substitutions (\ref{n_k, x_k}) and (\ref{parameters in Gasper-Rahman notation}), which lead to the expressions in Appendix A. This factor is positive by our requirement that \(0<q<1\), hence \(C_{\mathbf{j}^{(1,1)},\mathbf{j}^{(i,r)}}\) can be identified with its positive square root. This concludes the proof.
\end{proof}

\begin{remark}
	From the orthonormality of the vectors \(\vert\mathbf{j}^{(1,1)}\rangle\) one can deduce a second identity for the \((-q)\)-Racah polynomials, known as the dual orthogonality relation. Indeed, equating
	\[
	\delta_{\mathbf{j}^{(1,1)},\mathbf{j'}^{(1,1)}} = \langle\mathbf{j'}^{(1,1)}\vert \mathbf{j}^{(1,1)}\rangle = \sum_{\mathbf{j}^{(i,r)}} \langle\mathbf{j'}^{(1,1)}\vert \mathbf{j}^{(i,r)}\rangle\langle\mathbf{j}^{(i,r)}\vert \mathbf{j}^{(1,1)}\rangle,
	\]
	we obtain an expression of the form
	\[
	\sum_{\mathbf{n}\in\N^{r-1}}\frac{\rho(\mathbf{x})}{h_{\mathbf{n}}}R_{\mathbf{n}}\left(\left.\mathbf{x};\mathbf{a},b,N\right| -q\right)R_{\mathbf{n}}\left(\left.\mathbf{x'};\mathbf{a},b,N\right| -q\right) = \delta_{\mathbf{x},\mathbf{x'}},
	\]
	as was obtained in \cite[formula (49b)]{Genest&Iliev&Vinet-2018}, using duality properties of the multivariate \(q\)-Racah polynomials. This duality between the degrees \(n_k\) and the variables \(x_k\) will also be of use for our discrete realization of the \(q\)-Bannai--Ito algebra in the next section.
\end{remark}

\section{Realization with difference operators}
\label{Section - Discrete realization}

The multivariate \(q\)-Racah polynomials (\ref{q-Racah multivariate def}), like their univariate counterparts, can be obtained upon reparametrization and truncation from the multivariate Askey--Wilson polynomials. For convenience of the reader we repeat here the definition given by Gasper and Rahman in \cite{Gasper&Rahman-2005}. Let \(\mathbf{y} = (y_1,\dots,y_{s})\in\R^{s}\), \(\mathbf{n}=(n_1,\dots,n_s)\in\N^{s}\) and \(\boldsymbol{\alpha} = (\alpha_0,\dots,\alpha_{s+2})\in\R^{s+3}\), then the multivariate Askey--Wilson polynomials are given by
\begin{equation}
\label{multivariate AW def}
P_s\left(\left.\mathbf{n}; \mathbf{y}, \boldsymbol{\alpha}\right| q\right) = \prod_{k = 1}^s p_{n_k}\left(\left.y_k; \alpha_kq^{N_{k-1}},\frac{\alpha_k}{\alpha_0^2}q^{N_{k-1}}, \frac{\alpha_{k+1}}{\alpha_k}z_{k+1}, \frac{\alpha_{k+1}}{\alpha_k} z_{k+1}^{-1}\right| q\right),
\end{equation}
where \(z_i\) is such that \(y_i = \frac12(z_i+z_i^{-1})\) for \(i\in[1;s]\) and \(z_{s+1} = \alpha_{s+2}\), and where the univariate Askey--Wilson polynomials are defined as \cite{Askey&Wilson-1985}
\[
p_n(y;a,b,c,d\vert q) = \frac{(ab,ac,ad;q)_n}{a^n}\,_4\phi_3\left(\begin{array}{c}
q^{-n}, abcdq^{n-1}, az,az^{-1} \\ ab, ac, ad
\end{array}; q,q \right),
\]
with again \(y = \frac12(z+z^{-1})\).

In \cite{Iliev-2011} Iliev obtained a set of \(s\) mutually commuting and algebraically independent difference operators acting on the variables \(z_i\), which are all diagonalized by the polynomials (\ref{multivariate AW def}). A duality transformation between the degrees \(\mathbf{n}\) and the variables \(\mathbf{y}\) allowed to define a second set of \(s\) difference operators diagonalizing the same polynomials, again mutually commuting but this time acting by discrete shifts in the degrees \(n_i\). Both these sets of difference operators serve as generating sets for a commutative algebra. In this section, we will provide a larger algebraic framework for these algebras.

\subsection{Bispectrality of the multivariate $(-q)$-Racah polynomials}
\label{Paragraph Bispectrality}

Throughout this section we will take \(s\) to be the total number of variables, so we will consider \((-q)\)-Racah polynomials in \(s\) real variables \(x_i\), indexed by a tuple \(\mathbf{n} = (n_1,\dots,n_s)\) of \(s\) natural numbers, and difference operators in several of the variables \(x_i\) and the degrees \(n_i\).

The precise reparametrization that one needs to go from the multivariate \((-q)\)-Askey--Wilson polynomials \(P_s\left(\left.\mathbf{n}; \mathbf{y}, \boldsymbol{\alpha}\right| -q\right)\) in (\ref{multivariate AW def}) to the multivariate \((-q)\)-Racah polynomials \(R_{\mathbf{n}}\left(\left.\mathbf{x}; \mathbf{a},b,N\right|-q\right)\) of (\ref{q-Racah multivariate def}) is the following:
\begin{equation}
\label{Reparametrization AW to q-Racah}
\begin{alignedat}{5}
\alpha_0 & = \sqrt{-\frac{a_1}{qb}}, \quad & \alpha_i & = \prod_{k = 1}^{i}\sqrt{a_k}, \ \  i \in \{1,\dots,s+1\}, \quad & \alpha_{s+2}& = (-q)^{N}\prod_{k = 1}^{s+1}\sqrt{a_k}, \\
z_0& = \sqrt{-\frac{a_1}{qb}}, \quad& z_i &= (-q)^{x_i}\prod_{k = 1}^{i}\sqrt{a_k}, \ \ i \in \{1,\dots,s\}, \quad & z_{s+1}& = (-q)^{N}\prod_{k = 1}^{s+1}\sqrt{a_k}.
\end{alignedat}
\end{equation}
As observed in \cite[Remark 2.3]{Iliev-2011}, the \(q\)-difference operators for the Askey--Wilson polynomials translate to similar operators diagonalizing the multivariate \((-q)\)-Racah polynomials under the above reparametrization. For convenience of the reader, we repeat here the statements from \cite[Propositions 4.2, 4.5 and Theorem 5.5]{Iliev-2011}, translated to the setting of \((-q)\)-Racah polynomials by means of (\ref{Reparametrization AW to q-Racah}). Note that a \((-q)\)-shift \(z_j \to -qz_j\) corresponds to a discrete shift \(x_j \to x_j + 1\), as is immediate from (\ref{Reparametrization AW to q-Racah}). This suggests the notation \(T_{+,x_i}\) for the operator
\[
T_{+,x_i}f(x_1,\dots,x_s) = f(x_1,\dots,x_i+1,\dots,x_s).
\]
Let us first study the shift operators in the variables \(x_i\).
\begin{proposition}[\cite{Iliev-2011}]
	\label{prop Iliev x}
	Let us define the operator
		\begin{align*}
		& \mathcal{L}_j^{\mathbf{x}}(x_1,\dots,x_j;a_1,\dots,a_{j+1},b,x_{j+1}\vert -q)
		\\ = & \sum_{\nuu \in \{-1,0,1\}^j}C_{\nuu}(x_1,\dots,x_{j+1})T_{+,x_1}^{\nu_1}\dots T_{+,x_j}^{\nu_j} \\ &\phantom{=} - \left(1+\frac{bA_{j+1}}{a_1} - \tfrac{(-q)^{-x_{j+1}}}{1-q}\left(1-\tfrac{qb}{a_1}\right)\left(1+A_{j+1}(-q)^{2 x_{j+1}}\right) \right),
		\end{align*}
	referring to the notation (\ref{notation Nk, Ak}), with the convention that \(x_{s+1}=N\) and where the \(C_{\nuu}(x_1,\dots,x_{j+1})\) are functions in the variables \((-q)^{x_1},\dots,(-q)^{x_{j+1}}\), explicit expressions of which can be found in Appendix B. Then the operators
	\[
	\mathcal{L}_j^{\mathbf{x}}(x_1,\dots,x_j;a_1,\dots,a_{j+1},b,x_{j+1}\vert -q), \qquad j\in \{1,\dots,s\},
	\]
	form a set of \(s\) mutually commuting operators, with the multivariate \((-q)\)-Racah polynomials
	\[
	R_{\mathbf{n}}\left(\mathbf{x};\mathbf{a},b,N\vert -q\right) = R_{(n_1,\dots,n_s)}\left((x_1,\dots,x_s);a_1,\dots,a_{s+1},b,N\vert -q\right)
	\]
	as common eigenfunctions:
	\[
	\mathcal{L}_j^{\mathbf{x}}(x_1,\dots,x_j;a_1,\dots,a_{j+1},b,x_{j+1}\vert -q) R_{\mathbf{n}}\left(\mathbf{x};\mathbf{a},b,N\vert -q\right) = \mu_jR_{\mathbf{n}}\left(\mathbf{x};\mathbf{a},b,N\vert -q\right),
	\]
	with eigenvalue
	\[
	\mu_j = -\left(1-(-q)^{-N_j}\right)\left(1-\frac{bA_{j+1}}{a_1}(-q)^{N_{j}} \right).
	\]
\end{proposition}

In analogy to \cite{Iliev-2011}, we will denote by \(\mathcal{D}_{\mathbf{x},s}^{\mathbf{a},b,N}\) the algebra
\begin{equation}
\label{D_x,a def}
\mathcal{D}_{\mathbf{x},s}^{\mathbf{a},b,N} = \C\left((-q)^{x_1},\dots,(-q)^{x_s},\sqrt{a_1},\dots,\sqrt{a_{s+1}},\sqrt{b},(-q)^N\right)[T_{+,x_1}^{\pm1},\dots,T_{+,x_s}^{\pm 1}]
\end{equation}
of difference operators with coefficients rational in the variables \((-q)^{x_i}\) and the parameters \(\sqrt{a_i},\sqrt{b}\) and \((-q)^N\). The analog of Iliev's algebra \(\mathcal{A}_z\) in our \((-q)\)-Racah setting is the algebra
\begin{equation}
\label{A_z def}
\mathcal{A}_{\mathbf{x},s}(x_1,\dots,x_s) = \left\langle\mathcal{L}_j^{\mathbf{x}}(x_1,\dots,x_j;a_1,\dots,a_{j+1},b,x_{j+1}\vert -q): j\in \{1,\dots,s\}\right\rangle,
\end{equation}
considered as subalgebra of \(\mathcal{D}_{\mathbf{x},s}^{\mathbf{a},b,N}\). 
By Proposition \ref{prop Iliev x}, \(\mathcal{A}_{\mathbf{x},s}(x_1,\dots,x_s)\) is an abelian algebra of difference operators in the variables \(x_i\). We will relate this algebra to the higher rank \(q\)-Bannai--Ito algebra in Section \ref{Paragraph - Discrete realization}.

Like before, let us denote by \(T_{+,n_i}\) the forward shift operator in the discrete variables \(n_i\), i.e.
\[
T_{+,n_i}f(n_1,\dots,n_s) = f(n_1,\dots,n_i+1,\dots,n_s).
\]
Changing all \(x_i\) to \(n_i\) in (\ref{D_x,a def}), one obtains the algebra
\begin{equation}
\label{D_n,a def}
\mathcal{D}_{\mathbf{n},s}^{\mathbf{a},b,N} = \C\left((-q)^{n_1},\dots,(-q)^{n_s},\sqrt{a_1},\dots,\sqrt{a_{s+1}},\sqrt{b},(-q)^N\right)[T_{+,n_1}^{\pm1},\dots,T_{+,n_s}^{\pm 1}].
\end{equation}
Let us now introduce a transformation \(\mathfrak{b}\) acting as 
\begin{align*}
\bb((-q)^{x_j}) & = \frac{b A_{s-j+2}}{a_1A_{s+1}}(-q)^{N_{s-j+1}-N}, && \\
\bb(a_1) & = \frac{a_1}{b}A_{s+1}(-q)^{2N}, &\bb(a_j) & = a_{s-j+3}, \qquad\qquad\qquad\, j\in [2;s+1], \\
\bb(T_{+,x_1}) & = T_{+,n_s}, & \bb(T_{+,x_j}) & = T_{+,n_{s-j+1}}T_{+,n_{s-j+2}}^{-1},\quad j\in[2;s], \\
\bb((-q)^N) & = \frac{b}{A_{s+1}}(-q)^{-N}, & \bb(b) & = A_{s+1}(-q)^{2N},
\end{align*}
and extended by linearity and multiplicativity to an isomorphism from \(\mathcal{D}_{\mathbf{x},s}^{\mathbf{a},b,N}\) to \(\mathcal{D}_{\mathbf{n},s}^{\mathbf{a},b,N}\). 
Define also
\begin{align*}
g_{\mathbf{n}} = g_{\mathbf{n},\mathbf{a},b,N} & = \frac{\left(A_{s+1}(-q)^N\right)^{N_s}}{\left(A_{s+1}(-q)^N, \frac{b}{a_1}A_{s+1}(-q)^{N+1}; -q\right)_{N_s}}\left(\prod_{j = 1}^s\frac{\left(A_j\right)^{\frac{-n_j}{2}}}{\left(a_{j+1};-q\right)_{n_j}}\right),
\end{align*}
where we have again used the \((-q)\)-Pochhammer symbol (\ref{q-Pochhammer symbol}). The analogs of Iliev's renormalized Askey--Wilson polynomials, see \cite[(5.4)]{Iliev-2011}, are the renormalized \((-q)\)-Racah polynomials
\begin{equation}
\label{renormalized multivariate -q-Racah def}
\widehat{R_{\mathbf{n}}}(\mathbf{x};\mathbf{a},b,N\vert -q) = g_{\mathbf{n},\mathbf{a},b,N}R_{\mathbf{n}}(\mathbf{x};\mathbf{a},b,N\vert -q).
\end{equation}
Iliev's second set of difference operators is described in the following proposition.

\begin{proposition}[\cite{Iliev-2011}]
	\label{prop Iliev n}
	Let us define the operator
	\begin{align*}
	& \mathfrak{L}_j^{\mathbf{n}}\left(n_1,\dots,n_s; a_1,\dots,a_{s+1},b,N\vert -q\right) = \bb(\mathcal{L}_j^{\mathbf{x}}\left(x_1,\dots,x_j;a_1,\dots,a_{j+1},b,x_{j+1}\vert -q)\right) \\
	= & \sum_{\nuu\in \{-1, 0, 1\}^j}D_{\nuu}(n_1,\dots,n_s)T_{+,n_{s-j+1}}^{\nu_j}T_{+,n_{s-j+2}}^{\nu_{j-1}-\nu_j}\dots T_{+,n_s}^{\nu_1-\nu_2}\\&-\left(1+\frac{A_{s+1}^2}{A_{s-j+1}}(-q)^{2N}-\frac{A_{s+1}(-q)^{N}}{1-q}\left(1-\frac{qb}{a_1}\right)\left(\frac{a_1 }{bA_{s-j+1}}(-q)^{-N_{s-j}}+(-q)^{N_{s-j}}\right)\right),
	\end{align*}
	where the \(D_{\nuu}(n_1,\dots,n_s) = \bb(C_{\nuu}(x_1,\dots,x_{j+1}))\) are functions in the variables \((-q)^{n_1}\), \dots, \((-q)^{n_s}\), explicit expressions of which can be found in Appendix B. Then the operators
	\[
	\mathfrak{L}_j^{\mathbf{n}}\left(n_1,\dots,n_s; a_1,\dots,a_{s+1},b,N\vert -q\right), \qquad j\in\{1,\dots,s\},
	\] 
	are mutually commuting operators with the renormalized \((-q)\)-Racah polynomials as common eigenfunctions:
	\[
	\mathfrak{L}_j^{\mathbf{n}}\left(n_1,\dots,n_s; a_1,\dots,a_{s+1},b,N\vert -q\right)\widehat{R_{\mathbf{n}}}(\mathbf{x};\mathbf{a},b,N\vert -q) = \kappa_j\widehat{R_{\mathbf{n}}}(\mathbf{x};\mathbf{a},b,N\vert -q),
	\]
	or equivalently
	\[
	\left(g_{\mathbf{n}}^{-1}\mathfrak{L}_j^{\mathbf{n}}\left(n_1,\dots,n_s; a_1,\dots,a_{s+1},b,N\vert -q\right)g_{\mathbf{n}}\right)R_{\mathbf{n}}\left(\mathbf{x};\mathbf{a},b,N\vert -q\right) = \kappa_jR_{\mathbf{n}}\left(\mathbf{x};\mathbf{a},b,N\vert -q\right),
	\]
	with eigenvalue
	\[
	\kappa_j = -1 - \frac{A_{s+1}^2}{A_{s-j+1}}(-q)^{2N} + A_{s+1}(-q)^N\left(\frac{(-q)^{-x_{s-j+1}}}{A_{s-j+1}}+(-q)^{x_{s-j+1}}\right).
	\]
\end{proposition}

The role of Iliev's second algebra \(\mathcal{A}_n\) is now played by
\begin{equation}
\label{A_n def}
\mathcal{A}_{\mathbf{n},s}(n_1,\dots,n_s) = \left\langle\mathfrak{L}_j^{\mathbf{n}}(n_1,\dots,n_s,;a_1,\dots,a_{s+1},b,N\vert -q): j\in \{1,\dots,s\} \right\rangle,
\end{equation}
considered as subalgebra of \(\mathcal{D}_{\mathbf{n},s}^{\mathbf{a},b,N}\). 
By Proposition \ref{prop Iliev n}, \(\mathcal{A}_{\mathbf{n},s}(n_1,\dots,n_s)\) is an abelian algebra of difference operators in the variables \(n_i\).

\subsection{Discrete realization of the higher rank $q$-Bannai--Ito algebra}
\label{Paragraph - Discrete realization}
For ease of notation, let us from now on write \(\vert\mathbf{s}\rangle = \vert s_{n-1},s_{n-2},\dots,s_1\rangle\) for the vectors
\begin{equation}
\label{vector s}
\vert \mathbf{s}\rangle = \vert \mathbf{j}^{(1,1)}\rangle
\end{equation}
diagonalizing the subalgebra 
\[
\mathcal{Y}_{i,1} = \mathcal{Y}_{1,1} = \langle \Gamma_{[1;2]}^q,\Gamma_{[1;3]}^q,\dots,\Gamma_{[1;n]}^q\rangle,
\]
with \(i\in [1;n-1]\) arbitrary, where we have used the equality from Remark \ref{Remark Y_i,r}. Similarly, we write \(\vert\mathbf{j}\rangle = \vert j_{n-1}, j_{n-2},\dots, j_1\rangle\) for the vectors
\begin{equation}
\label{vector j}
\vert\mathbf{j}\rangle = \vert \mathbf{j}^{(1,n-1)}\rangle
\end{equation}
that diagonalize
\[
\mathcal{Y}_{1,n-1} = \langle \Gamma_{[2;3]}^q, \Gamma_{[2;4]}^q,\dots,\Gamma_{[2;n]}^q,\Gamma_{[1;n]}^q\rangle.
\]

We will now use the results from Section \ref{Paragraph Bispectrality}, for a total number of variables \(s = n-2\) and with certain substitutions for the parameters \(a_i\) and \(b\) and renamings for the variables \(x_i\) and the degrees \(n_i\). These will lead us to a realization of the higher rank \(q\)-Bannai--Ito algebra \(\mathcal{A}_n^q\), or more precisely of the algebra abstractly defined by the relations of Proposition \ref{prop q-anticommutation Gamma}. At this time, it is not clear whether the two are isomorphic, we refer to \cite{DeClercq-2019} for a discussion. Nevertheless, we will take the freedom to call the operator system we are about to construct a realization of the algebra \(\mathcal{A}_n^q\). We will call this realization discrete, as it represents the algebra generators \(\Gamma_A^q\) as shift operators in \(2n-2\) discrete variables \(j_1,\dots,j_{n-1},s_1,\dots,s_{n-1}\in\N\), i.e.\ the quantum numbers labelling the vectors \(\vert\mathbf{j}\rangle\) and \(\vert\mathbf{s}\rangle\). We will first introduce two function spaces, one of which will be our representation space.

\begin{definition}
\label{representation spaces Vq def}
We denote by \(\widetilde{\mathcal{V}}_q\) the infinite-dimensional vector space over \(\mathbb{C}\) spanned by all overlap coefficients
\[
\langle \mathbf{j}\vert\mathbf{s}\rangle, \quad \mathbf{j}\in\N^{n-1},
\]
considered as functions of \(s_1,\dots,s_{n-1}\in\N\). We will write \(\mathcal{V}_q\) for the infinite-dimensional vector space spanned by the renormalized multivariate \((-q)\)-Racah polynomials
\[
\widehat{R}_{(j_1,\dots,j_{n-2})}\left((s_1,s_1+s_2,\dots,\vss{n-2}); \mathbf{a}, b, \vss{n-1}\left|\right. -q\right), \quad \mathbf{j}\in\N^{n-1},
\]
again considered as functions of \(\mathbf{s}\in\N^{n-1}\), where \(j_{n-1}\) is fixed by the constraint \(\vert\mathbf{j}_{n-1}\vert = \vert\mathbf{s}_{n-1}\vert\) and with
\begin{equation}
\label{parametrization overlap j and s}
\mathbf{a} = (a_1,\dots,a_{n-1}), \quad a_1 = -q^{2(\gamma_1+\gamma_2)-1}, \quad a_k = q^{2\gamma_{k+1}}, k\in [2;n-1], \quad b = -q^{2\gamma_2-1}.
\end{equation}
\end{definition}

It will suffice to give an explicit realization for each of the elements \(\Gamma_{[i+1;i+r]}^q\), \(i\in[0;n-1], r\in[1;n-i]\), agreeing with the algebra relations in Propositions \ref{prop q-anticommutation Gamma} and \ref{prop A contained in B}, since these generate the full algebra \(\mathcal{A}_n^q\) by Corollary \ref{cor Generating set}. We will use the approach from \cite{DeBie&vandeVijver-2018} to find appropriate realizations. In order to pass from the abstract algebraic perspective to the setting of operators acting on functions of discrete variables, we will lift the action of the algebra to the connection coefficients. More precisely, in a first step we will look for suitable operators \(\widetilde{\Gamma_A^q}\) which act on the connection coefficients as follows:
\begin{equation}
\label{lift the action}
\widetilde{\Gamma_A^q}\langle \mathbf{j}\vert\mathbf{s}\rangle = \langle \mathbf{j}\vert \Gamma_A^q\vert \mathbf{s}\rangle.
\end{equation}
Afterwards, we will slightly modify the considered \(\widetilde{\Gamma_A^q}\) to obtain more elegant expressions. This will be our strategy of proof in the following theorem. 

\begin{theorem}
	\label{th discrete realization}
	Let \(\gamma_i > \frac12\), \(i\in[1;n]\), be a set of real parameters. Let us define
	\begin{align}
	\label{Gamma start from 1}
	\begin{split}
	\Gamma_{[1;m+1]}^q = &\, -\frac{(-q)^{-\vss{n-1}}q^{-\gamma_{[1;n]}-\gamma_{[m+2;n]}+\frac12}}{q-q^{-1}}\mathfrak{L}_{n-m-1}^{\mathbf{n}}\\& + (-1)^{\vss{n-1}}\left[\vss{n-1}+\gamma_{[1;n]}+\gamma_{[m+2;n]}-\frac12\right]_q,
	\end{split}
	\end{align}
	for \(m\in [0;n-1]\) and
	\begin{align}
	\label{Gamma start from i}
	\Gamma_{[i+1;i+r]}^q & = -\frac{q^{-\gamma_{[i+1;i+r]}+\frac12}}{q-q^{-1}}\mathcal{L}_{i,r-1}^{\mathbf{x}} + \left[\gamma_{[i+1;i+r]}-\frac12\right]_q,
	\end{align}
	for \(i\in [1;n-1], r\in[1;n-i]\), where we have written \(\mathfrak{L}_{n-m-1}^{\mathbf{n}}\) for 
	\[
	\mathfrak{L}_{n-m-1}^{\mathbf{n}}\left(j_1,\dots,j_{n-2}; -q^{2(\gamma_1+\gamma_2)-1},q^{2\gamma_3},\dots,q^{2\gamma_n},-q^{2\gamma_2-1},\vjj{n-1}\right)
	\]
	and \(\mathcal{L}_{i,r-1}^{\mathbf{x}}\) for 
	\[
	\mathcal{L}_{r-1}^{\mathbf{x}}\left(s_i,\dots,\vss{i+r-2}-\vss{i-1};(-q)^{2\vss{i-1}-1}q^{2\gamma_{[1;i+1]}}, q^{2\gamma_{i+2}},\dots,q^{2\gamma_{i+r}}, -q^{2\gamma_{i+1}-1},\vss{i+r-1}-\vss{i-1}\right),
	\]
	with the convention that \(\mathcal{L}_{i,0}^{\mathbf{x}} = \mathfrak{L}_0^{\mathbf{n}} = 0\) and that
	\[
	\mathfrak{L}_{n-1}^{\mathbf{n}} = (q-q^{-1})q^{\vss{n-1}+\gamma_1+2\gamma_{[2;n]}-\frac12}\left(\left[\vss{n-1}+\gamma_1+2\gamma_{[2;n]}-\frac12\right]_q-(-1)^{\vss{n-1}}\left[\gamma_1-\frac12\right]_q\right).
	\]
	The algebra generated by the operators (\ref{Gamma start from 1}) and (\ref{Gamma start from i}) forms a discrete realization of the \(q\)-Bannai--Ito algebra \(\mathcal{A}_n^q\) of rank \(n-2\) on the module \(\mathcal{V}_q\).
\end{theorem}
\begin{proof}
	Recall from Theorem \ref{th - Overlap q-case}, with \(i = 1\) and  \(r = n-1\), that
	\begin{align}
	\label{overlap with j and s}
	\begin{split}
	& \langle\mathbf{j}\vert\mathbf{s}\rangle = \langle j_{n-1},\dots,j_1\vert s_{n-1},\dots,s_1\rangle  \\= &\  \sqrt{\frac{\rho^{(1,n-1)}(\mathbf{s})}{h_{\mathbf{j}}}} R_{(j_1,\dots,j_{n-2})}\left((s_1,s_1+s_2,\dots,\vss{n-2}); \mathbf{a}, b, \vss{n-1}\left|\right. -q\right),
	\end{split}
	\end{align}
	where the parameters \(\mathbf{a}\) and \(b\) are given by (\ref{parametrization overlap j and s}). Here we assume that \(\vjj{n-1} = \vss{n-1}\),  otherwise \(\langle\mathbf{j}\vert\mathbf{s}\rangle = 0\).
	
	Our first objective will be to find suitable operators \(\widetilde{\Gamma_A^q}\) subject to (\ref{lift the action}). If the proposed operators satisfy this requirement, then they will automatically comply with the algebra relations and thus form a proper realization of our algebra on the space \(\widetilde{\mathcal{V}}_q\). 
	Indeed, referring to the notation from Definition \ref{definition bases final}, Proposition \ref{prop decomposition} asserts that for any set \(B\), \(\Gamma_B^q\vert\mathbf{s}\rangle\) can be written as a linear combination of vectors \(\vert\mathbf{s}';m'\rangle\) with \(\mathbf{s}'\in\N^{n-1}\) and \(m'\in\N\). Moreover \(\Gamma_B^q\) commutes with both \(K_{[1;n]}\) and \(\Gamma_{[1;n]}^q\), such that the linear combination will only contain terms with \(m'= m\) and \(\vert\mathbf{s}'_{n-1}\vert = \vss{n-1}\). As a consequence, if we assume that (\ref{lift the action}) holds for any set \(A\), then we also find
	\[
	\widetilde{\Gamma_A^q}\widetilde{\Gamma_B^q}\langle\mathbf{j}\vert\mathbf{s}\rangle = \langle\mathbf{j}\vert\Gamma_A^q\Gamma_B^q\vert\mathbf{s}\rangle,
	\]
	for all sets \(A\) and \(B\). It follows that the \(\widetilde{\Gamma_A^q}\) will satisfy the algebra relations from Propositions \ref{prop q-anticommutation Gamma} and \ref{prop A contained in B}. It will suffice to fix an operator realization for each of the generators \(\Gamma_A^q\) with \(A\) a set of consecutive numbers. By Corollary \ref{cor Generating set}, this uniquely determines the corresponding expressions for each of the remaining generators.
	
	In case \(A\) is a singleton, (\ref{lift the action}) is met by (\ref{Remark singletons}). Hence we may restrict ourselves to sets of consecutive numbers with at least two elements. We will subdivide those sets into three different classes.
	
	\noindent\textbf{Case 1:} \(A = {[1;m+1]}\) with \(m\in [1;n-1]\)
	
	\noindent The eigenvalue equations (\ref{eigenvalue equation 1}) for \(\vert \mathbf{s}\rangle \) assert that
	\begin{equation}
	\label{Gamma starts with 1 assertion}
	\langle \mathbf{j}\vert \Gamma_{[1;m+1]}^q\vert \mathbf{s}\rangle = (-1)^{\vss{m}}\left[\vss{m}+\gamma_{[1;m+1]}-\frac12\right]_q\langle \mathbf{j}\vert \mathbf{s}\rangle.
	\end{equation}
	On the other hand, one can observe from Proposition \ref{prop Iliev n} and (\ref{overlap with j and s}) that
	\begin{align*}
	& \left( h_{\mathbf{j}}^{-\frac12}g_{\mathbf{j}}^{-1}\mathfrak{L}_{n-m-1}^{\mathbf{n}} g_{\mathbf{j}}h_{\mathbf{j}}^{\frac12}\right)\langle \mathbf{j}\vert \mathbf{s}\rangle \\
	= &\  \sqrt{\frac{\rho^{(1,n-1)}(\mathbf{s})}{h_{\mathbf{j}}}}\left(g_{\mathbf{j}}^{-1}\mathfrak{L}_{n-m-1}^{\mathbf{n}} g_{\mathbf{j}}\right)R_{(j_1,\dots,j_{n-2})}\left(\left.(s_1,s_1+s_2,\dots,\vss{n-2}); \mathbf{a},b,\vss{n-1}\right| -q\right) \\
	= & \ \kappa_{n-m-1}\langle \mathbf{j}\vert \mathbf{s}\rangle,
	\end{align*}
	where the eigenvalue is given by
	\begin{align*}
	\kappa_{n-m-1} = &\  -(q-q^{-1})(-q)^{\vss{n-1}}q^{\gamma_{[1;n]}+\gamma_{[m+2;n]}-\frac12}\\\times&\left((-1)^{\vss{n-1}+1}\left[\vss{n-1}+\gamma_{[1;n]}+\gamma_{[m+2;n]}-\frac12\right]_q + (-1)^{\vss{m}}\left[\vss{m}+\gamma_{[1;m+1]}-\frac12\right]_q\right).
	\end{align*}
	This suggests to define
	\begin{align}
	\label{Gamma tilde case 1}
	\begin{split}
	\widetilde{\Gamma_{[1;m+1]}^q} = & -\frac{(-q)^{-\vss{n-1}}q^{-\gamma_{[1;n]}-\gamma_{[m+2;n]}+\frac12}}{q-q^{-1}}\left(h_{\mathbf{j}}^{-\frac12}g_{\mathbf{j}}^{-1}\mathfrak{L}_{n-m-1}^{\mathbf{n}} g_{\mathbf{j}}h_{\mathbf{j}}^{\frac12}\right) \\  & + (-1)^{\vss{n-1}}\left[\vss{n-1}+\gamma_{[1;n]}+\gamma_{[m+2;n]}-\frac12\right]_q\\
	= &\, \left(g_{\mathbf{j}}h_{\mathbf{j}}^{\frac12}\right)^{-1}\Gamma_{[1;m+1]}^q\left(g_{\mathbf{j}}h_{\mathbf{j}}^{\frac12}\right),
	\end{split}
	\end{align}
	referring to the notation (\ref{Gamma start from 1}), as this leads to
	\begin{align*}
	\widetilde{\Gamma_{[1;m+1]}^q}\langle \mathbf{j}\vert \mathbf{s}\rangle & = (-1)^{\vss{m}}\left[\vss{m}+\gamma_{[1;m+1]}-\frac12\right]_q\langle \mathbf{j}\vert \mathbf{s}\rangle,
	\end{align*}
	in agreement with (\ref{Gamma starts with 1 assertion}). We will correct for the factors \(g_{\mathbf{j}}\) and \(h_{\mathbf{j}}\) in the final step.
	
	\noindent\textbf{Case 2:} \(A = {[2;m+1]}\) with \(m\in [2;n-1]\)
	
	\noindent The eigenvalue equations (\ref{eigenvalue equation 2}) for \(\vert \mathbf{j}\rangle \) assert that
	\begin{equation}
	\label{Gamma starts with 2 assertion}
	\langle \mathbf{j}\vert \Gamma_{[2;m+1]}^q\vert \mathbf{s}\rangle = (-1)^{\vjj{m-1}}\left[\vjj{m-1}+\gamma_{[2;m+1]}-\frac12\right]_q\langle \mathbf{j}\vert \mathbf{s}\rangle.
	\end{equation}
	On the other hand, Proposition \ref{prop Iliev x} and (\ref{overlap with j and s}) imply that
	\begin{align*}
	& \left(\left(\rho^{(1,n-1)}(\mathbf{s})\right)^{\frac12}\mathcal{L}_{1,m-1}^{\mathbf{x}}\left(\rho^{(1,n-1)}(\mathbf{s})\right)^{-\frac12}\right)\langle \mathbf{j}\vert \mathbf{s}\rangle \\
	= &\  \sqrt{\frac{\rho^{(1,n-1)}(\mathbf{s})}{h_{\mathbf{j}}}}\mathcal{L}_{1,m-1}^{\mathbf{x}}R_{(j_1,\dots,j_{n-2})}\left((s_1,s_1+s_2,\dots,\vss{n-2});\mathbf{a},b,\vss{n-1}\vert -q\right) \\
	= & \ \mu_{m-1}\langle \mathbf{j}\vert \mathbf{s}\rangle,
	\end{align*}
	with eigenvalue
	\begin{align*}
	\mu_{m-1} = -(q-q^{-1})q^{\gamma_{[2;m+1]}-\frac12}\left((-1)^{\vjj{m-1}}\left[\vjj{m-1}+\gamma_{[2;m+1]}-\frac12\right]_q-\left[\gamma_{[2;m+1]}-\frac12\right]_q\right).
	\end{align*}
	This suggests the definition
	\begin{align}
	\label{Gamma tilde case 2}
	\begin{split}
	\widetilde{\Gamma_{[2;m+1]}^q} & = -\frac{q^{-\gamma_{[2;m+1]}+\frac12}}{q-q^{-1}}\left(\left(\rho^{(1,n-1)}(\mathbf{s})\right)^{\frac12}\mathcal{L}_{1,m-1}^{\mathbf{x}}\left(\rho^{(1,n-1)}(\mathbf{s})\right)^{-\frac12}\right) + \left[\gamma_{[2;m+1]}-\frac12\right]_q \\
	& = \left(\rho^{(1,n-1)}(\mathbf{s})\right)^{\frac12}\Gamma_{[2;m+1]}^q\left(\rho^{(1,n-1)}(\mathbf{s})\right)^{-\frac12},
	\end{split}
	\end{align}
	referring to the notation (\ref{Gamma start from i}), as this leads to
	\begin{align*}
	\widetilde{\Gamma_{[2;m+1]}^q} \langle \mathbf{j}\vert \mathbf{s}\rangle & = (-1)^{\vjj{m-1}}\left[\vjj{m-1}+\gamma_{[2;m+1]}-\frac12\right]_q\langle \mathbf{j}\vert \mathbf{s}\rangle
	\end{align*}
	agreeing with (\ref{Gamma starts with 2 assertion}). We will correct for the factor \(\rho^{(1,n-1)}(\mathbf{s})\) in the final step. 
	
	\noindent \textbf{Case 3:} \(A = [i+1;i+r]\) with\, \(i \in [2;n-2], r\in [2;n-i]\)
	
	\noindent The operator \(\Gamma_{[i+1;i+r]}^q\) is not diagonalized by \(\vert \mathbf{j}\rangle\) or \(\vert \mathbf{s}\rangle\), but writing
	\begin{equation}
	\label{resolution of the identity}
	\vert \mathbf{s}\rangle = \sum_{\mathbf{j}^{(i,r)}} \vert \mathbf{j}^{(i,r)}\rangle\langle \mathbf{j}^{(i,r)}\vert \mathbf{s}\rangle,
	\end{equation}
	where the sum is over all vectors \(\mathbf{j}^{(i,r)}\in\N^{n-1}\), and using the eigenvalue equations (\ref{eigenvalue equation 2}) for \(\vert\mathbf{j}^{(i,r)}\rangle\), we find
	\begin{align}
	\label{Gamma starts with i assertion}
	\langle \mathbf{j}\vert \Gamma_{[i+1;i+r]}^q\vert \mathbf{s}\rangle
	= \sum_{\mathbf{j}^{(i,r)}} (-1)^{\vj{i,r}{i+r-2}-\vj{i,r}{i-1}}\left[\vj{i,r}{i+r-2}-\vj{i,r}{i-1}+\gamma_{[i+1;i+r]}-\frac12\right]_q \langle \mathbf{j}\vert \mathbf{j}^{(i,r)}\rangle \langle \mathbf{j}^{(i,r)}\vert \mathbf{s} \rangle.
	\end{align}
	Recall from Theorem \ref{th - Overlap q-case} that
	\begin{align*}
	\langle \mathbf{j}^{(i,r)}\vert \mathbf{s}\rangle =  \sqrt{\frac{\rho^{(i,r)}(\mathbf{s})}{h_{\mathbf{j}^{(i,r)}}}} R_{(j_i^{(i,r)},\dots,j_{i+r-2}^{(i,r)})}\left(\left.(s_i,\dots,\vss{i+r-2}-\vss{i-1}); \mathbf{a}, b, \vss{i+r-1}-\vss{i-1}\right| -q\right),
	\end{align*}
	with
	\[
	\mathbf{a} = (a_1,\dots,a_{r}), \quad a_1 = (-q)^{2\vss{i-1}-1}q^{2\gamma_{[1;i+1]}}, \qquad a_k = q^{2\gamma_{i+k}},k\in[2;r], \qquad b =  -q^{2\gamma_{i+1}-1},
	\]
	in the assumption that
	\[
	j_{\ell}^{(i,r)} = s_{\ell}, \quad \ell\in[1;i-1]\cup[i+r;n-1], \qquad \vj{i,r}{i+r-1} = \vss{i+r-1},
	\]
	otherwise \(\langle \mathbf{j}^{(i,r)}\vert \mathbf{s}\rangle = 0\). As a consequence of Proposition \ref{prop Iliev x} we thus have
	\begin{align*}
	& \left(\left(\rho^{(i,r)}(\mathbf{s})\right)^{\frac12}\mathcal{L}_{i,r-1}^{\mathbf{x}}\left(\rho^{(i,r)}(\mathbf{s})\right)^{-\frac12}\right)\langle \mathbf{j}^{(i,r)}\vert \mathbf{s}\rangle \\
	= &\  \sqrt{\frac{\rho^{(i,r)}(\mathbf{s})}{h_{\mathbf{j}^{(i,r)}}}}\mathcal{L}_{i,r-1}^{\mathbf{x}}R_{(j_i^{(i,r)},\dots,j_{i+r-2}^{(i,r)})}\left(\left.(s_i,\dots,\vss{i+r-2}-\vss{i-1}); \mathbf{a}, b, \vss{i+r-1}-\vss{i-1}\right| -q\right) \\
	= & \ \mu_{i,r-1}\langle \mathbf{j}^{(i,r)}\vert \mathbf{s}\rangle,
	\end{align*}
	where the eigenvalue is given by
	\begin{align*}
	\mu_{i,r-1} = & -(q-q^{-1})q^{\gamma_{[i+1;i+r]}-\frac12}\\\times &\left((-1)^{\vj{i,r}{i+r-2}-\vj{i,r}{i-1}}\left[\vj{i,r}{i+r-2}-\vj{i,r}{i-1}+\gamma_{[i+1;i+r]}-\frac12\right]_q-\left[\gamma_{[i+1;i+r]}-\frac12\right]_q\right).
	\end{align*}
	This suggests the definition
	\begin{align*}
	\widetilde{\Gamma_{[i+1;i+r]}^q} & = -\frac{q^{-\gamma_{[i+1;i+r]}+\frac12}}{q-q^{-1}}\left(\left(\rho^{(i,r)}(\mathbf{s})\right)^{\frac12}\mathcal{L}_{i,r-1}^{\mathbf{x}}\left(\rho^{(i,r)}(\mathbf{s})\right)^{-\frac12}\right)+\left[\gamma_{[i+1;i+r]}-\frac12\right]_q \\
	& = \left(\rho^{(i,r)}(\mathbf{s})\right)^{\frac12}\Gamma_{[i+1;i+r]}^q\left(\rho^{(i,r)}(\mathbf{s})\right)^{-\frac12},
	\end{align*}
	referring to the notation (\ref{Gamma start from i}), as this leads to
	\begin{align*}
	\widetilde{\Gamma_{[i+1;i+r]}^q}\langle \mathbf{j}^{(i,r)}\vert \mathbf{s}\rangle =  (-1)^{\vj{i,r}{i+r-2}-\vj{i,r}{i-1}}\left[\vj{i,r}{i+r-2}-\vj{i,r}{i-1}+\gamma_{[i+1;i+r]}-\frac12\right]_q \langle \mathbf{j}^{(i,r)}\vert \mathbf{s}\rangle.
	\end{align*}
	Since \(\widetilde{\Gamma_{[i+1;i+r]}^q}\) only acts on the variables \(s_i,\dots,s_{i+r-2}\), we may use (\ref{resolution of the identity}) to write
	\begin{align*}
	& \ \widetilde{\Gamma_{[i+1;i+r]}^q}\langle \mathbf{j}\vert \mathbf{s}\rangle \\
	= & \sum_{\mathbf{j}^{(i,r)}} \langle \mathbf{j}\vert \mathbf{j}^{(i,r)}\rangle \widetilde{\Gamma_{[i+1;i+r]}^q}\langle \mathbf{j}^{(i,r)}\vert \mathbf{s}\rangle \\
	= &  \sum_{\mathbf{j}^{(i,r)}} (-1)^{\vj{i,r}{i+r-2}-\vj{i,r}{i-1}}\left[\vj{i,r}{i+r-2}-\vj{i,r}{i-1}+\gamma_{[i+1;i+r]}-\frac12\right]_q \langle \mathbf{j}\vert \mathbf{j}^{(i,r)}\rangle \langle \mathbf{j}^{(i,r)}\vert \mathbf{s} \rangle,
	\end{align*}
	in agreement with (\ref{Gamma starts with i assertion}).

	Each of the obtained operators \(\widetilde{\Gamma_A^q}\) satisfies the requirement (\ref{lift the action}), hence they form a representation of \(\mathcal{A}_n^q\) on the space \(\widetilde{\mathcal{V}}_q\) of overlap coefficients \(\langle \mathbf{j}\vert \mathbf{s} \rangle\). As a final step, we will transform these operators into the anticipated \(\Gamma_A^q\) in (\ref{Gamma start from 1}) and (\ref{Gamma start from i}), thereby changing the representation space to \(\mathcal{V}_q\), but without altering the algebra relations. 
	
	First observe that each of the factors
	\[
	\frac{\rho^{(1,n-1)}(\mathbf{s})}{\rho^{(i,r)}(\mathbf{s})},
	\]
	explicit expressions of which are given in Appendix C, is independent of the variables \(\vss{i},\dots,\vss{i+r-2}\). As a consequence, they commute with the difference operators \(\mathcal{L}_{i,r-1}^{\mathbf{x}}\) acting on precisely these variables, such that one can write
	\begin{align}
	\label{Gamma tilde case 3}
	\begin{split}
	\widetilde{\Gamma_{[i+1;i+r]}^q} & = \left(\frac{\rho^{(1,n-1)}(\mathbf{s})}{\rho^{(i,r)}(\mathbf{s})}\right)^{\frac12}\widetilde{\Gamma_{[i+1;i+r]}^q} \left(\frac{\rho^{(1,n-1)}(\mathbf{s})}{\rho^{(i,r)}(\mathbf{s})}\right)^{-\frac12} \\ & = -\frac{q^{-\gamma_{[i+1;i+r]}+\frac12}}{q-q^{-1}}\left(\rho^{(1,n-1)}(\mathbf{s})\right)^{\frac12}\mathcal{L}_{i,r-1}^{\mathbf{x}}\left(\rho^{(1,n-1)}(\mathbf{s})\right)^{-\frac12} +\left[\gamma_{[i+1;i+r]}-\frac12\right]_q \\
	& = \left(\rho^{(1,n-1)}(\mathbf{s})\right)^{\frac12}\Gamma_{[i+1;i+r]}^q\left(\rho^{(1,n-1)}(\mathbf{s})\right)^{-\frac12}.
	\end{split}
	\end{align}
	
	Finally, observe that \(\widetilde{\Gamma_{[1;m+1]}^q}\) commutes with functions of \(\mathbf{s}\), and is hence invariant under conjugation with such functions. The analogous statement holds for \(\widetilde{\Gamma_{[i+1;i+r]}^q}\), \(i\geq 1\), and functions of \(\mathbf{j}\).	Hence it follows from (\ref{Gamma tilde case 1}), (\ref{Gamma tilde case 2}) and (\ref{Gamma tilde case 3}) that conjugation with the function
	\[
	\frac{g_{\mathbf{j}}h_{\mathbf{j}}^{\frac12}}{\left(\rho^{(1,n-1)}(\mathbf{s})\right)^{\frac12}}
	\]
	acts on each of the generators as
	\[
	\widetilde{\Gamma_A^q} \to \Gamma_A^q.
	\]
	Such a conjugation leaves the algebra relations invariant, but requires to take as a representation space the space of all functions
	\[
	\frac{g_{\mathbf{j}}h_{\mathbf{j}}^{\frac12}}{\left(\rho^{(1,n-1)}(\mathbf{s})\right)^{\frac12}}\langle \mathbf{j}\vert \mathbf{s} \rangle = \widehat{R}_{(j_1,\dots,j_{n-2})}\left((s_1,s_1+s_2,\dots,\vss{n-2}); \mathbf{a}, b, \vss{n-1}\left|\right. -q\right),
	\]
	with the parametrization (\ref{parametrization overlap j and s}), i.e.\ the space \(\mathcal{V}_q\). This concludes the proof.
\end{proof}

Let us now explain the significance of the previous theorem in relation to Iliev's work. In \cite{Iliev-2011} two commutative algebras of difference operators were defined, denoted here by \(\mathcal{A}_{\mathbf{x},s}(x_1,\dots,x_s)\), see (\ref{A_z def}), and \(\mathcal{A}_{\mathbf{n},s}(n_1,\dots,n_s)\), see (\ref{A_n def}). These algebras are diagonalized by the renormalized multivariate \((-q)\)-Racah polynomials (\ref{renormalized multivariate -q-Racah def}). Restricting their action to the space \(\mathcal{V}_q\), the algebraic relations between these two algebras are encoded in Theorem \ref{th discrete realization}. In fact they are observed to be embedded in a larger algebra, namely the rank \(n-2\) \(q\)-Bannai--Ito algebra \(\mathcal{A}_n^q\). In the realization of Theorem \ref{th discrete realization}, one observes that
\begin{align*}
\langle \Gamma_{[1;m+1]}^q: m\in [0;n-1]\rangle = \mathcal{A}_{\mathbf{n},n-2}(j_1,j_2,\dots,j_{n-2})
\end{align*}
and
\begin{align*}
\langle \Gamma_{[i+1; i+k]}^q: k \in [1;n-i]\rangle = \mathcal{A}_{\mathbf{x},n-i-1}(s_i,s_i+s_{i+1},\dots,\vss{n-2}-\vss{i-1}),
\end{align*}
for all \(i\in [1;n-2]\).
Combining this observation with Corollary \ref{cor Generating set}, we conclude that several such difference operator algebras are enough to generate the rank \(n-2\) \(q\)-Bannai--Ito algebra, namely
\begin{align*}
\mathcal{A}_{\mathbf{n},n-2}(j_1,\dots,j_{n-2}), \mathcal{A}_{\mathbf{x},n-2}(s_1,\dots,\vss{n-2}), \mathcal{A}_{\mathbf{x},n-3}(s_2,\dots,\vss{n-2}-s_1), \dots, \mathcal{A}_{\mathbf{x},1}(s_{n-2})
\end{align*} 
together generate the whole \(\mathcal{A}_n^q\). 

\begin{remark}
	Another discrete realization could be obtained upon replacing (\ref{Gamma start from 1}) by
	\[
	\Gamma_{[1;m+1]}^q = (-1)^{\vss{m}}\left[\vss{m}+\gamma_{[1;m+1]}-\frac12\right]_q,
	\]
	in analogy with \cite{DeBie&vandeVijver-2018}, which by (\ref{Gamma starts with 1 assertion}) immediately agrees with (\ref{lift the action}). This reduces the total number of variables from \(2n-2\) to \(n-1\), but has the disadvantage that neither Iliev's algebras \(\mathcal{A}_{\mathbf{n},s}(n_1,\dots,n_{s})\) nor the bispectrality of the considered polynomials will play a role.
\end{remark}

\section{The limit $q \to 1$}
\label{Section - limit q to 1}

The results obtained so far establish a strong connection between the higher rank \(q\)-Bannai--Ito algebra \(\mathcal{A}_n^q\) on the one hand and the multivariate \((-q)\)-Racah polynomials on the other. In the limit \(q\to 1\), \(\mathcal{A}_n^q\) reduces to the rank \(n-2\) Bannai--Ito algebra, which was introduced in \cite{DeBie&Genest&Vinet-2016} and has been in its own right the subject of intensive study, see for example \cite{DeBie&Genest&Lemay&Vinet-2017}. In this section we will subject our results to this limiting process and show how this establishes an algebraic framework for a novel class of multivariate \((-1)\)-orthogonal polynomials.

\subsection{Multivariate Bannai--Ito polynomials}

The Bannai--Ito polynomials were introduced in \cite{Bannai&Ito-1984} in the context of algebraic combinatorics. More specifically, Bannai and Ito obtained them as limits \(q\to -1\) of the \(q\)-Racah polynomials, in their classification of orthogonal polynomials satisfying the Leonard duality property. Following the conventions of most recent literature on the subject, for example \cite{Tsujimoto&Vinet&Zhedanov-2012}, we will denote them as \(B_n(x) = B_n(x;\rho_1,\rho_2,r_1,r_2,N)\), where \(\rho_1,\rho_2,r_1,r_2\) are real parameters and \(N\in\N\) is a truncation parameter. They are invariant under the transformations \(\rho_1\leftrightarrow\rho_2 \) and \(r_1\leftrightarrow r_2\) and can be defined through the three-term recurrence relation
\begin{equation}
\label{univariate BI - recurrence relation}
xB_n(x) = B_{n+1}(x) + (\rho_1-A_n-C_n)B_n(x) + A_{n-1}C_nB_{n-1}(x)
\end{equation}
with initial conditions \(B_{-1}(x) = 0, B_0(x) = 1\) and recurrence coefficients
\begin{align*}
A_n = & 
\frac{(n+2\rho_1-2r_1+n_p(2\rho_2-2r_2)+1)(n+2\rho_1-2r_2+n_p(2r_2+2\rho_2)+1)}{4(n+\rho_1+\rho_2-r_1-r_2+1)}, \\
C_n = & 
-\frac{(n+n_p(2\rho_2-2r_2))(n-2r_1-2r_2+n_p(2r_2+2\rho_2))}{4(n+\rho_1+\rho_2-r_1-r_2)},
\end{align*}
where we have written
\[
n = 2n_e + n_p, \quad n_e\in\N, n_p \in \{0,1\}
\]
as in \cite{Lemay&Vinet-2018}, in order to combine the expressions for even and odd \(n\). For explicit expressions in terms of hypergeometric series \(\,_4F_3\) we refer the reader to \cite{Tsujimoto&Vinet&Zhedanov-2012}.

The Bannai--Ito polynomials were observed to satisfy a discrete orthogonality relation of the form
\begin{equation}
	\label{orthogonality BI}
	\sum_{k = 0}^{N}w_{k}B_n(x_k;\rho_1,\rho_2,r_1,r_2,N)B_m(x_k;\rho_1,\rho_2,r_1,r_2,N) = h_n\delta_{n,m},
\end{equation}
given that the positivity condition
\begin{equation}
\label{positivity condition}
A_{n-1}C_n > 0, \qquad \forall n \in \{1,\dots,N \}
\end{equation}
is satisfied, as well as one of the following truncation conditions. For \(N\) even we must have
\begin{equation}
\label{truncation 1}
i)\ r_j - \rho_{\ell} = \frac{N+1}{2},
\end{equation}
for some \(j,\ell\in \{1,2\}\), whereas for \(N\) odd one of the following requirements must be met:
\begin{equation}
\label{truncation 2}
ii)\ \rho_1+\rho_2 = -\frac{N+1}{2}, \qquad iii)\ r_1+r_2 = \frac{N+1}{2}, \qquad iv)\ \rho_1+\rho_2-r_1-r_2 = -\frac{N+1}{2}.
\end{equation}
These conditions are typically referred to as type \(i)\) to \(iv)\). The cases of interest to us will be type \(i)\) with \(j = \ell = 1\) and type \(ii)\), so
\begin{equation}
\label{truncation conditions}
\left\{
\arraycolsep=1.4pt\def\arraystretch{2}
\begin{array}{ll}
r_1-\rho_1 = \dfrac{N+1}{2}, \qquad &\mathrm{for}\ N\ \mathrm{even}, \\
\rho_1+\rho_2 = -\dfrac{N+1}{2}, \qquad &\mathrm{for}\ N\ \mathrm{odd}.
\end{array}
\right.
\end{equation}
If (\ref{positivity condition}) and (\ref{truncation conditions}) are fulfilled, then the Bannai--Ito polynomials satisfy the relation (\ref{orthogonality BI}), with explicit expressions \cite{Genest&Vinet&Zhedanov-2012}
\[
x_k = (-1)^k\left(\frac{k}{2}+\rho_1+\frac14\right)-\frac14
\]
for the grid points,
\begin{align*}
w_k & = (-1)^{k_p}\frac{\left(\rho_1+\rho_2+1\right)_{k_e}\left(\rho_1-r_1+\frac12\right)_{k_e+k_p}\left(\rho_1-r_2+\frac12\right)_{k_e+k_p}\left(2\rho_1+1\right)_{k_e}}{k_e!\left(\rho_1+r_2+\frac12\right)_{k_e+k_p}\left(\rho_1+r_1+\frac12\right)_{k_e+k_p}\left(\rho_1-\rho_2+1\right)_{k_e}}
\end{align*}
for the weight function and
\begin{align*}
h_n & = \frac{N_e!n_e!}{\left(N_e-n_e-n_p(1-N_p)\right)!}\frac{\left(\rho_1+\rho_2-r_1-r_2+n_e+1\right)_{N_e+N_p-n_e}}{\left(\rho_2-r_1+n_e+n_p+\frac12\right)_{N_e+N_p-n_e-n_p}}\\
\times & \frac{\left(2\rho_1+1\right)_{N_e+N_p}\left(N_p(\rho_1-\rho_2-r_1+r_2)+\rho_2-r_2+\frac12\right)_{n_e+n_p}}{\left(\rho_1+r_2+\frac12\right)_{N_e-n_e+N_p(1-n_p)}\left(\left(\rho_1+\rho_2-r_1-r_2+n_e+1\right)_{n_e+n_p} \right)^2} \\
\times & \left(\rho_1-r_2+\frac12\right)_{n_e+n_p}\left(-N_p(\rho_1+\rho_2+r_1+r_2)+\rho_1+\rho_2+1\right)_{n_e}
\end{align*}
for the normalization coefficient. Here we have used the notation
\[
(a)_n = \prod_{\ell = 0}^{n-1}(a+\ell)
\]
for the Pochhammer symbol and as before we have written
\[
N = 2N_e + N_p, \qquad k = 2k_e + k_p, \qquad n = 2n_e+n_p.
\]
with \(N_e, k_e, n_e \in\N\) and \(N_p,k_p,n_p\in\{0,1\}\). In what follows we will also use the hypergeometric series
\[
\,_4F_3\left(\left. 
\begin{array}{c}
a_1,a_2,a_3,a_4 \\ b_1, b_2, b_3
\end{array}\right| z
\right) = \sum_{n = 0}^{\infty}\frac{(a_1,a_2,a_3,a_4)_n}{(b_1,b_2,b_3)_n}\frac{z^n}{n!},
\]
with
\[
(a_1,\dots,a_s)_n = \prod_{k = 1}^s (a_k)_n.
\]

In order to establish the connection with the previously obtained results, we want to provide an extension of these polynomials to multiple variables. In the spirit of Tratnik \cite{Tratnik-1991} and Gasper and Rahman \cite{Gasper&Rahman-2005, Gasper&Rahman-2007}, these multivariate Bannai--Ito polynomials should be entangled products of their univariate counterparts, which moreover should correspond to a limit \(q \to -1\) of the multivariate \(q\)-Racah polynomials (\ref{q-Racah multivariate def}). In order to see how the parameters should be related and to motivate our upcoming Definition \ref{multivariate Bannai-Ito def}, we compute in the next lemma the limit \(q\to 1\) of the basic hypergeometric functions that arise in the definition (\ref{q-Racah multivariate def}) of the multivariate \((-q)\)-Racah polynomials. The proof uses the same techniques as \cite[Section 2.1]{Lemay&Vinet-2018}.

\begin{lemma}
	\label{lemma basic hypergeometric in limit}
	Let \(n_1,\dots,n_k\) and \(x_k,x_{k+1}\) be integers and let \(\alpha_1,\dots,\alpha_{k+1}\) and \(\beta\) be real parameters. In the limit \(q \to 1\), the basic hypergeometric function 
	\begin{equation}
	\label{basic hypergeometric series in the limit}
	\,_4\phi_3\left(\left.
	\begin{array}{c}
	(-q)^{-n_k}, -(-q)^{n_k}q^{2N_{k-1}+\alpha_{[2;k+1]}+\beta}, (-q)^{N_{k-1}-x_{k}}, -(-q)^{N_{k-1}+x_{k}}q^{\alpha_{[1;k]}} \\
	q^{2N_{k-1}+\alpha_{[2;k]}+\beta+1}, -(-q)^{N_{k-1}+x_{k+1}}q^{\alpha_{[1;k+1]}}, (-q)^{N_{k-1}-x_{k+1}}
	\end{array}
	\right|-q,-q
	\right),
	\end{equation}
	with \(N_k = \sum_{i = 1}^kn_i\) and \(\alpha_A = \sum_{i\in A}\alpha_i\), reduces to 
	\begin{equation}
	\label{4F3 expressions}
	_4F_3\left(\left.\begin{array}{c}
	a_1, a_2, a_3, a_4 \\
	b_1,b_2,b_3
	\end{array}\right|1 \right) 
	+ t_{\mathbf{n},\mathbf{x}}\,_4F_3\left(\left.\begin{array}{c}
	a_1+1, a_2, a_3+1, a_4 \\
	b_1+1,b_2+1,b_3
	\end{array}\right|1 \right),
	\end{equation}
	where
	\begin{align}
	\label{4F3 parametrization}
	\begin{split}
	t_{\mathbf{n},\mathbf{x}} & =  \frac{n_k-\left(1-(-1)^{n_k}\right)\left(N_{k}+\frac12(\alpha_{[2;k+1]}+\beta)\right)}{2N_{k-1}+\alpha_{[2;k]}+\beta+1}\\&\times\frac{N_{k-1}-x_{k}+\left(1-(-1)^{N_{k-1}-x_{k}}\right)\left(x_{k}+\frac{\alpha_{[1;k]}}{2}\right)}{N_{k-1}-x_{k+1}+\left(1-(-1)^{N_{k-1}-x_{k+1}}\right)\left(x_{k+1}+\frac{\alpha_{[1;k+1]}}{2}\right)}, \\
	a_1 & = \frac12\left( -n_k+\left(1-(-1)^{-n_k}\right)\left(N_{k}+\frac12(\alpha_{[2;k+1]}+\beta)\right)\right), \\
	a_2 & = \frac12\left( -n_k+1+\left(1-(-1)^{-n_k+1}\right)\left(N_{k}+\frac12(\alpha_{[2;k+1]}+\beta)\right)\right),\\
	a_3 & = \frac12\left(  N_{k-1}-x_{k}+\left(1-(-1)^{N_{k-1}-x_{k}}\right)\left(x_k+\frac{\alpha_{[1;k]}}{2}\right)\right), \\
	a_4 & = \frac12\left(  N_{k-1}-x_{k}+1+\left(1-(-1)^{N_{k-1}-x_{k}+1}\right)\left(x_k+\frac{\alpha_{[1;k]}}{2}\right)\right), \\
	b_1 & = N_{k-1}+\frac{\alpha_{[2;k]}+\beta+1}{2}, \\
	b_2 & = \frac12\left( N_{k-1}-x_{k+1}+\left(1-(-1)^{N_{k-1}-x_{k+1}}\right)\left(x_{k+1}+\frac{\alpha_{[1;k+1]}}{2}\right)\right), \\
	b_3 & = \frac12\left( N_{k-1}-x_{k+1}+1+\left(1-(-1)^{N_{k-1}-x_{k+1}+1}\right)\left(x_{k+1}+\frac{\alpha_{[1;k+1]}}{2}\right)\right).
	\end{split}
	\end{align}
\end{lemma}
\begin{proof}
	Writing down the definition (\ref{basic hypergeometric series def}) of the function \(\,_4\phi_3\), the function (\ref{basic hypergeometric series in the limit}) becomes
	\begin{align}
	\label{4phi3 fully}
	\begin{split}
	& \sum_{n = 0}^{\infty}\left[\prod_{\ell = 0}^{n-1}\left(\frac{\left(1-(-q)^{-n_k+\ell}\right)\left(1-(-q)^{-n_k+\ell+1}q^{2N_{k-1}+2n_k+\alpha_{[2;k+1]}+\beta-1}\right)}{\left(1-(-q)^{\ell+1}\right)\left(1-(-q)^{\ell+2}q^{2N_{k-1}+\alpha_{[2;k]}+\beta-1}\right)}\right.\right.\\
	& \phantom{\sum_{n = 0}^{\infty}\prod_{\ell = 0}^{n-1}}\times\left.\left.\frac{\left(1-(-q)^{N_{k-1}-x_{k}+\ell}\right)\left(1-(-q)^{N_{k-1}-x_{k}+\ell+1}q^{2x_{k}+\alpha_{[1;k]}-1}\right)}{\left(1-(-q)^{N_{k-1}-x_{k+1}+\ell}\right)\left(1-(-q)^{N_{k-1}-x_{k+1}+\ell+1}q^{2x_{k+1}+\alpha_{[1;k+1]}-1}\right)}\right)(-q)^n\right].
	\end{split}
	\end{align}
	The factors in this product have a remarkable symmetry: they have all been arranged in the form 
	\[
	\frac{\left(1-(-q)^{d+\ell}\right)\left(1-(-q)^{d+\ell+1}q^{\delta}\right)}{\left(1-(-q)^{e+\ell}\right)\left(1-(-q)^{e+\ell+1}q^{\epsilon}\right)},
	\]
	with \(d,e\in\Z\) and \(\delta,\epsilon\in\R\). Irrespective of the parities of \(d\) and \(e\), we will always find \(\frac00 \) in the limit \(q\to 1\), which by L'H\^{o}pital's rule reduces to
	\[
	\frac{2d+2\ell+\left(1-(-1)^{d+\ell}\right)(\delta+1)}{2e+2\ell+\left(1-(-1)^{e+\ell}\right)(\epsilon+1)}.
	\]
	Applying this to (\ref{4phi3 fully}), we find that the function (\ref{basic hypergeometric series in the limit}) in the limit \(q\to 1\) becomes
	\begin{align*}
	& \sum_{n = 0}^{\infty}\left((-1)^n\prod_{\ell = 0}^{n-1}f(\ell)\right),
	\end{align*}
	with
	\begin{align*}
	f(\ell) = &\, 
	\frac{2\ell-2n_k+\left(1-(-1)^{-n_k+\ell}\right)\left(2N_{k}+\alpha_{[2;k+1]}+\beta\right)}{2\ell+2+\left(1-(-1)^{\ell+1}\right)\left(2N_{k-1}+\alpha_{[2;k]}+\beta\right)}\\
	& \times\frac{2\ell+2N_{k-1}-2x_{k}+\left(1-(-1)^{N_{k-1}-x_{k}+\ell}\right)\left(2x_{k}+\alpha_{[1;k]}\right)}{2\ell+2N_{k-1}-2x_{k+1}+\left(1-(-1)^{N_{k-1}-x_{k+1}+\ell}\right)\left(2x_{k+1}+\alpha_{[1;k+1]}\right)}.
	\end{align*}
	In order to recognize a hypergeometric series \(\,_4F_3\) in this expression, we proceed as follows:
	\begin{enumerate}
		\item We split the sum over \(n\) in a sum over \(n\) even and one over \(n\) odd, thereby eliminating the \((-1)^n\). Then we rename \(n\to 2n\) in the first sum, and \(n\to 2n+1\) in the second, such that we obtain
		\[
		\sum_{n = 0}^{\infty}\prod_{\ell = 0}^{2n-1}f(\ell) - \sum_{n = 0}^{\infty}\prod_{\ell = 0}^{2n}f(\ell).
		\]
		\item We split the first product over \(\ell\) above into a product for \(\ell\) even, where we rename \(\ell\to 2\ell\), and one for \(\ell\) odd, where we rename \(\ell\to 2\ell+1\). We do the same thing for the second product, but we first separate the factor corresponding to \(\ell = 0\), which will be contained in \(t_{\mathbf{n},\mathbf{x}}\). The limit now becomes
		\[
		\sum_{n = 0}^{\infty}\prod_{\ell = 0}^{n-1}f(2\ell)f(2\ell+1) - f(0)\sum_{n = 0}^{\infty}\prod_{\ell = 0}^{n-1}f(2\ell+2)f(2\ell +1)
		\]
		\item In all obtained fractions we divide numerator and denominator by 4, such we are left with only numerators and denominators of the form \(\prod_{\ell = 0}^{n-1}(\ell+ c)\), for \(c\in\R\) independent of \(\ell\), which we can write as a Pochhammer symbol \((c)_n\).
	\end{enumerate}
	Combining all these steps, the anticipated result follows.
\end{proof}

We are now ready to state our definition of the multivariate Bannai--Ito polynomials.
\begin{definition}
	\label{multivariate Bannai-Ito def}
	The Bannai--Ito polynomials in \(s\) real variables \(x_k\) are defined as
	\begin{align}
	\label{multivariate Bannai-Ito}
	\begin{split}
	& B_{(n_1,\dots,n_s)}\left((x_1,\dots,x_s); \alpha_1,\dots,\alpha_{s+1}, \beta, N\right)\\ = & \, \prod_{k = 1}^sB_{n_k}\left(\frac{(-1)^{N_{k-1}+x_k}}{2}\left(x_k+\frac{\alpha_{[1;k]}}{2}\right)-\frac14;\rho_1^{(k)},\rho_2^{(k)}, r_1^{(k)}, r_2^{(k)}, M^{(k)} \right)
	\end{split}
	\end{align}
	where the parameters are given by
	\begin{equation}
	\label{multivariate BI parameters}
	\begin{split}
	\rho_1^{(k)} & = \frac{N_{k-1}}{2}+\frac{\alpha_{[1;k]}-1}{4}, \\ \rho_2^{(k)} & = \frac{(-1)^{N_{k-1}+x_{k+1}}}{2}\left(x_{k+1}+\frac{\alpha_{[1;k+1]}}{2}\right)+\frac{\alpha_{k+1}-1}{4}, \\
	r_1^{(k)} & = \frac{(-1)^{N_{k-1}+x_{k+1}}}{2}\left(x_{k+1}+\frac{\alpha_{[1;k+1]}}{2}\right)-\left(\frac{\alpha_{k+1}-1}{4}\right), \\ r_2^{(k)} & = -\frac{N_{k-1}}{2}+\frac{\alpha_1-\alpha_{[2;k]}-1}{4}-\frac{\beta}{2}, \\
	M^{(k)} & = x_{k+1}-N_{k-1},
	\end{split}
	\end{equation}
	with \(x_{s+1} = N\in\N\), where \(n_1,\dots,n_s\) are natural numbers and \(\alpha_1,\dots,\alpha_{s+1}\) and \(\beta\) are real parameters subject to the conditions 
	\begin{equation}
	\label{positivity condition alpha}
	\alpha_1-1>\beta>0, \qquad \alpha_i>1, i\in \{2,\dots,s+1\}.
	\end{equation}
\end{definition}
\begin{remark}
	\label{Remark orthogonality}
	Note that the conditions (\ref{positivity condition}) and (\ref{truncation conditions}) for orthogonality are satisfied by each of the univariate Bannai--Ito polynomials in (\ref{multivariate Bannai-Ito}), if we require each \(M^{(k)}\) to be a natural number.
	Indeed, the truncation follows immediately from (\ref{multivariate BI parameters}). As shown in \cite[(1.16) and (1.24)]{Genest&Vinet&Zhedanov-2012}, the positivity condition (\ref{positivity condition}) is fulfilled if
	\[
	\rho_1^{(k)}-r_2^{(k)} > 0, \qquad \rho_1^{(k)}+r_2^{(k)}>0, \qquad \rho_2^{(k)}-r_1^{(k)}>0.
	\]
	With the parametrization (\ref{multivariate BI parameters}), this is tantamount to (\ref{positivity condition alpha}).
\end{remark}
\begin{remark}
	In the special case \(s = 2\), these polynomials coincide with the bivariate Bannai--Ito polynomials recently defined by Lemay and Vinet, up to a change in expression for \(\rho_2^{(2)}\) and \(r_2^{(2)}\). The parametrization in \cite{Lemay&Vinet-2018} can be obtained through the formulas 
	\[
	p_1 = \frac{-\alpha_1+2\beta+3}{4}, \quad p_i = \frac{\alpha_i}{4}, i\in\{2,3\}, \quad c = \frac{\beta}{2}, \quad z_k = \frac{(-1)^{x_k}}{2}\left(x_k + \frac{\alpha_{[1;k]}}{2}\right).
	\]
	The difference in expression of \(\rho_2^{(2)}\) and \(r_2^{(2)}\) can be explained by the truncation: for odd \(N\) the truncation condition of type \(iii)\) was used in \cite{Lemay&Vinet-2018}, whereas we use type \(ii)\).
\end{remark}

Our definition originates from the following lemma, where we compute the limit \(q\to 1\) of a renormalized multivariate \((-q)\)-Racah polynomial with a certain parametrization.

\begin{lemma}
	\label{lemma -q-Racah limit}
	Let \(\widehat{R_{\mathbf{n}}}(\mathbf{x};\mathbf{a},b,N\vert -q)\) be a renormalized \((-q)\)-Racah polynomial (\ref{renormalized multivariate -q-Racah def}) in \(s\) integer variables, where the parameters are such that
	\begin{equation}
	\label{q-Racah parametrization}
	a_1 = -q^{\alpha_1}, \quad a_k = q^{\alpha_k}, k \in [2;s+1], \quad b = -q^{\beta},
	\end{equation}
	for certain \(\alpha_i,\beta\in\R\) subject to (\ref{positivity condition alpha}). Then in the limit \(q\to 1\) these polynomials become proportional to the multivariate Bannai--Ito polynomials:
	\[
	\lim_{q\to 1}\left(\widehat{R_{\mathbf{n}}}(\mathbf{x};\mathbf{a},b,N\vert -q) \right) = k_{\mathbf{n},\mathbf{x},\boldsymbol{\alpha},\beta,N} B_{\mathbf{n}}\left(\mathbf{x};\alpha_1,\dots,\alpha_{s+1},\beta,N\right),
	\]
	with \(\mathbf{n} = (n_1,\dots,n_s)\in\N^s\) and \(\mathbf{x} = (x_1,\dots,x_s)\in\Z^s\), and where the proportionality coefficient is given by
	\begin{equation}
	\label{k_n,x}
	k_{\mathbf{n},\mathbf{x},\boldsymbol{\alpha},\beta,N} =\frac{(-1)^{N_sN+\sum_{k=2}^sn_kN_{k-1}}\prod_{k = 1}^s \eta_{\mathbf{n};k}\zeta_{\mathbf{n},\mathbf{x};k}\xi_{\mathbf{n},\mathbf{x};k}}{\prod_{\ell = 0}^{N_s-1}\left(-(2\ell+2N+\alpha_1+2\alpha_{[2;s+1]}+\beta+1) + (-1)^{N+\ell}(\alpha_1-\beta-1) \right)}
	\end{equation}
	with
	\begin{align}
	\label{eta, zeta, xi}
	\begin{split}
	\eta_{\mathbf{n};k} & = \prod_{\ell = 0}^{(n_k)_e+(n_k)_p-1}\frac{2\ell+2N_{k-1}+\alpha_{[2;k]}+\beta+1}{2\ell+\alpha_{k+1}} \\
	\zeta_{\mathbf{n},\mathbf{x};k} & =
	\dfrac{4^{n_k}(\rho_1^{(k)}+\rho_2^{(k)}-r_1^{(k)}-r_2^{(k)}+1)_{n_k}}{\prod_{\ell = 0}^{n_k-1}\left(\ell + 2\rho_1^{(k)}-2r_1^{(k)}+1+(1-(-1)^{\ell})\left(\rho_2^{(k)}-r_2^{(k)}\right)\right)}\\
	& \times \prod_{\ell = 0}^{n_k-1}\left(\ell + 2\rho_1^{(k)}-2r_2^{(k)}+1+(1-(-1)^{\ell})\left(\rho_2^{(k)}+r_2^{(k)}\right)\right)^{-1} \\
	\xi_{\mathbf{n},\mathbf{x};k} & = \prod_{\ell = 0}^{n_k-1}\left(-(2\ell+2N_{k-1}+\alpha_{[1;k+1]})+(-1)^{x_{k+1}+N_{k-1}+\ell}(2x_{k+1}+\alpha_{[1;k+1]})\right)
	\end{split}
	\end{align}
	and with \(\rho_i^{(k)}\) and \(r_i^{(k)}\) as in (\ref{multivariate BI parameters}).
\end{lemma}
\begin{proof}
	Comparing the definitions (\ref{q-Racah univariate}), (\ref{q-Racah multivariate def}) and (\ref{renormalized multivariate -q-Racah def}) and the parametrization (\ref{q-Racah parametrization}), we find that \(\widehat{R_{\mathbf{n}}}(\mathbf{x}; \mathbf{a}, b,N\vert -q)\) is proportional to
	\[
	\prod_{k = 1}^s \,_4\phi_3\left(\left.
	\begin{array}{c}
	(-q)^{-n_k}, -(-q)^{n_k}q^{2N_{k-1}+\alpha_{[2;k+1]}+\beta}, (-q)^{N_{k-1}-x_{k}}, -(-q)^{N_{k-1}+x_{k}}q^{\alpha_{[1;k]}} \\
	q^{2N_{k-1}+\alpha_{[2;k]}+\beta+1}, -(-q)^{N_{k-1}+x_{k+1}}q^{\alpha_{[1;k+1]}}, (-q)^{N_{k-1}-x_{k+1}}
	\end{array}
	\right|-q,-q
	\right),
	\]
	where we recognize the basic hypergeometric function considered in Lemma \ref{lemma basic hypergeometric in limit}. In the limit \(q\to 1\), the polynomials \(\widehat{R_{\mathbf{n}}}(\mathbf{x};\mathbf{a},b,N\vert -q)\) thus become proportional to products of hypergeometric functions of the form (\ref{4F3 expressions}) with (\ref{4F3 parametrization}). Each such function determines a polynomial \(u_{i}(\theta_j;s,s^{\ast},r_1,r_2,r_3,N)\) as originally defined by Bannai and Ito in \cite[p. 272--273]{Bannai&Ito-1984}, under the reparametrizations
	\begin{align*}
	& \begin{array}{ll}
	i = n_k, & j  = x_{k}-N_{k-1} \\
	s = -2N_{k-1}-\alpha_{[1;k]}+1, \qquad\qquad\quad\ & s^{\ast}  = -2N_{k-1}-\alpha_{[2;k+1]}-\beta+1 \\
	r_1 = 2N_{k-1} + \alpha_{[2;k]}+\beta & 	N = x_{k+1}-N_{k-1},
	\end{array} \\
	& r_2 = \left\{
	\begin{array}{ll}
	-x_{k+1}+N_{k-1}-1& \mathrm{if}\ N\ \mathrm{is\ even} \\
	N_{k-1}+x_{k+1}+\alpha_{[1;k+1]}-1\quad\ \ & \mathrm{if}\ N\ \mathrm{is\ odd} \\
	\end{array}
	\right. \\
	& r_3 = \left\{
	\begin{array}{ll}
	-N_{k-1}-x_{k+1}-\alpha_{[1;k+1]}+1\quad  & \mathrm{if}\ N\ \mathrm{is\ even} \\
	x_{k+1}-N_{k-1}+1& \mathrm{if}\ N\ \mathrm{is\ odd} \\
	\end{array}
	\right.
	\end{align*}
	and where 
	\[
	\theta_j = \left\{
	\arraycolsep=1.4pt\def\arraystretch{1.4}
	\begin{array}{ll}
	-\frac{s}{4} + \frac{j}{2} & \mathrm{if}\ j\ \mathrm{is\ even} \\
	-\frac{s}{4} - \frac12(j+1-s)\qquad & \mathrm{if}\ j\ \mathrm{is\ odd}.
	\end{array}
	\right.
	\]
	This can be checked upon comparing (\ref{4F3 expressions}) and (\ref{4F3 parametrization}) with the hypergeometric expressions in \cite{Bannai&Ito-1984}. Following \cite{Vinet&Zhedanov-2012}, it is clear that the polynomials \(u_i(\theta_j)\) are rescaled Bannai--Ito polynomials
	\[
	\zeta_{\mathbf{n},\mathbf{x};k}\, B_{n_k}\left(\frac{(-1)^{N_{k-1}+x_k}}{2}\left(x_k+\frac{\alpha_{[1;k]}}{2}\right)-\frac14;\rho_1^{(k)},\rho_2^{(k)}, r_1^{(k)}, r_2^{(k)}, M^{(k)} \right),
	\]
	with \(\zeta_{\mathbf{n},\mathbf{x};k}\) as in (\ref{eta, zeta, xi}). The other factors in \(k_{\mathbf{n},\mathbf{x},\boldsymbol{\alpha},\beta,N}\) follow by taking limits of the proportionality coefficients in (\ref{q-Racah univariate}) and (\ref{renormalized multivariate -q-Racah def}). This proves our claim.
\end{proof}
\begin{remark}
	Note that the same choice of parametrization (\ref{q-Racah parametrization}) was made in \cite[(2.11)]{Lemay&Vinet-2018}, for the bivariate case.
\end{remark}

Like their univariate counterparts, the multivariate Bannai--Ito polynomials satisfy a discrete orthogonality relation, which we derive in the next proposition.

\begin{proposition}
	\label{prop BI orthogonality multivariate} 
	The multivariate Bannai--Ito polynomials	satisfy the relation
	\begin{align}
	\label{orthogonality BI multivariate}
	\begin{split}
	& \sum_{\ell_s = 0}^{N}\sum_{\ell_{s-1}= 0}^{\ell_{s}}\dots \sum_{\ell_1 = 0}^{\ell_2}\Omega\left(\boldsymbol{\ell}, \mathbf{n}; \boldsymbol{\alpha}, \beta, N\right)B_{\mathbf{n}}\left(\boldsymbol{\ell}; \boldsymbol{\alpha}, \beta, N\right)B_{\mathbf{m}}\left(\boldsymbol{\ell}; \boldsymbol{\alpha}, \beta, N\right) \\
	= & \ H\left(\mathbf{n}; \boldsymbol{\alpha}, \beta, N\right)\delta_{\mathbf{n},\mathbf{m}}
	\end{split}
	\end{align}
	where the orthogonality grid is given by the simplex
	\[
	\{\boldsymbol{\ell} = (\ell_1,\dots,\ell_s)\in\N^s: 0\leq \ell_1\leq \ell_2\leq\dots\leq \ell_s\leq N\}
	\]
	and we have the expressions
	\begin{align*}
	\Omega\left(\boldsymbol{\ell}, \mathbf{n}; \boldsymbol{\alpha}, \beta, N\right) = \frac{\prod_{i = 1}^s w_{\ell_i-N_{i-1}}^{(i)}\left(\ell_i,\ell_{i+1}, N_{i-1}, \alpha_1,\dots, \alpha_{i+1}, \beta \right)}{\prod_{i = 1}^{s-1}h_{n_i}^{(i)}\left(\ell_{i+1}, N_{i-1},n_i, \alpha_1,\dots,\alpha_{i+1}, \beta\right) }
	\end{align*}
	for the weight function and
	\[
	H\left(\mathbf{n}; \boldsymbol{\alpha}, \beta, N\right) = h_{n_s}^{(s)}\left(N, N_{s-1},n_s, \alpha_1,\dots, \alpha_{s+1},\beta\right),
	\]
	for the normalization coefficient, where we use the convention that \(\ell_{s+1} = N\). Expressions for the functions \(w_{\ell_i-N_{i-1}}^{(i)}\) and \(h_{n_i}^{(i)}\) can be found in Appendix D.
\end{proposition}
\begin{proof}
	We prove this by induction on \(s\). For \(s = 1\) the relation (\ref{orthogonality BI multivariate}) is precisely the univariate orthogonality relation (\ref{orthogonality BI}), which is applicable by Remark \ref{Remark orthogonality} and where the parameters are given by (\ref{multivariate BI parameters}) with \(k = 1\).
	
	Now suppose the claim has been proven for \(s-1\). Let us write
	\begin{align*}
	\boldsymbol{\ell} & = (\ell_1,\dots,\ell_{s}), & \overline{\boldsymbol{\ell}} & =  (\ell_1,\dots,\ell_{s-1}) \\
	\mathbf{n} & = (n_1,\dots,n_s), & \overline{\mathbf{n}} &= (n_1,\dots,n_{s-1}) \\
	\boldsymbol{\alpha} & = (\alpha_1,\dots,\alpha_{s+1}), & \overline{\boldsymbol{\alpha}} &= (\alpha_1,\dots,\alpha_{s}),
	\end{align*}
	Observe that
	\begin{align*}
	& \ B_{\mathbf{n}}\left(\boldsymbol{\ell};\boldsymbol{\alpha},\beta,N\right) \\ = &\  B_{\overline{\mathbf{n}}}\left(\overline{\boldsymbol{\ell}};\overline{\boldsymbol{\alpha}},\beta,\ell_s \right) B_{n_s}\left(\frac{(-1)^{N_{s-1}+\ell_s}}{2}\left(\ell_s+\frac{\alpha_{[1;s]}}{2}\right)-\frac14;\rho_1^{(s)},\rho_2^{(s)},r_1^{(s)},r_2^{(s)},N-N_{s-1} \right),
	\end{align*}
	with \(\rho_i^{(s)}, r_i^{(s)}\) as in (\ref{multivariate BI parameters}). It is immediate that
	\begin{align*}
	\Omega\left(\boldsymbol{\ell}, \mathbf{n}; \boldsymbol{\alpha},\beta,N\right)
	= \Omega(\overline{\boldsymbol{\ell}}, \overline{\mathbf{n}}; \overline{\boldsymbol{\alpha}},\beta,\ell_s)\frac{w_{\ell_s-N_{s-1}}^{(s)}}{h_{n_{s-1}}^{(s-1)}},
	\end{align*}
	where the factor \(\frac{w_{\ell_s-N_{s-1}}^{(s)}}{h_{n_{s-1}}^{(s-1)}}\) is independent of \(\ell_1,\dots,\ell_{s-1}\). Hence we find that the left-hand side of (\ref{orthogonality BI multivariate}) equals
	\begin{align*}
	&\sum_{\ell_s = 0}^{N}\frac{w_{\ell_s-N_{s-1}}^{(s)}}{h_{n_{s-1}}^{(s-1)}}B_{n_s}\left(\frac{(-1)^{N_{s-1}+\ell_s}}{2}\left(\ell_s+\frac{\alpha_{[1;s]}}{2}\right)-\frac14;\rho_1^{(s)},\rho_2^{(s)},r_1^{(s)},r_2^{(s)},N-N_{s-1} \right)\\&\times B_{m_s}\left(\frac{(-1)^{M_{s-1}+\ell_s}}{2}\left(\ell_s+\frac{\alpha_{[1;s]}}{2}\right)-\frac14;\rho_1^{(s)},\rho_2^{(s)},r_1^{(s)},r_2^{(s)},N-M_{s-1} \right)\\
	&\sum_{\ell_{s-1}= 0}^{\ell_{s}}\dots \sum_{\ell_1 = 0}^{\ell_2}\Omega\left(\overline{\boldsymbol{\ell}}, \overline{\mathbf{n}}; \overline{\boldsymbol{\alpha}}, \beta, \ell_s\right) B_{\overline{\mathbf{n}}}\left(\overline{\boldsymbol{\ell}}; \overline{\boldsymbol{\alpha}}, \beta, \ell_s\right) B_{\overline{\mathbf{m}}}\left(\overline{\boldsymbol{\ell}}; \overline{\boldsymbol{\alpha}}, \beta, \ell_s\right),
	\end{align*}
	where we have written \(M_{s-1}\) for \(\sum_{i = 1}^{s-1}m_i\).
	By the induction hypothesis, the sum in the third line yields 
	\[
	\delta_{n_1,m_1}\dots \delta_{n_{s-1},m_{s-1}}H(\overline{\mathbf{n}}; \overline{\boldsymbol{\alpha}},\beta,\ell_s)
	\]
	and by definition, we have \(H(\overline{\mathbf{n}}; \overline{\boldsymbol{\alpha}},\beta,\ell_s) = h_{n_{s-1}}^{(s-1)}(\ell_s,N_{s-2},n_{s-1},\overline{\boldsymbol{\alpha}},\beta)\).
	Hence the left hand side of (\ref{orthogonality BI multivariate}) reduces to
	\begin{align*}
	& \ \delta_{n_1,m_1}\dots \delta_{n_{s-1},m_{s-1}}\sum_{k = 0}^{N-N_{s-1}}w_k^{(s)}B_{n_s}\left(x_k;\rho_1^{(s)},\rho_2^{(s)},r_1^{(s)},r_2^{(s)};N-N_{s-1} \right)\\&\times B_{m_s}\left(x_k;\rho_1^{(s)},\rho_2^{(s)},r_1^{(s)},r_2^{(s)},N-N_{s-1} \right),
	\end{align*}
	with
	\[
	x_k = (-1)^{k}\left(\frac{k}{2}+\rho_1^{(s)}+\frac14\right)-\frac14,
	\]
	and where we have used (\ref{multivariate BI parameters}) and the fact that \(w_{\ell_s-N_{s-1}}^{(s)} = 0\) for \(\ell_s\in \{0,\dots,N_{s-1}-1\}\), as explicitly stated in Appendix D. The assertion now follows immediately from the orthogonality relation (\ref{orthogonality BI}) in the univariate case.
\end{proof}

\subsection{The higher rank $q = 1$ Bannai--Ito algebra and connection coefficients}

The rank \(n-2\) Bannai--Ito algebra \(\mathcal{A}_n\) was introduced in \cite{DeBie&Genest&Vinet-2016} as the abstract associative algebra with generators \(\Gamma_A\), indexed by all possible subsets \(A\subseteq [1;n]\), subject to the relations
\begin{equation}
\label{BI algebra relations}
\{\Gamma_A,\Gamma_B\} = \Gamma_{(A\cup B)\setminus(A\cap B)} + 2\Gamma_{A\cap B}\Gamma_{A\cup B} + 2 \Gamma_{A\setminus(A\cap B)}\Gamma_{B\setminus(A\cap B)},
\end{equation}
for all sets \(A, B \subseteq[1;n]\). It was originally constructed as the symmetry algebra of the so-called Dirac--Dunkl operator
\[
\underline{D} = \sum_{i = 1}^n e_iT_i, \qquad T_i = \partial_{x_i}+(\gamma_i-\tfrac12)\frac{1-r_i}{x_i},
\]
where the \(e_i\) are Clifford elements, satisfying \(\{e_i,e_j\} = -2\delta_{ij}\), the \(\gamma_i>\frac12\) are real parameters and the \(r_i\) are the reflections defined in Section \ref{Paragraph q-Dirac-Dunkl model}. The scalar version of this model, with all \(e_i\) replaced by products of reflections \(\prod_{j = i+1}^nr_j\), coincides with the \(\mathbb{Z}_2^n\) \(q\)-Dirac--Dunkl model introduced in Section \ref{Paragraph q-Dirac-Dunkl model} in the limit \(q\to 1\).

Under the specialization \(q\to 1\), the quantum superalgebra \(\ospq\) reduces to the Lie superalgebra \(\mathfrak{osp}(1\vert 2)\). This is manifest from the alternative presentation for \(\ospq\), with generators \(A_0, A_{\pm}, P\) and relations
\begin{equation}
\label{def - ospq without K}
[A_0,A_{\pm}] = \pm A_{\pm}, \quad \{A_+,A_-\} = [2A_0]_{q^{1/2}},\quad  [P,A_0]= 0, \quad \{P,A_{\pm}\} = 0, \quad P^2 = 1,
\end{equation}
which clearly reduce to the defining relations for \(\mathfrak{osp}(1\vert 2)\) in the limit \(q\to 1\). This presentation relates to (\ref{def - ospq with K}) by taking \(K = q^{A_0/2}\). The coproduct action on \(A_0\) compatible with (\ref{Coproduct}) is hence
\[
\Delta(A_0) = A_0\otimes 1 + 1\otimes A_0.
\]

Comparison of the relations (\ref{BI algebra relations}) with the defining relations (\ref{q-anticommutation Gamma}) of the algebra \(\mathcal{A}_n^q\) explains why the latter is considered a \(q\)-deformation of \(\mathcal{A}_n\). As a consequence, the limit \(q\to 1\) of the extension process (\ref{Gamma_A^q})--(\ref{extension process}) should yield a different construction method for \(\mathcal{A}_n\) by means of embedding in \(\mathfrak{osp}(1\vert 2)^{\otimes n}\). Note that the expression (\ref{def:tau isomorphism}) for the extension morphism \(\tau\) simplifies remarkably under this limit. As a consequence, the requirements (\ref{most general sets 1})--(\ref{most general sets 3}) on the sets \(A\) and \(B\) in Proposition \ref{prop q-anticommutation Gamma} can be omitted in case \(q = 1\).

Taking \(q\to 1\) in (\ref{ospq-action}) we obtain unitary irreducible modules for the Lie superalgebra \(\mathfrak{osp}(1\vert 2)\), which we will denote by \(\widetilde{W}^{(\gamma_i)}\). The solution to the spectral problem proposed in Definition \ref{definition bases final} survives the limit \(q\to 1\) and the expressions for the eigenvalues follow immediately from the fact that
\[
\lim_{q\to 1}\left([\alpha]_q\right) = \alpha,
\]
for any \(\alpha\in\R\). Moreover, it follows from Corollary \ref{cor characterized by equations} that the vectors constructed in Definition \ref{definition bases final} are uniquely determined by the equations (\ref{eigenvalue equation 1}), (\ref{eigenvalue equation 2}) and (\ref{eigenvalue equation K}). Hence we are led to propose the following analog of Definition \ref{definition bases final} for \(q = 1\). For ease of notation, we will consider only the \(q \to 1\) limits of the vectors (\ref{vector s}) and (\ref{vector j}), although one could similarly define \(\vert \tilde{\mathbf{j}}^{(i,r)};N\rangle\) for any \(i,r\).

\begin{definition}
	We denote by \(\vert\tilde{\mathbf{s}};N\rangle = \vert \tilde{s}_{n-1},\tilde{s}_{n-2},\dots,\tilde{s}_{1};N\rangle\), with all \(\tilde{s}_k\in\N\) and \(N\in\N\), the vectors \(\in \widetilde{W}^{(\gamma_1)}\otimes\dots\otimes \widetilde{W}^{(\gamma_n)}\) diagonalizing the abelian subalgebra
	\[
	\langle \Gamma_{[1;2]},\Gamma_{[1;3]},\dots,\Gamma_{[1;n]}\rangle
	\]
	of \(\mathcal{A}_n\) through the eigenvalue equations
	\[
	\Gamma_{[1;k+1]}\vert \tilde{s}_{n-1},\dots,\tilde{s}_{1};N\rangle = (-1)^{\vst{k}}\left(\vst{k}+\gamma_{[1;k+1]}-\frac12\right)\vert \tilde{s}_{n-1},\dots,\tilde{s}_{1};N\rangle
	\]
	for all \(k\in [1;n-1]\), and satisfying
	\begin{equation}
	\label{eigenvalue equation A0 1}
	\left(A_0\right)_{[1;n]}\vert \tilde{s}_{n-1},\dots,\tilde{s}_{1};N\rangle = \left(\vert\tilde{\mathbf{s}}_{n-1}\vert+N+\gamma_{[1;n]}\right)\vert \tilde{s}_{n-1},\dots,\tilde{s}_{1};N\rangle,
	\end{equation}
	where
	\[
	\left(A_0\right)_{[1;n]} = \Delta^{(n)}(A_0) = \sum_{i = 0}^{n-1}1^{\otimes i}\otimes A_0 \otimes 1^{\otimes (n-i-1)},
	\]
	with \(A_0\) as in (\ref{def - ospq without K}). Similarly, the vectors \(\vert\tilde{\mathbf{j}};N\rangle = \vert \tilde{j}_{n-1},\tilde{j}_{n-2},\dots,\tilde{j}_{1};N\rangle\) diagonalize the abelian subalgebra
	\[
	\langle \Gamma_{[2;3]}, \Gamma_{[2;4]},\dots,\Gamma_{[2;n]},\Gamma_{[1;n]}\rangle
	\]
	of \(\mathcal{A}_n\) through the relations
	\[
		\Gamma_{[2;k+1]}\vert \tilde{j}_{n-1},\dots,\tilde{j}_{1};N\rangle =  (-1)^{\vert\tilde{\mathbf{j}}_{k-1}\vert}\left(\vert\tilde{\mathbf{j}}_{k-1}\vert+\gamma_{[2;k+1]}-\frac12\right)\vert \tilde{j}_{n-1},\dots,\tilde{j}_{1};N\rangle,
	\]
	for all \(k\in [2;n-1]\) and
	\[
	\Gamma_{[1;n]}\vert\tilde{j}_{n-1},\dots,\tilde{j}_{1};N\rangle = (-1)^{\vert\tilde{\mathbf{j}}_{n-1}\vert}\left(\vert\tilde{\mathbf{j}}_{n-1}\vert+\gamma_{[1;n]}-\frac12\right)\vert \tilde{j}_{n-1},\dots,\tilde{j}_{1};N\rangle.
	\]
	The quantum number \(N\) is again fixed by the equation
	\begin{equation}
	\label{eigenvalue equation A0 2}
	\left(A_0\right)_{[1;n]}\vert \tilde{j}_{n-1},\dots,\tilde{j}_{1};N\rangle = \left(\vert\tilde{\mathbf{j}}_{n-1}\vert+N+\gamma_{[1;n]}\right)\vert \tilde{j}_{n-1},\dots,\tilde{j}_{1};N\rangle.
	\end{equation}
	As in the \(q\)-deformed case, these vectors are orthogonal and we may again choose the normalization such that
	\[
	\langle\tilde{\mathbf{s}};N\vert\tilde{\mathbf{s}}';N'\rangle = \delta_{\tilde{\mathbf{s}},\tilde{\mathbf{s}}'}\delta_{N,N'}, \qquad \langle\tilde{\mathbf{j}};N\vert\tilde{\mathbf{j}}';N'\rangle = \delta_{\tilde{\mathbf{j}},\tilde{\mathbf{j}}'}\delta_{N,N'}.
	\]
\end{definition}
\begin{remark}
	The equations (\ref{eigenvalue equation A0 1}) and (\ref{eigenvalue equation A0 2}) are the \(q \to 1\) analogs of (\ref{eigenvalue equation K}), as is motivated by the expression 
	\[
	K = q^{A_0/2}.
	\]
\end{remark}

We will now focus on the overlap coefficients \(\langle\tilde{\mathbf{j}};m\vert\tilde{\mathbf{s}};m\rangle\), which by the previously established connections are in fact the Racah coefficients of the algebra \(\mathfrak{osp}(1\vert 2)\). As before, it suffices to treat the case \(m = 0\), as for general \(m\) we have
\[
\langle\tilde{\mathbf{j}};m\vert\tilde{\mathbf{s}};m\rangle = \left(\prod_{\ell = 1}^m\left(\ell+(1-(-1)^{\ell})(\vert\bj_{n-1}\vert+\gamma_{[1;n]}-\tfrac12)\right) \right) \langle\tilde{\mathbf{j}};0\vert\tilde{\mathbf{s}};0\rangle,
\]
as follows from (\ref{effect of m not zero}) in the limit \(q \to 1\). Therefore, we will write \(\vert\tilde{\mathbf{j}}\rangle\) and \(\vert\tilde{\mathbf{s}}\rangle\) for \(\vert\tilde{\mathbf{j}};0\rangle \) and \(\vert\tilde{\mathbf{s}};0\rangle\). Like in the \(q\)-deformed case, the overlap coefficients have a nice expression in terms of orthogonal polynomials.

\begin{theorem}
	\label{th - Overlap q = 1}
	The overlap coefficients \(\langle\tilde{\mathbf{j}}\vert\tilde{\mathbf{s}}\rangle\) are proportional to multivariate Bannai--Ito polynomials
	\begin{align*}
	& \langle \tilde{j}_{n-1},\dots,\tilde{j}_{1}\vert \tilde{s}_{n-1},\dots,\tilde{s}_1\rangle \\ = &\, \delta_{\vst{n-1},\vert\tilde{\mathbf{j}}_{n-1}\vert}C_{\tilde{\mathbf{s}},\tilde{\mathbf{j}}}
	B_{(\tilde{j}_1,\dots,\tilde{j}_{n-2})}\left((\tilde{s}_1,\dots,\vst{n-2}); 2\gamma_{[1;2]}-1, 2\gamma_3,\dots,2\gamma_n, 2\gamma_2-1, \vst{n-1}\right).
	\end{align*}
	The proportionality coefficient is given by
	\[
	C_{\tilde{\mathbf{s}},\tilde{\mathbf{j}}} = \sqrt{\frac{\Omega\left((\tilde{s}_1,\dots,\vst{n-2}),(\tilde{j}_1,\dots,\tilde{j}_{n-2}); 2\gamma_{[1;2]}-1, 2\gamma_3,\dots,2\gamma_n,2\gamma_2-1,\vst{n-1}\right)}{H\left((\tilde{j}_1,\dots,\tilde{j}_{n-2});2\gamma_{[1;2]}-1, 2\gamma_3,\dots,2\gamma_n,2\gamma_2-1,\vst{n-1}\right)}},
	\]
	where the functions \(\Omega\) and \(H\) have been defined in Proposition \ref{prop BI orthogonality multivariate}.
\end{theorem}
\begin{proof}
	By definition, these are the limits \(q\to 1\) of the overlap coefficients (\ref{overlap with j and s}). The choice of parameters (\ref{parametrization overlap j and s}) coincides with (\ref{q-Racah parametrization}), and (\ref{positivity condition alpha}) is satisfied since we require all \(\gamma_i>\frac12\). Hence we find from Lemma \ref{lemma -q-Racah limit}
	\[
	\langle\tilde{\mathbf{j}}\vert\tilde{\mathbf{s}}\rangle \propto B_{(\tilde{j}_1,\dots,\tilde{j}_{n-2})}\left((\tilde{s}_1,\dots,\vst{n-2}); 2\gamma_{[1;2]}-1, 2\gamma_3,\dots,2\gamma_n, 2\gamma_2-1, \vst{n-1}\right).
	\]
	The proportionality coefficient \(C_{\tilde{\mathbf{s}},\tilde{\mathbf{j}}}\) has been derived from the orthonormality of the vectors \(\vert\tilde{\mathbf{j}}\rangle\) in combination with Proposition \ref{prop BI orthogonality multivariate}, following the same lines as the proof of Theorem \ref{th - Overlap q-case}. 
\end{proof}
Note that this agrees with the results obtained in \cite{Genest&Vinet&Zhedanov-2012} for the univariate case, i.e. the case \(n = 3\), modulo a transformation \(\gamma_1\leftrightarrow\gamma_3\) of the representation parameters. 
\begin{remark}
	If we denote by \(\vert\tilde{\mathbf{j}}^{(i,r)};N\rangle\) the vector introduced in Definition \ref{definition bases final} for \(q = 1\), then as before we would find that \(\langle\tilde{\mathbf{j}}^{(i,r)}; m\vert\tilde{\mathbf{s}};m\rangle\)
	is proportional to a multivariate Bannai--Ito polynomial
	\[
	B_{(\tilde{j}^{(i,r)}_i,\dots,\tilde{j}^{(i,r)}_{i+r-2})}\left((\tilde{s}_i,\dots,\vst{i+r-2}-\vst{i-1});\boldsymbol{\alpha},2\gamma_{i+1}-1, \vst{i+r-1}-\vst{i-1} \right),
	\]
	with
	\[
	\boldsymbol{\alpha} = (\alpha_1,\dots,\alpha_{r}), \qquad \alpha_1 = 2\vst{i-1}+2\gamma_{[1;i+1]}-1, \qquad \alpha_k = 2\gamma_{i+k},\ k\in [2;r].
	\]
\end{remark} 

\subsection{Discrete realization}

In this final section we will investigate the bispectrality of the multivariate Bannai--Ito polynomials from Definition \ref{multivariate Bannai-Ito def}. Like before, we will first renormalize our polynomials as follows:
\begin{equation}
\label{renormalized BI polynomials def}
\widehat{B}_{\mathbf{n}}(\mathbf{x};\boldsymbol{\alpha},\beta,N) = k_{\mathbf{n},\mathbf{x},\boldsymbol{\alpha},\beta,N} B_{\mathbf{n}}(\mathbf{x};\boldsymbol{\alpha},\beta,N),
\end{equation}
with \(k_{\mathbf{n},\mathbf{x},\boldsymbol{\alpha},\beta,N}\) as in (\ref{k_n,x}). By Lemma \ref{lemma -q-Racah limit}, these are the renormalized multivariate \((-q)\)-Racah polynomials (\ref{renormalized multivariate -q-Racah def}) with parametrization (\ref{q-Racah parametrization}) in the limit \(q\to 1\). The following is the \(q = 1\) analog of Proposition \ref{prop Iliev x}, i.e.\ the \(q = -1\) analog of \cite[Propositions 4.2 and 4.5]{Iliev-2011}.

\begin{proposition}
	\label{prop Iliev x -1}
	Let \(x_1,\dots,x_{s}\) be integer variables. Let us define the operator
	\begin{align*}
	& \widetilde{\mathcal{L}}_j^{\mathbf{x}}\left(x_1,\dots,x_j;\alpha_1,\dots,\alpha_{j+1},\beta,x_{j+1}\right) \\
	= & \sum_{\nuu\in\{-1,0,1\}^j}\widetilde{C}_{\nuu}(x_1,\dots,x_{j+1})T_{+,x_1}^{\nu_1}\dots T_{+,x_j}^{\nu_j}\\&+\left(\frac{\alpha_{[2;j+1]}+\beta}{2}+(-1)^{x_{j+1}}(\alpha_1-\beta-1)\left(x_{j+1}+\frac{\alpha_{[1;j+1]}}{2}\right)\right),
	\end{align*}
	with the convention that \(x_{s+1} = N\) and where the \(\widetilde{C}_{\nuu}(x_1,\dots,x_{j+1})\) are functions in the variables \(x_1,\dots,x_{j+1}\), explicit expressions of which can be found in Appendix E. Then the operators \[\widetilde{\mathcal{L}}_j^{\mathbf{x}}\left(x_1,\dots,x_j;\alpha_1,\dots,\alpha_{j+1},\beta,x_{j+1}\right), \qquad j\in\{1,\dots,s\},\] are mutually commuting operators, with the renormalized Bannai--Ito polynomials in \(s\) variables as common eigenfunctions:
	\begin{equation}
	\label{difference relation BI}
	\widetilde{\mathcal{L}}_j^{\mathbf{x}}\left(x_1,\dots,x_j;\alpha_1,\dots,\alpha_{j+1},\beta,x_{j+1}\right)\widehat{B}_{\mathbf{n}}(\mathbf{x};\boldsymbol{\alpha},\beta,N) = \widetilde{\mu}_j\widehat{B}_{\mathbf{n}}(\mathbf{x};\boldsymbol{\alpha},\beta,N),
	\end{equation}
	with eigenvalue
	\[
	\widetilde{\mu}_j = \left(\frac{\alpha_{[2;j+1]}+\beta}{2}-(-1)^{N_j}\left(N_j+\frac{\alpha_{[2;j+1]}+\beta}{2}\right) \right).
	\]
\end{proposition}
\begin{proof}
	This is immediate from Proposition \ref{prop Iliev x}, Lemma \ref{lemma -q-Racah limit} and the fact that with the parametrization (\ref{q-Racah parametrization}) one finds 
	\begin{align*}
	\widetilde{\mathcal{L}}_j^{\mathbf{x}}\left(x_1,\dots,x_j;\alpha_1,\dots,\alpha_{j+1},\beta,x_{j+1}\right) & = \lim_{q\to 1}\left(\frac{\mathcal{L}_j^{\mathbf{x}}(x_1,\dots,x_{j};a_1,\dots,a_{j+1},b,x_{j+1}\vert-q)}{q-q^{-1}} \right),\\ \widehat{B}_{\mathbf{n}}(\mathbf{x};\boldsymbol{\alpha},\beta,N) & = \lim_{q\to 1}\left(\widehat{R_{\mathbf{n}}}(\mathbf{x};\mathbf{a},b,N\vert-q)\right), \\\widetilde{\mu}_j & = \lim_{q\to 1}\left(\frac{\mu_j}{q-q^{-1}}\right).
	\end{align*}
	Note that the restriction that all \(x_i\) be integer is necessary for the limits to converge.
\end{proof}

Observe that the equation (\ref{difference relation BI}) establishes in fact a \(3^j\)-term difference relation for the multivariate Bannai--Ito polynomials.
Similarly we can define a second commutative algebra of shift operators in the discrete variables \(n_i\) that diagonalize the polynomials \(\widehat{B}_{\mathbf{n}}(\mathbf{x};\boldsymbol{\alpha},\beta,N)\), as a \(q = 1\) analog of Proposition \ref{prop Iliev n}. This will lead to a \(3^j\)-term recurrence relation for the \(s\)-variate Bannai--Ito polynomials, for any \(j\in\{1,\dots,s\}\).

\begin{proposition}
	\label{prop Iliev n -1}
	Let \(x_1,\dots,x_{s},n_1,\dots,n_{s}\) be a set of integer variables. Let us define the operator 
	\begin{align*}
	& \widetilde{\mathfrak{L}}_j^{\mathbf{n}}\left(n_1,\dots,n_s;\alpha_1,\dots,\alpha_{s+1},\beta,N\right) \\
	= & \sum_{\nuu\in\{-1,0,1\}^j}\widetilde{D}_{\nuu}(n_1,\dots,n_s)T_{+,n_{s-j+1}}^{\nu_j}T_{+,n_{s-j+2}}^{\nu_{j-1}-\nu_j}\dots T_{+,n_s}^{\nu_1-\nu_2}\\&+\left(\frac{\alpha_{[1;s+1]}+\alpha_{[s-j+2;s+1]}}{2}+N-(-1)^{N+N_{s-j}}\left(N_{s-j}+\frac{\alpha_{[2;s-j+1]}+\beta}{2}\right)(\beta-\alpha_1+1) \right),
	\end{align*}
	where the \(\widetilde{D}_{\nuu}(n_1,\dots,n_s)\) are functions in the variables \(n_1,\dots,n_s\), explicit expressions of which can be found in Appendix E. Then the operators \[\widetilde{\mathfrak{L}}_j^{\mathbf{n}}\left(n_1,\dots,n_s;\alpha_1,\dots,\alpha_{s+1},\beta,N\right), \qquad j\in\{1,\dots,s\},\] are mutually commuting operators, with the renormalized Bannai--Ito polynomials in \(s\) variables as common eigenfunctions:
	\begin{equation}
	\widetilde{\mathfrak{L}}_j^{\mathbf{n}}\left(n_1,\dots,n_s;\alpha_1,\dots,\alpha_{s+1},\beta,N\right)\widehat{B}_{\mathbf{n}}(\mathbf{x};\boldsymbol{\alpha},\beta,N) = \widetilde{\kappa}_j\widehat{B}_{\mathbf{n}}(\mathbf{x};\boldsymbol{\alpha},\beta,N),
	\end{equation}
	with eigenvalue
	\[
	\widetilde{\kappa}_j = \left(\frac{\alpha_{[1;s+1]}+\alpha_{[s-j+2;s+1]}}{2}+N-(-1)^{N+x_{s-j+1}}\left(x_{s-j+1}+\frac{\alpha_{[1;s-j+1]}}{2}\right) \right).
	\]
\end{proposition}
\begin{proof}
	This is again immediate from Proposition \ref{prop Iliev n}, Lemma \ref{lemma -q-Racah limit} and the fact that with the parametrization (\ref{q-Racah parametrization}) one finds
	\begin{align*}
	\widetilde{\mathfrak{L}}_j^{\mathbf{n}}\left(n_1,\dots,n_s;\alpha_1,\dots,\alpha_{s+1},\beta,N\right) & = \lim_{q\to 1}\left(\frac{\mathfrak{L}_j^{\mathbf{n}}(n_1,\dots,n_s;a_1,\dots,a_{s+1},b,N\vert-q)}{q-q^{-1}} \right), \\ 
	\widehat{B}_{\mathbf{n}}(\mathbf{x};\boldsymbol{\alpha},\beta,N) & = \lim_{q\to 1}\left(\widehat{R_{\mathbf{n}}}(\mathbf{x};\mathbf{a},b,N\vert-q)\right),\\	
	\widetilde{\kappa}_j & = \lim_{q\to 1}\left(\frac{\kappa_j}{q-q^{-1}}\right).
	\end{align*}
\end{proof}

We will conclude with a discrete realization of the rank \(n-2\) Bannai--Ito algebra. Our representation space will be defined as follows.

\begin{definition}
	\label{representation spaces V1 def}
	We will denote by \(\mathcal{V}\) the infinite-dimensional vector space spanned by the renormalized multivariate Bannai--Ito polynomials
	\[
	\widehat{B}_{(\tilde{j}_1,\dots,\tilde{j}_{n-2})}\left((\tilde{s}_1,\tilde{s}_1+\tilde{s}_2,\dots,\vst{n-2}); \boldsymbol{\alpha}, \beta, \vst{n-1}\right), \quad \tilde{\mathbf{j}}\in\N^{n-1},
	\]
	considered as functions of \(\tilde{\mathbf{s}}\in\N^{n-1}\), where \(\tilde{j}_{n-1}\) is fixed by the constraint \(\vert\tilde{\mathbf{j}}_{n-1}\vert = \vst{n-1}\) and with
	\[
	\boldsymbol{\alpha} = (\alpha_1,\dots,\alpha_{n-1}), \quad \alpha_1 = 2\gamma_1+2\gamma_2-1, \quad \alpha_k = 2\gamma_{k+1}, k\in[2;n-1], \quad \beta = 2\gamma_2-1. 
	\]
\end{definition}
By Lemma \ref{lemma -q-Racah limit}, this is the \(q \to 1\) analog of the space \(\mathcal{V}_q\) from Definition \ref{representation spaces Vq def}. Finally, taking limits in Theorem \ref{th discrete realization}, we obtain a realization of the algebra \(\mathcal{A}_n\) on the space \(\mathcal{V}\) by means of difference operators.

\begin{theorem}
	\label{th discrete realization q = 1}
	Let \(\gamma_i>\frac12\), \(i\in[1;n]\), be a set of real parameters. Let us define
	\begin{equation}
	\label{Gamma starts with 1 q = 1}
	\Gamma_{[1;m+1]} = (-1)^{\vst{n-1}}\left(-\widetilde{\mathfrak{L}}_{n-m-1}^{\mathbf{n}}+\left(\vst{n-1}+\gamma_{[1;n]}+\gamma_{[m+2;n]}-\frac12\right)\right)
	\end{equation}
	for \(m\in[0;n-1]\) and
	\begin{equation}
	\label{Gamma starts with i q = 1}
	\Gamma_{[i+1;i+r]} = -\widetilde{\mathcal{L}}_{i,r-1}^{\mathbf{x}}+\left(\gamma_{[i+1;i+r]}-\frac12\right)
	\end{equation}
	for \(i\in[1;n-1],r\in[1;n-i]\), with the abbreviations \(\widetilde{\mathfrak{L}}_{n-m-1}^{\mathbf{n}}\) for
	\[
	\widetilde{\mathfrak{L}}_{n-m-1}^{\mathbf{n}}\left(\tilde{j}_1,\dots,\tilde{j}_{n-2};2\gamma_{[1;2]}-1,2\gamma_3,\dots,2\gamma_n,2\gamma_2-1,\vert\tilde{\mathbf{j}}_{n-1}\vert\right)
	\]
	and \(\widetilde{\mathcal{L}}_{i,r-1}^{\mathbf{x}}\) for
	\[
	\widetilde{\mathcal{L}}_{r-1}^{\mathbf{x}}\left(\tilde{s}_i,\dots,\vst{i+r-2}-\vst{i-1};2\vst{i-1}+2\gamma_{[1;i+1]}-1,2\gamma_{i+2},\dots,2\gamma_{i+r},2\gamma_{i+1}-1,\vst{i+r-1}-\vst{i-1}\right),
	\]
	with the convention that \(\widetilde{\mathcal{L}}_{i,0}^{\mathbf{x}} = \widetilde{\mathfrak{L}}_0^{\mathbf{n}} = 0\) and that
	\[
	\widetilde{\mathfrak{L}}_{n-1}^{\mathbf{n}} = \left(\vss{n-1}+\gamma_1+2\gamma_{[2;n]}-\frac12\right)-(-1)^{\vss{n-1}}\left(\gamma_1-\frac12\right).
	\]
	Then the algebra generated by the operators (\ref{Gamma starts with 1 q = 1}) and (\ref{Gamma starts with i q = 1}) forms a discrete realization of the rank \(n-2\) Bannai--Ito algebra \(\mathcal{A}_n\) on the module \(\mathcal{V}\).
\end{theorem}

\section{Conclusions and outlook}
In this paper we have obtained the connection coefficients between the eigenbases of several abelian subalgebras of the higher rank \(q\)-Bannai--Ito algebra \(\mathcal{A}_n^q\). We have given an explicit expression of these coefficients in terms of multivariable \((-q)\)-Racah polynomials and shown how these generate a family of infinite-dimensional modules for \(\mathcal{A}_n^q\), thereby providing an algebraic underpinning for \cite{Iliev-2011}. The limit \(q\to 1\) led to similar results for the higher rank \(q = 1\) Bannai--Ito algebra, and suggested a natural extension of the Bannai--Ito polynomials to multiple variables.

In Definition \ref{multivariate Bannai-Ito def} the univariate Bannai--Ito polynomials were chosen to satisfy specific truncation conditions. Other choices of truncation (\ref{truncation 1})--(\ref{truncation 2}), as well as different parametrizations in (\ref{q-Racah parametrization}), should lead to other types of multivariate Bannai--Ito polynomials, such as the polynomials obtained in \cite{Lemay&Vinet-2018} for the bivariate case. It might be of interest to determine the influence of these choices on the orthogonality and bispectrality of the multivariate Bannai--Ito polynomials. Moreover, limiting processes such as established in \cite{Vinet&Zhedanov-2012} should lead to \(q = -1\) analogs of several other multivariate \(q\)-orthogonal polynomials, such as \(q\)-Hahn, \(q\)-Jacobi and many others, thereby providing multivariate extensions for the whole \(q = -1\) Askey scheme. 

In Theorem \ref{th - Overlap q = 1} we have obtained a realization of the higher rank \(q = 1\) Bannai--Ito algebra by operators acting on multivariate polynomials through discrete shifts in the variables. It should be possible to obtain a different realization, with both difference operators and reflections, as a multivariate extension of \cite{Tsujimoto&Vinet&Zhedanov-2012}, thereby complementing the bispectrality of the multivariate Bannai--Ito polynomials. The same principles should work for the \(q\)-deformed case.

Throughout this paper, we have restricted our attention to abelian subalgebras of the type (\ref{subalgebra Y_i,r}). An interesting but undoubtedly very technical problem would be the generalization of Theorem \ref{th - Overlap q-case} to other types of abelian subalgebras. In the terminology of \cite{DeBie&Genest&vandeVijver&Vinet-2016}, this would correspond to asking whether the recoupling graph is connected.

In \cite[Section 4]{Genest&Iliev&Vinet-2018} a method is suggested to construct superintegrable quantum Calogero-Gaudin models with \(\mathfrak{osp}_q(1\vert 2)\)-symmetry from the \(q\)-Bannai--Ito generators. This would establish a \(q\)-deformation of the superintegrable systems with reflections in \cite{DeBie&Genest&Lemay&Vinet-2017}. 

Finally, the results presented here strenghten the suggestion made in \cite{Genest&Vinet&Zhedanov-2016} of a deeper connection between quantum algebras and quantum superalgebras under the transformation \(q\to -q\), based on \cite{Tsujimoto&Vinet&Zhedanov-2011} and \cite{Zachos-1992}. 

We hope to report on many of these issues in the near future.

\section*{Acknowledgements}
This work was supported by the Research Foundation Flanders (FWO) under Grant EOS 30889451. We would like to thank Plamen Iliev for sharing with us his Maple-code that was used for the verification of \cite[Theorem 5.5]{Iliev-2011}. We are also grateful to the anonymous referee for useful comments.

\newpage
\section*{Appendix A: Weight function and normalization coefficient for the multivariate $(-q)$-Racah polynomials in Theorem \ref{th - Overlap q-case}}
The functions \(\rho^{(i,r)}(\mathbf{j}^{(1,1)})\) and \(h_{\mathbf{j}^{(i,r)}}\) in Theorem \ref{th - Overlap q-case} have the expression
\begin{align}
\label{gauge coefficient rho}
\begin{split}
& \rho^{(i,r)}(\mathbf{j}^{(1,1)}) \\
= & \ \frac{(-q, (-q)^{2\vj{1,1}{i-1}}q^{2\gamma_{[1;i+r-1]}};-q)_{\vj{1,1}{i+r-1}-\vj{1,1}{i-1}}}{(q^{2\gamma_{i+r}}, (-q)^{2\vj{1,1}{i-1}-1}q^{2\gamma_{[1;i+r]}};-q)_{\vj{1,1}{i+r-1}-\vj{1,1}{i-1}}}\frac{((-q)^{2\vj{1,1}{i-1}-1}q^{2\gamma_{[1;i+1]}},q^{2\gamma_{i+1}}; -q)_{j_i^{(1,1)}}}{(-q, (-q)^{2\vj{1,1}{i-1}}q^{2\gamma_{[1;i]}};-q)_{j_i^{(1,1)}}}
\\
\times & \prod_{k = 1}^{r-1}\left[\frac{(q^{2\gamma_{i+k+1}};-q)_{j_{i+k}^{(1,1)}}}{(-q;-q)_{j_{i+k}^{(1,1)}}}\frac{((-q)^{2\vj{1,1}{i-1}-1}q^{2\gamma_{[1;i+k+1]}};-q)_{\vj{1,1}{i+k-1}+\vj{1,1}{i+k}-2\vj{1,1}{i-1}}}{((-q)^{2\vj{1,1}{i-1}}q^{2\gamma_{[1;i+k]}};-q)_{\vj{1,1}{i+k-1}+\vj{1,1}{i+k}-2\vj{1,1}{i-1}}}\right] \\
\times & \prod_{k = 1}^{r-1}\left[\frac{1-(-q)^{2\vj{1,1}{i+k-1}-1}q^{2\gamma_{[1;i+k]}}}{1-(-q)^{2\vj{1,1}{i-1}-1}q^{2\gamma_{[1;i+k]}}}\left(q^{2\gamma_{i+k}}\right)^{-\vj{1,1}{i+k-1}+\vj{1,1}{i-1}}\right]
\end{split}
\end{align}
and
\begin{align*}
& h_{\mathbf{j}^{(i,r)}} \\
= & \ \frac{((-q)^{2\vj{1,1}{i-1}}q^{2\gamma_{[1;i+r-1]}}, q^{2\gamma_{[i+1;i+r]}}; -q)_{\vj{1,1}{i+r-1}-\vj{1,1}{i-1}}}{(q^{2\gamma_{i+r}},(-q)^{2\vj{1,1}{i-1}}q^{2\gamma_{[1;i]}}; -q)_{\vj{1,1}{i+r-1}-\vj{1,1}{i-1}}}\left(q^{-2\gamma_{[i+1;i+r-1]}}\right)^{\vj{1,1}{i+r-1}-\vj{1,1}{i-1}} \\
\times & \left((-q)^{\vj{1,1}{i+r-1}-\vj{1,1}{i-1}}q^{2\gamma_{[i+1;i+r]}}, (-q)^{\vj{1,1}{i+r-1}+\vj{1,1}{i-1}-1}q^{2\gamma_{[1;i+r]}}; -q \right)_{\vj{i,r}{i+r-2}-\vj{i,r}{i-1}} \\ \times &\left((-q)^{1-\vj{1,1}{i+r-1}-\vj{1,1}{i-1}}q^{-2\gamma_{[1;i]}}, (-q)^{-\vj{1,1}{i+r-1}+\vj{1,1}{i-1}}; -q \right)_{\vj{i,r}{i+r-2}-\vj{i,r}{i-1}} \\
\times & \prod_{k = 1}^{r-1}\left[(-q, q^{2\gamma_{i+k+1}}; -q)_{j_{i+k-1}^{(i,r)}}\frac{(q^{2\gamma_{[i+1;i+k]}};-q)_{\vj{i,r}{i+k-2}+\vj{i,r}{i+k-1}-2\vj{i,r}{i-1}}}{(-q^{2\gamma_{[i+1;i+k+1]}-1};-q)_{\vj{i,r}{i+k-2}+\vj{i,r}{i+k-1}-2\vj{i,r}{i-1}}}\right.\\&\left.\phantom{\prod_{k = 1}^{j-2}}\times\frac{1+q^{2\gamma_{[i+1;i+k+1]}-1}}{1+q^{2\gamma_{[i+1;i+k+1]}-1-2\vj{1,1}{i-1}+2\vj{i,r}{i+k-1}}}\right].
\end{align*}

\newpage
\section*{Appendix B: Explicit expressions for the funtions $C_{\boldsymbol{\nu}}(\mathbf{x})$ in Proposition \ref{prop Iliev x} and $D_{\boldsymbol{\nu}}(\mathbf{n})$ in Proposition \ref{prop Iliev n}}
The following expressions for the functions \(C_{\nuu}(x_1,\dots,x_{j+1})\) can be obtained from \cite[(4.2)--(4.4)]{Iliev-2011}, upon using the reparametrization (\ref{Reparametrization AW to q-Racah}). For \(\nuu\in \{-1,0,1\}^j\), let
\[
\nuu^{\pm} = (\nu^{\pm}_1,\dots,\nu^{\pm}_{j})\in\{0,1\}^j, \quad \nu_j^+ = \max(\nu_j, 0), \quad \nu_j^- = -\min(\nu_j, 0).
\]
Let also \(\vert\nuu^{\pm}\vert = \sum_{k = 1}^j\nu_k^{\pm}\). Then we have
\[
C_{\nuu}(x_1,\dots,x_{j+1}) = \left(q(q-1)\right)^{j-\vert\nuu^{+}\vert-\vert\nuu^-\vert}\frac{\prod_{\ell = 0}^{j}B_{\ell}^{(\nu_{\ell},\nu_{\ell+1})}(\mathbf{x})}{\prod_{\ell = 1}^jb_{\ell}^{\nu_{\ell}}(\mathbf{x})},
\]
with the convention that \(\nu_0 = \nu_{j+1} = 0\) and where
\begin{align*}
B_0^{(0,0)}(\mathbf{x}) = &\, 1 + b + \frac{1}{q-1}\left(1-\frac{qb}{a_1}\right)\left(a_1(-q)^{x_1}+(-q)^{-x_1}\right), \\
B_0^{(0,1)}(\mathbf{x}) = & \left(1-a_1(-q)^{x_1}\right)\left(1-b(-q)^{x_1+1}\right), \\
B_0^{(0,-1)}(\mathbf{x}) = & \left(1-(-q)^{-x_1}\right)\left(1-\frac{b}{a_1}(-q)^{-x_1+1}\right),
\end{align*}
and where for \(\ell \in \{1,\dots,j\}\) we have
\begin{align*}
B_{\ell}^{(0,0)}(\mathbf{x}) & = \frac{1}{q-1}\frac{(-q)^{-x_{\ell}-x_{\ell+1}}}{A_{\ell}}\left(1+(-q)^{2x_{\ell}}A_{\ell}\right)\left(1+(-q)^{2x_{\ell+1}}A_{\ell+1}\right)+ 1-\frac{a_{\ell+1}}{q}, \\
B_{\ell}^{(0,1)}(\mathbf{x}) & = \left(1-A_{\ell+1}(-q)^{x_{\ell}+x_{\ell+1}}\right)\left(1-a_{\ell+1}(-q)^{x_{\ell+1}-x_{\ell}}\right), \\
B_{\ell}^{(1,0)}(\mathbf{x}) & = \left(1-A_{\ell+1}(-q)^{x_{\ell}+x_{\ell+1}}\right)\left(1-(-q)^{x_{\ell}-x_{\ell+1}}\right), \\
B_{\ell}^{(1,1)}(\mathbf{x}) & = \left(1-A_{\ell+1}(-q)^{x_{\ell}+x_{\ell+1}}\right)\left(1-A_{\ell+1}(-q)^{x_{\ell}+x_{\ell+1}+1}\right), \\
B_{\ell}^{(-1,0)}(\mathbf{x}) & = \left(1-a_{\ell+1}(-q)^{-x_{\ell}+x_{\ell+1}}\right)\left(1-\frac{(-q)^{-x_{\ell}-x_{\ell+1}}}{A_{\ell}}\right), \\
B_{\ell}^{(-1,1)}(\mathbf{x}) & = \left(1-a_{\ell+1}(-q)^{-x_{\ell}+x_{\ell+1}}\right)\left(1-a_{\ell+1}(-q)^{-x_{\ell}+x_{\ell+1}+1}\right), \\
B_{\ell}^{(0,-1)}(\mathbf{x}) & = \left(1-(-q)^{x_{\ell}-x_{\ell+1}}\right)\left(1-\frac{(-q)^{-x_{\ell}-x_{\ell+1}}}{A_{\ell}}\right), \\
B_{\ell}^{(1,-1)}(\mathbf{x}) & = \left(1-(-q)^{x_{\ell}-x_{\ell+1}}\right)\left(1-(-q)^{x_{\ell}-x_{\ell+1}+1}\right), \\
B_{\ell}^{(-1,-1)}(\mathbf{x}) & = \left(1-\frac{(-q)^{-x_{\ell}-x_{\ell+1}}}{A_{\ell}}\right)\left(1-\frac{(-q)^{-x_{\ell}-x_{\ell+1}+1}}{A_{\ell}}\right),
\end{align*}
and
\begin{align*}
b_{\ell}^0(\mathbf{x}) &= \left(1-A_{\ell}(-q)^{2x_{\ell}+1}\right)\left(1-\frac{(-q)^{-2x_{\ell}+1}}{A_{\ell}}\right), \\
b_{\ell}^1(\mathbf{x}) & = \left(1-A_{\ell}(-q)^{2x_{\ell}}\right)\left(1-A_{\ell}(-q)^{2x_{\ell}+1} \right), \\
b_{\ell}^{-1}(\mathbf{x}) & = \left(1-\frac{(-q)^{-2x_{\ell}}}{A_{\ell}}\right)\left(1-\frac{(-q)^{-2x_{\ell}+1}}{A_{\ell}}\right).
\end{align*}

Similarly, the functions \(D_{\nuu}(n_1,\dots,n_s)\) are given by
\[
D_{\nuu}(n_1,\dots,n_s) = \left(q(q-1)\right)^{j-\vert\nuu^{+}\vert-\vert\nuu^-\vert}\frac{\prod_{\ell = 0}^{j}E_{\ell}^{(\nu_{\ell},\nu_{\ell+1})}(\mathbf{n})}{\prod_{\ell = 1}^je_{\ell}^{\nu_{\ell}}(\mathbf{n})},
\]
with the convention that \(\nu_0 = \nu_{j+1} = 0\) and where
\begin{align*}
E_0^{(0,0)}(\mathbf{n}) = &\, 1 + A_{s+1}(-q)^{2N} + \frac{1}{q-1}\left(1-\frac{qb}{a_1}\right)\left(\frac{a_1}{b}(-q)^{N-N_s}+A_{s+1}(-q)^{N+N_s}\right), \\
E_0^{(0,1)}(\mathbf{n}) = & \left(1-A_{s+1}(-q)^{N_s+N}\right)\left(1-\frac{b}{a_1}A_{s+1}(-q)^{N_s+N+1}\right), \\
E_0^{(0,-1)}(\mathbf{n}) = & \left(1-\frac{a_1}{b}(-q)^{N-N_s}\right)\left(1-(-q)^{N-N_s+1}\right),
\end{align*}
and where for \(\ell \in \{1,\dots,j\}\) we have
\begin{align*}
E_{\ell}^{(0,0)}(\mathbf{n}) & = 1-\frac{a_{s-\ell+2}}{q}+\frac{1}{q-1}\left((-q)^{-N_{s-\ell+1}}+\frac{bA_{s-\ell+2}}{a_1}(-q)^{N_{s-\ell+1}} \right)\\&\times\left(\frac{a_1}{bA_{s-\ell+1}}(-q)^{-N_{s-\ell}}+(-q)^{N_{s-\ell}}\right), \\
E_{\ell}^{(0,1)}(\mathbf{n}) & = \left(1-\frac{b}{a_1}A_{s-\ell+2}(-q)^{2N_{s-\ell}+n_{s-\ell+1}}\right)\left(1-(-q)^{-n_{s-\ell+1}}\right), \\
E_{\ell}^{(1,0)}(\mathbf{n}) & = \left(1-\frac{b}{a_1}A_{s-\ell+2}(-q)^{2N_{s-\ell}+n_{s-\ell+1}}\right)\left(1-a_{s-\ell+2}(-q)^{n_{s-\ell+1}}\right), \\
E_{\ell}^{(1,1)}(\mathbf{n}) & =  \left(1-\frac{b}{a_1}A_{s-\ell+2}(-q)^{2N_{s-\ell}+n_{s-\ell+1}}\right) \left(1-\frac{b}{a_1}A_{s-\ell+2}(-q)^{2N_{s-\ell}+n_{s-\ell+1}+1}\right), \\
E_{\ell}^{(-1,0)}(\mathbf{n}) & = \left(1-(-q)^{-n_{s-\ell+1}}\right)\left(1-\frac{a_1(-q)^{-2N_{s-\ell}-n_{s-\ell+1}}}{bA_{s-\ell+1}} \right), \\
E_{\ell}^{(-1,1)}(\mathbf{n}) & = \left(1-(-q)^{-n_{s-\ell+1}}\right)\left(1-(-q)^{-n_{s-\ell+1}+1}\right), \\
E_{\ell}^{(0,-1)}(\mathbf{n}) & = \left(1-a_{s-\ell+2}(-q)^{n_{s-\ell+1}}\right) \left(1-\frac{a_1(-q)^{-2N_{s-\ell}-n_{s-\ell+1}}}{bA_{s-\ell+1}} \right), \\
E_{\ell}^{(1,-1)}(\mathbf{n}) & = \left(1-a_{s-\ell+2}(-q)^{n_{s-\ell+1}}\right)\left(1-a_{s-\ell+2}(-q)^{n_{s-\ell+1}+1}\right), \\
E_{\ell}^{(-1,-1)}(\mathbf{n}) & = \left(1-\frac{a_1(-q)^{-2N_{s-\ell}-n_{s-\ell+1}}}{bA_{s-\ell+1}} \right)\left(1-\frac{a_1(-q)^{-2N_{s-\ell}-n_{s-\ell+1}+1}}{bA_{s-\ell+1}} \right),
\end{align*}
and
\begin{align*}
e_{\ell}^0(\mathbf{n}) &= \left(1-\frac{b}{a_1}A_{s-\ell+2}(-q)^{2N_{s-\ell+1}+1}\right)\left(1-\frac{a_1(-q)^{-2N_{s-\ell+1}+1}}{bA_{s-\ell+2}} \right), \\
e_{\ell}^1(\mathbf{n}) & = \left(1-\frac{b}{a_1}A_{s-\ell+2}(-q)^{2N_{s-\ell+1}}\right)\left(1-\frac{b}{a_1}A_{s-\ell+2}(-q)^{2N_{s-\ell+1}+1}\right), \\
e_{\ell}^{-1}(\mathbf{n}) & = \left(1-\frac{a_1(-q)^{-2N_{s-\ell+1}}}{bA_{s-\ell+2}} \right)\left(1-\frac{a_1(-q)^{-2N_{s-\ell+1}+1}}{bA_{s-\ell+2}} \right).
\end{align*}

\newpage

\section*{Appendix C: Gauge coefficients in the proof of Theorem \ref{th discrete realization}}

Let us denote by \(A_{i,r}\) the set \([1;i-2]\cup[i+r-1;n-2]\). Starting from the definition (\ref{gauge coefficient rho}), a long but straightforward calculation shows
\begin{align*}
& \frac{\rho^{(1,n-1)}(\mathbf{s})}{\rho^{(i,r)}(\mathbf{s})} \\
= &\  \frac{\left(-q,q^{2\gamma_{[1;n-1]}};-q\right)_{\vss{n-1}}\left(q^{2\gamma_{i+r}},(-q)^{2\vss{i-1}-1}q^{2\gamma_{[1;i+r]}};-q\right)_{\vss{i+r-1}-\vss{i-1}}}{\left(-q,(-q)^{2\vss{i-1}}q^{2\gamma_{[1;i+r-1]}};-q\right)_{\vss{i+r-1}-\vss{i-1}}\left(q^{2\gamma_n},-q^{2\gamma_{[1;n]}-1};-q\right)_{\vss{n-1}}}\\\times&\frac{\left(-q^{2\gamma_{[1;2]}-1},q^{2\gamma_2};-q\right)_{s_1}}{(-q,q^{2\gamma_1};-q)_{s_1}}\frac{\left(-q^{2\gamma_{[1;i+1]}-1};-q\right)_{2\vss{i-1}}}{\left(q^{2\gamma_{[1;i]}};-q\right)_{2\vss{i-1}}\left(q^{2\gamma_{[i;i+r-1]}}\right)^{\vss{i-1}}}
\\\times&\prod_{k\in A_{i,r}}\frac{\left(q^{2\gamma_{k+2}};-q\right)_{s_{k+1}}}{(-q;-q)_{s_{k+1}}}\frac{\left(-q^{2\gamma_{[1;k+2]}-1};-q\right)_{\vss{k}+\vss{k+1}}}{\left(q^{2\gamma_{[1;k+1]}};-q\right)_{\vss{k}+\vss{k+1}}}\frac{\left(1+q^{2\gamma_{[1;k+1]}+2\vss{k}-1}\right)}{\left(q^{2\gamma_{k+1}}\right)^{\vss{k}}}\\\times&\frac{\prod_{k = 0}^{r-1}\left(1+q^{2\gamma_{[1;i+k]}+2\vss{i-1}-1}\right)}{\prod_{k = 1}^{n-2}\left(1+q^{2\gamma_{[1;k+1]}-1}\right)}\prod_{k = 1}^{r-1}\frac{\left(-q^{2\gamma_{[1;i+k+1]}-1};-q\right)_{2\vss{i-1}}}{\left(q^{2\gamma_{[1;i+k]}};-q\right)_{2\vss{i-1}}},
\end{align*}
for \(i\in \{2,\dots,n-1\}\), whereas for \(i = 1\) one finds
\begin{align*}
& \frac{\rho^{(1,n-1)}(\mathbf{s})}{\rho^{(1,r)}(\mathbf{s})} \\
= &\ \frac{\left(-q,q^{2\gamma_{[1;n-1]}};-q\right)_{\vss{n-1}}}{\left(-q,q^{2\gamma_{[1;r]}};-q\right)_{\vss{r}}}\frac{\left(q^{2\gamma_{r+1}},-q^{2\gamma_{[1;r+1]}-1};-q\right)_{\vss{r}}}{\left(q^{2\gamma_n},-q^{2\gamma_{[1;n]}-1};-q\right)_{\vss{n-1}}}\prod_{k= r}^{n-2}\left(q^{2\gamma_{k+1}}\right)^{-\vss{k}}
\\\times&\prod_{k= r}^{n-2}\frac{\left(q^{2\gamma_{k+2}};-q\right)_{s_{k+1}}}{(-q;-q)_{s_{k+1}}}\frac{\left(-q^{2\gamma_{[1;k+2]}-1};-q\right)_{\vss{k}+\vss{k+1}}}{\left(q^{2\gamma_{[1;k+1]}};-q\right)_{\vss{k}+\vss{k+1}}}\frac{1+q^{2\gamma_{[1;k+1]}+2\vss{k}-1}}{1+q^{2\gamma_{[1;k+1]}-1}}.
\end{align*} \newpage

\section*{Appendix D: Weight function and normalization coefficient for the multivariate Bannai--Ito polynomials in Proposition \ref{prop BI orthogonality multivariate}}
	The functions \(w_{\ell_i-N_{i-1}}^{(i)}\) in Proposition \ref{prop BI orthogonality multivariate} vanish if \(\ell_i < N_{i-1}\), whereas if \(\ell_i\geq N_{i-1}\), they have the expression 
	\begin{align*}
	& w_{\ell_i-N_{i-1}}^{(i)}\left(\ell_i,\ell_{i+1}, N_{i-1}, \alpha_1,\dots, \alpha_{i+1}, \beta \right) \\
	= & \, (-1)^{\ell_i-N_{i-1}}\frac{\left(\frac{N_{i-1}}{2}+\frac{\alpha_{[1;i+1]}}{4}+\frac12+\frac{(-1)^{\ell_{i+1}+N_{i-1}}}{2}\left(\ell_{i+1}+\frac{\alpha_{[1;i+1]}}{2}\right)\right)_{\left(\ell_i-N_{i-1}\right)_e}}{\left(\ell_i-N_{i-1}\right)_e!\left(\frac{\alpha_1-\beta}{2}\right)_{\left(\ell_i-N_{i-1}\right)_e+\left(\ell_i-N_{i-1}\right)_{p}}} \\
	\times & \frac{\left(\frac{N_{i-1}}{2}+\frac{\alpha_{[1;i+1]}}{4}-\frac{(-1)^{\ell_{i+1}+N_{i-1}}}{2}\left(\ell_{i+1}+\frac{\alpha_{[1;i+1]}}{2}\right)\right)_{\left(\ell_{i}-N_{i-1}\right)_e+\left(\ell_{i}-N_{i-1}\right)_p}}{\left(\frac{N_{i-1}}{2}+\frac{\alpha_{[1;i]}-\alpha_{i+1}}{4}+\frac12+\frac {(-1)^{\ell_{i+1}+N_{i-1}}}{2}\left(\ell_{i+1}+\frac{\alpha_{[1;i+1]}}{2}\right)\right)_{\left(\ell_i-N_{i-1}\right)_e+\left(\ell_i-N_{i-1}\right)_p}} \\
	\times & \frac{\left(N_{i-1}+\frac{\alpha_{[2;i]}+\beta+1}{2} \right)_{\left(\ell_i-N_{i-1}\right)_e+\left(\ell_i-N_{i-1}\right)_{p}}\left(N_{i-1}+\frac{\alpha_{[1;i]}+1}{2}\right)_{\left(\ell_i-N_{i-1}\right)_e}}{\left(\frac{N_{i-1}}{2}+\frac{\alpha_{[1;i]}-\alpha_{i+1}}{4}+1+\frac{(-1)^{\ell_{i+1}+N_{i-1}}}{2}\left(\ell_{i+1}+\frac{\alpha_{[1;i+1]}}{2}\right)\right)_{\left(\ell_i-N_{i-1}\right)_e}}.
	\end{align*}
	The functions \(h_{n_i}^{(i)}\) are given by
	\begin{align*}
	& h_{n_i}^{(i)}\left(\ell_{i+1}, N_{i-1},n_i, \alpha_1,\dots,\alpha_{i+1}, \beta\right) \\
	= &\, \frac{\left(\ell_{i+1}-N_{i-1}\right)_e!\left(n_i\right)_e!\left(N_{i-1}+\frac{\alpha_{[1;i]}+1}{2}\right)_{\left(\ell_{i+1}-N_{i-1}\right)_e+\left(\ell_{i+1}-N_{i-1}\right)_p}}{\left(\left(\ell_{i+1}-N_{i-1}\right)_e-\left(n_i\right)_e-(n_i)_p(1-\left(\ell_{i+1}-N_{i-1}\right)_p)\right)!} \\
	\times & \frac{\left(N_{i-1}+\left(n_i\right)_e+\frac{\alpha_{[2;i+1]}+\beta+1}{2}\right)_{\left(\ell_{i+1}-N_{i-1}\right)_e+\left(\ell_{i+1}-N_{i-1}\right)_p-\left(n_i\right)_e}}{\left(\left(n_i\right)_e+\left(n_i\right)_p+\frac{\alpha_{i+1}}{2}\right)_{\left(\ell_{i+1}-N_{i-1}\right)_e+\left(\ell_{i+1}-N_{i-1}\right)_p-\left(n_i\right)_e-\left(n_i\right)_p}} \\
	\times &\, \frac{\left(N_{i-1}+\frac{\alpha_{[2;i]}+\beta+1}{2}\right)_{\left(n_i\right)_e+\left(n_i\right)_p}}{\left(\left(N_{i-1}+\left(n_i\right)_e+\frac{\alpha_{[2;i+1]}+\beta+1}{2}\right)_{\left(n_i\right)_e+\left(n_i\right)_p}\right)^2} \\
	\times &\, \frac{\left(\frac{1}{2}\left(N_{i-1}+\ell_{i+1}+\alpha_{[2;i+1]}+\beta+1+(\ell_{i+1}-N_{i-1})_p(\alpha_1-\beta-1)\right) \right)_{\left(n_i\right)_e+(n_i)_p}}{\left(\frac{\alpha_1-\beta}{2}\right)_{\left(\ell_{i+1}-N_{i-1}\right)_e-(n_i)_e+\left(\ell_{i+1}-N_{i-1}\right)_p(1-\left(n_i\right)_p)}}\\
	\times & \left(\frac{1}{2}\left(N_{i-1}+\ell_{i+1}+\alpha_{[1;i+1]}+1-(\ell_{i+1}-N_{i-1})_p(\alpha_1-\beta-1)\right) \right)_{\left(n_i\right)_e}.
	\end{align*}
	\newpage
	
\section*{Appendix E: Explicit expressions for the functions $\widetilde{C}_{\nuu}(\mathbf{x})$ in Proposition \ref{prop Iliev x -1} and $\widetilde{D}_{\nuu}(\mathbf{n})$ in Proposition \ref{prop Iliev n -1}}

For \(\nuu\in\{-1,0,1\}^j\) we have
\[
\widetilde{C}_{\nuu}(x_1,\dots,x_{j+1}) = \frac{\prod_{\ell = 0}^{j}\widetilde{B}_{\ell}^{(\nu_{\ell},\nu_{\ell+1})}(\mathbf{x})}{\prod_{\ell = 1}^j\tilde{b}_{\ell}^{\nu_{\ell}}(\mathbf{x})},
\]
with the convention that \(\nu_0 = \nu_{j+1} = 0\) and where
\begin{align*}
\widetilde{B}_0^{(0,0)}(\mathbf{x}) = &\, \frac12\left(-\beta+(-1)^{x_1}(2x_1+\alpha_1)(\beta-\alpha_1+1)\right), \\
\widetilde{B}_0^{(0,1)}(\mathbf{x}) = &\, -\frac12\left((2x_1+\alpha_1+\beta+1)+(-1)^{x_1}(\beta-\alpha_1+1) \right), \\
\widetilde{B}_0^{(0,-1)}(\mathbf{x}) = &\, \frac12\left((2x_1+\alpha_1-\beta-1)+(-1)^{x_1}(\beta-\alpha_1+1)\right),
\end{align*}
and where for \(\ell \in \{1,\dots,j\}\) we have
\begin{align*}
\widetilde{B}_{\ell}^{(0,0)}(\mathbf{x}) & = \, (1-\alpha_{\ell+1})-(-1)^{x_{\ell}+x_{\ell+1}}(2x_{\ell}+\alpha_{[1;\ell]})(2x_{\ell+1}+\alpha_{[1;\ell+1]}), \\
\widetilde{B}_{\ell}^{(0,1)}(\mathbf{x}) & = \, -(2x_{\ell+1}+\alpha_{[1;\ell+1]}+\alpha_{\ell+1})+(-1)^{x_{\ell}+x_{\ell+1}}(2x_{\ell}+\alpha_{[1;\ell]}), \\
\widetilde{B}_{\ell}^{(1,0)}(\mathbf{x}) & = \, -(2x_{\ell}+\alpha_{[1;\ell+1]})+(-1)^{x_{\ell}+x_{\ell+1}}(2x_{\ell+1}+\alpha_{[1;\ell+1]}), \\
\widetilde{B}_{\ell}^{(1,1)}(\mathbf{x}) & = \, -(2x_{\ell}+2x_{\ell+1}+2\alpha_{[1;\ell+1]}+1)+(-1)^{x_{\ell}+x_{\ell+1}+1}, \\
\widetilde{B}_{\ell}^{(-1,0)}(\mathbf{x}) & = \, (2x_{\ell}+\alpha_{[1;\ell]}-\alpha_{\ell+1})-(-1)^{x_{\ell}+x_{\ell+1}}(2x_{\ell+1}+\alpha_{[1;\ell+1]}), \\
\widetilde{B}_{\ell}^{(-1,1)}(\mathbf{x}) & = \, (2x_{\ell}-2x_{\ell+1}-2\alpha_{\ell+1}-1)+(-1)^{x_{\ell}+x_{\ell+1}}, \\
\widetilde{B}_{\ell}^{(0,-1)}(\mathbf{x}) & = \, (2x_{\ell+1}+\alpha_{[1;\ell]})-(-1)^{x_{\ell}+x_{\ell+1}}(2x_{\ell}+\alpha_{[1;\ell]}), \\
\widetilde{B}_{\ell}^{(1,-1)}(\mathbf{x}) & = \, -(2x_{\ell}-2x_{\ell+1}+1)+(-1)^{x_{\ell}+x_{\ell+1}}, \\
\widetilde{B}_{\ell}^{(-1,-1)}(\mathbf{x}) & = \, (2x_{\ell}+2x_{\ell+1}+2\alpha_{[1;\ell]}-1)-(-1)^{x_{\ell}+x_{\ell+1}},
\end{align*}
and
\begin{align*}
\tilde{b}_{\ell}^0(\mathbf{x}) &= -(2x_{\ell}+\alpha_{[1;\ell]}+1)(2x_{\ell}+\alpha_{[1;\ell]}-1), \\
\tilde{b}_{\ell}^1(\mathbf{x}) & = -2(2x_{\ell}+\alpha_{[1;\ell]}+1), \\
\tilde{b}_{\ell}^{-1}(\mathbf{x}) & = 2(2x_{\ell}+\alpha_{[1;\ell]}-1).
\end{align*}

Similarly, the functions \(\widetilde{D}_{\nuu}(n_1,\dots,n_s)\), with \(\nuu\in\{-1,0,1\}^j\), are given by
\[
\widetilde{D}_{\nuu}(n_1,\dots,n_s) = \frac{\prod_{\ell = 0}^{j}\widetilde{E}_{\ell}^{(\nu_{\ell},\nu_{\ell+1})}(\mathbf{n})}{\prod_{\ell = 1}^j\tilde{e}_{\ell}^{\nu_{\ell}}(\mathbf{n})}
\]
with the convention that \(\nu_0 = \nu_{j+1} = 0\) and where
\begin{align*}
\widetilde{E}_0^{(0,0)}(\mathbf{n}) = &\, -\frac12\left((2N+\alpha_{[1;s+1]})+(-1)^{N+N_s}(2N_s+\alpha_{[2;s+1]}+\beta)(\alpha_1-\beta-1)\right), \\
\widetilde{E}_0^{(0,1)}(\mathbf{n}) = &\, -\frac12\left((2N+2N_s+2\alpha_{[2;s+1]}+\alpha_1+\beta+1)+(-1)^{N-N_s}(\beta-\alpha_1+1) \right), \\
\widetilde{E}_0^{(0,-1)}(\mathbf{n}) = &\, -\frac12\left((2N-2N_s+\alpha_1-\beta+1)-(-1)^{N-N_s}(\beta-\alpha_1+1)\right),
\end{align*}
and where for \(\ell \in \{1,\dots,j\}\) we have
\begin{align*}
\widetilde{E}_{\ell}^{(0,0)}(\mathbf{n}) & = \, (1-\alpha_{s-\ell+2})-(-1)^{n_{s-\ell+1}}(2N_{s-\ell+1}+\alpha_{[2;s-\ell+2]}+\beta)(2N_{s-\ell}+\alpha_{[2;s-\ell+1]}+\beta), \\
\widetilde{E}_{\ell}^{(0,1)}(\mathbf{n}) & = \, -(2N_{s-\ell}+\alpha_{[2;s-\ell+2]}+\beta)+(-1)^{n_{s-\ell+1}}(2N_{s-\ell+1}+\alpha_{[2;s-\ell+2]}+\beta), \\
\widetilde{E}_{\ell}^{(1,0)}(\mathbf{n}) & = \, -(2N_{s-\ell+1}+\alpha_{[2;s-\ell+2]}+\alpha_{s-\ell+2}+\beta)+(-1)^{n_{s-\ell+1}}(2N_{s-\ell}+\alpha_{[2;s-\ell+1]}+\beta), \\
\widetilde{E}_{\ell}^{(1,1)}(\mathbf{n}) & = \, -(4N_{s-\ell}+2n_{s-\ell+1}+2\alpha_{[2;s-\ell+2]}+2\beta+1)-(-1)^{n_{s-\ell+1}}, \\
\widetilde{E}_{\ell}^{(-1,0)}(\mathbf{n}) & = \, (2N_{s-\ell+1}+\alpha_{[2;s-\ell+1]}+\beta)-(-1)^{n_{s-\ell+1}}(2N_{s-\ell}+\alpha_{[2;s-\ell+1]}+\beta), \\
\widetilde{E}_{\ell}^{(-1,1)}(\mathbf{n}) & = \, 2n_{s-\ell+1}-1+(-1)^{n_{s-\ell+1}}, \\
\widetilde{E}_{\ell}^{(0,-1)}(\mathbf{n}) & = \, (2N_{s-\ell}+\alpha_{[2;s-\ell+1]}-\alpha_{s-\ell+2}+\beta)-(-1)^{n_{s-\ell+1}}(2N_{s-\ell+1}+\alpha_{[2;s-\ell+2]}+\beta), \\
\widetilde{E}_{\ell}^{(1,-1)}(\mathbf{n}) & = \, -(2n_{s-\ell+1}+2\alpha_{s-\ell+2}+1)+(-1)^{n_{s-\ell+1}}, \\
\widetilde{E}_{\ell}^{(-1,-1)}(\mathbf{n}) & = \, (4N_{s-\ell}+2n_{s-\ell+1}+2\alpha_{[2;s-\ell+1]}+2\beta-1)-(-1)^{n_{s-\ell+1}},
\end{align*}
and
\begin{align*}
\tilde{e}_{\ell}^0(\mathbf{n}) &= -(2N_{s-\ell+1}+\alpha_{[2;s-\ell+2]}+\beta+1)(2N_{s-\ell+1}+\alpha_{[2;s-\ell+2]}+\beta-1), \\
\tilde{e}_{\ell}^1(\mathbf{n}) & = -2(2N_{s-\ell+1}+\alpha_{[2;s-\ell+2]}+\beta+1), \\
\tilde{e}_{\ell}^{-1}(\mathbf{n}) & = 2(2N_{s-\ell+1}+\alpha_{[2;s-\ell+2]}+\beta-1).
\end{align*}


\newpage
\begin{thebibliography}{99}
	\bibitem{Askey&Wilson-1985}
	R.~Askey, J.A.~Wilson, \newblock{Some basic hypergeometric orthogonal polynomials that generalize Jacobi polynomials.} {\em Memoirs Amer. Math. Soc.} {\bf 319} (1985).
	
	\bibitem{Bannai&Ito-1984}
	E. Bannai, T. Ito, Algebraic combinatorics. I.
	Association schemes. The Benjamin/Cummings Publishing Co., Inc., Menlo Park, CA, 1984.
	
	\bibitem{Baseilhac-2006}
	P.~Baseilhac, \newblock{The \(q\)-deformed analogue of the Onsager algebra: Beyond the Bethe ansatz approach.} {\em Nucl. Phys. B} {\bf 754} (2006), 309--328.
	
	\bibitem{Baseilhac&Crampe&Pimenta-2018}
	P.~Baseilhac, N.~Cramp\'{e}, R.A.~Pimenta, \newblock{Higher rank classical analogs of the Askey--Wilson algebra from the \(\mathfrak{sl}_N\) Onsager algebra.} {\em J. Math. Phys.} {\bf 60} (2019), 081703.
	
	\bibitem{Baseilhac&Martin-2015}
	P.~Baseilhac, X.~Martin, \newblock{A bispectral \(q\)-hypergeometric basis for a class of quantum integrable models.} {\em J. Math. Phys.} {\bf 59} (2018), 011704.
	
	\bibitem{Baseilhac&Vinet&Zhedanov-2017}
	P.~Baseilhac, L.~Vinet, A.~Zhedanov, \newblock{The \(q\)-Onsager algebra and multivariable \(q\)-special functions.} {\em J. Phys. A: Math. Theor.} {\bf 50}, (2017), 395201.
	
	\bibitem{DeBie&DeClercq&vandeVijver-2018}
	H.~De Bie, H.~De Clercq, W.~van de Vijver, \newblock{The higher rank \(q\)-deformed Bannai--Ito and Askey--Wilson algebra.} {\em Commun. Math. Phys.} {\bf 374}(1) (2020), 277--316.
	
	\bibitem{DeBie&Genest&Lemay&Vinet-2017}
	H.~De Bie, V.X.~Genest, J.-M.~Lemay, L.~Vinet, \newblock{A superintegrable model with reflections on $S^{n-1}$ and the higher rank Bannai--Ito algebra.} {\em J. Phys. A: Math. Theor.} {\bf 50} (2017), 195202.
	
	\bibitem{DeBie&Genest&vandeVijver&Vinet-2016}
	H.~De Bie, V.X.~Genest, W.~van de Vijver, L.~Vinet, \newblock{A higher rank Racah algebra and the $\mathbb{Z}_2^n$ Laplace-Dunkl operator.} {\em J. Phys. A: Math. Theor.} {\bf 51} (2018), 025203.
	
	\bibitem{DeBie&Genest&Vinet-2016}
	H.~De Bie, V.X.~Genest, L.~Vinet,
	\newblock {The \(\mathbb{Z}_2^n\) Dirac--Dunkl operator and a higher rank Bannai--Ito algebra}.
	\newblock {\em Adv. Math.} {\bf 303} (2016), 390--414.
	
	\bibitem{DeBie&Genest&Vinet-2016-2}
	H.~De Bie, V.X.~Genest, L.~Vinet,
	\newblock{A Dirac--Dunkl equation on \(S^2\) and the Bannai--Ito algebra.}
	\newblock {\em Commun. Math. Phys.} {\bf 344} (2016), 447--464.
	
	\bibitem{DeBie&Oste&VanderJeugt-2018}
	H.~De Bie, R.~Oste, J.~Van der Jeugt, \newblock{On the algebra of symmetries of Laplace and Dirac operators.} {\em Lett. Math. Phys.} {\bf 108} (2018), 1905--1953.	
	
	\bibitem{DeBie&vandeVijver-2018}
	H.~De Bie, W.~van de Vijver, \newblock{A discrete realization of the higher rank Racah algebra.} {\em Constr. Approx.} (2019), \textit{https://doi.org/10.1007/s00365-019-09475-0}.
	
	\bibitem{DeClercq-2019}
	H.~De Clercq, \newblock{Higher rank relations for the Askey--Wilson and $q$-Bannai--Ito algebra.} {\em SIGMA} {\bf 15} (2019), Paper 099, 32pp.
	
	\bibitem{Duistermaat&Grunbaum-1986}
	J.J.~Duistermaat, F.A.~Gr\"{u}nbaum, \newblock{Differential equations in the spectral parameter.} {\em Commun. Math. Phys.} {\bf 103} (1986), no. 2, 177--240.
	
	\bibitem{Frenkel&Khovanov-1997}
	I.B.~Frenkel, M.G.~Khovanov, \newblock{Canonical bases in tensor products and graphical calculus for \(U_q(\mathfrak{sl}_2)\)}, {\em Duke Math. J.} {\bf 87} (1997), no. 3, 409--480. 
	
	\bibitem{Gasper&Rahman-2005}
	G.~Gasper and M.~Rahman, \newblock{Some systems of multivariable orthogonal Askey--Wilson polynomials.} In: \newblock{Theory and applications of special functions,} p. 209--219, {\em Dev. Math.} {\bf 13}, Springer, New York, 2005.
	
	\bibitem{Gasper&Rahman-2007}
	G. Gasper and M. Rahman,
	Some systems of multivariable orthogonal $q$-Racah polynomials. 
	{\em Ramanujan J.} {\bf 13} (2007), 389--405.
	
	\bibitem{Genest&Iliev&Vinet-2018}
	V.X. Genest, P. Iliev, L.~Vinet,	 
	Coupling coefficients of $\mathfrak{su}_q(1,1)$ and multivariate $q$-Racah polynomials.  
	{\em Nucl. Phys. B} {\bf 927} (2018), 97--123.
	
	\bibitem{Genest&Lapointe&Vinet-2018}
	V.X.~Genest, L.~Lapointe, L.~Vinet, \newblock{$\mathfrak{osp}(1\vert 2)$ and generalized Bannai--Ito algebras.} {\em Trans. Amer. Math. Soc.} {\bf 372} (2019), 4127--4148.
	
	\bibitem{Genest&Vinet&Zhedanov-2012}
	V.~X.~Genest, L.~Vinet, A.~Zhedanov,
	\newblock{The Bannai--Ito polynomials as Racah coefficients of the \(\mathfrak{sl}_{-1}(2)\) algebra.}
	\newblock{\em Proc. Am. Math. Soc.}, {\bf 142}(5) (2014), 1545--1560.
	
	\bibitem{Genest&Vinet&Zhedanov-2016}
	V.X. Genest, L.~Vinet, A.~Zhedanov,	 
	The quantum superalgebra $\mathfrak{osp}_q(1|2)$ and a $q$-generalization of the Bannai--Ito polynomials. {\em Commun. Math. Phys.} {\bf 344} (2016), 465--481.
	
	\bibitem{Geronimo&Iliev-2010}
	J. S.~Geronimo, P.~Iliev, \newblock{Bispectrality of the multivariable Racah-Wilson polynomials.} {\em Constr. Approx.} {\bf 31} (2010), 417--457.
	
	\bibitem{Groenevelt-2018}
	W.~Groenevelt, \newblock{A quantum algebra approach to multivariate Askey--Wilson polynomials.} To appear in \emph{Int. Math. Res. Not.}, arXiv:1809.04327.
	
	\bibitem{Iliev-2011}
	P.~Iliev, \newblock{Bispectral commuting difference operators for multivariable Askey--Wilson polynomials.} {\em Trans. Amer. Math. Soc.}, {\bf 363}(3) (2011), 1577--1598.	
	
	\bibitem{Iliev-2012}
	P.~Iliev, \newblock{A Lie theoretic interpretation of multivariate hypergeometric polynomials.} {\em Compositio Mathematica} {\bf 148}(3) (2012), 991--1002.
	
	\bibitem{Iliev&Xu-2017}
	P.~Iliev, Y.~Xu, \newblock{Hahn polynomials on polyhedra and quantum integrability.} {\em Adv. Math.} {\bf 364} (2020), 107032.
	
	\bibitem{Khovanov-1997}
	M.G.~Khovanov, \newblock{Graphical calculus, canonical bases and Kazhdan-Lusztig theory}. PhD thesis, Yale University, 1997.
	
	\bibitem{Kirillov&Reshetikhin-1989}
	A.N.~Kirillov, N.Yu.~Reshetikhin, \newblock{Representations of the algebra \(U_q(\mathfrak{sl}_2)\), \(q\)-orthogonal polynomials and invariants of links.} In: {Infinite dimensional Lie algebras and groups (V.G. Kac (ed.))} {\em Adv. Ser. in Math. Phys.} {\bf 7}, 285–-339, World Scientific, Singapore, 1989.
	
	\bibitem{Koekoek&Lesky&Swarttouw-2010}
	R.~Koekoek, P.A.~Lesky, R.F.~Swarttouw,
	\newblock{Hypergeometric Orthogonal Polynomials and Their $q$-Analogues}.
	\newblock{Springer}, 2010.
	
	\bibitem{Koelink-1996}
	H.T.~Koelink, \newblock{Askey--Wilson polynomials and the quantum \(SU(2)\) group: survey and applications.} {\em Acta Appl. Math.} {\bf 44} (1996), 295--352.
	
	\bibitem{Koornwinder-1992}
	T.H.~Koornwinder,
	\newblock{Askey--Wilson polynomials for root systems of type BC.} In: \newblock{Hypergeometric functions on domains of positivity, Jack polynomials and applications. (D. St. P. Richards (ed.))} {\em Contemp. Math.}, vol. 138, Amer. Math. Soc., Providence, R.I., 1992, pp. 189--204.
	
	\bibitem{Lemay&Vinet-2018}
	J.-M.~Lemay, L.~Vinet,
	\newblock{Bivariate Bannai--Ito polynomials.} \newblock{\em J. Math. Phys.} {\bf 59}, (2018), 121703.
	
	\bibitem{Macdonald-1988}
	I. G.~Macdonald, \newblock{A new class of symmetric functions.} Actes 20e S\'{e}minaire Lotharingen, {\em Publ. Inst. Rech. Math. Av.}, Strasbourg, France, (1988), pp. 131--171.
	
	\bibitem{Noumi&Mimachi-1990}
	M.~Noumi, K.~Mimachi, \newblock{Askey--Wilson polynomials and the quantum group \(SU_q(2)\).} {\em Proc. Japan Acad. Ser. A Math. Sci.} {\bf 66}, (1990), no. 6, 146--149.
	
	\bibitem{Post&Walter-2017}
	S.~Post, A.~Walter, \newblock{A higher rank extension of the Askey--Wilson algebra.} arXiv:1705.01860v2.
	
	\bibitem{Rosengren-2000}
	H.~Rosengren, \newblock{A new quantum algebraic interpretation of the Askey--Wilson polynomials.} In: {\(q\)-series from a contemporary perspective (M.E.H. Ismail and D.W. Stanton (eds.))}, {\em Comtemp. Math.} {\bf 254}, 371--394, Amer. Math. Soc., Providence, RI, 2000.
	
	\bibitem{Rosengren-2007}
	H.~Rosengren, \newblock{An elementary approach to \(6j\)-symbols (classical, quantum, rational, trigonometric and elliptic)}, {\em Ramanujan J.} {\bf 13} (2007), 133--168.
	
	\bibitem{Shibukawa-1992}
	Y.~Shibukawa, \newblock{Clebsch-Gordan coefficients for \(U_q(\mathfrak{su}(1,1))\) and \(U_q(\mathfrak{sl}(2))\) and linearization formula of matrix elements.} {\em Publ. Res. I. Math. Sci.}, Kyoto Univ. {\bf 28} (1992), 775--807.
	
	\bibitem{Stokman-2000}
	J.V.~Stokman, \newblock{Koornwinder polynomials and affine Hecke algebras.} {\em Internat. Math. Res. Notices}, {\bf 19} (2000), 1006--1042.
	
	\bibitem{Stokman-2018}
	J.~Stokman, \newblock{Generalized Onsager algebras.} To appear in \emph{Algebr. Repres. Theor.}, arXiv:1810.07408.
	
	\bibitem{Stokman&vanDiejen-1998}
	J.V.~Stokman, J.F.~van Diejen, \newblock{Multivariable \(q\)-Racah polynomials.} {\em Duke Math. J.}, {\bf 91} (1998), 89-136.
	
	\bibitem{Terwilliger-2001}
	P.~Terwilliger, 
	\newblock{Two linear transformations each tridiagonal with respect to an eigenbasis of the other.} {\em Linear Algebra Appl.} {\bf 330} (2001), 149--203.
	
	\bibitem{Terwilliger-2011}
	P.~Terwilliger, 
	\newblock{The universal Askey--Wilson algebra.}
	{\em SIGMA} {\bf 7} (2011), Paper 069, 24 pp. 
	
	\bibitem{Tratnik-1991}
	M.V.~Tratnik, \newblock{Some multivariable orthogonal polynomials of the Askey tableau-discrete families.} {\em J. Math. Phys.} {\bf 32} (1991), 2337--2342. 
	
	\bibitem{Tsujimoto&Vinet&Zhedanov-2011}
	S.~Tsujimoto, L.~Vinet, A.~Zhedanov, \newblock{From \(\mathfrak{sl}_q(2)\) to a parabosonic Hopf algebra.} {\em SIGMA} {\bf 7} (2011), 93--105.
	
	\bibitem{Tsujimoto&Vinet&Zhedanov-2012}
	S.~Tsujimoto, L.~Vinet, A.~Zhedanov, \newblock{Dunkl shift operators and Bannai--Ito polynomials.} {\em Adv. Math.} {\bf 229}(4), (2012), 2123--2158.
	
	\bibitem{Vinet&Zhedanov-2012}
	L.~Vinet, A.~Zhedanov, 
	\newblock{A limit \(q = -1\) for the big \(q\)-Jacobi polynomials.}
	\newblock{\em Trans. Amer. Math. Soc.}, {\bf 374}(10) (2012), 5391--5507.
	
	\bibitem{Zachos-1992}
	C.K.~Zachos,
	\newblock{Altering the symmetry of wavefunctions in quantum algebras and supersymmetry.} {\em Mod. Phys. Lett. A} {\bf 7}(18) (1992), 1595--1600.
	
	\bibitem{Zhedanov-1991}
	A.S.~Zhedanov,
	\newblock {``Hidden symmetry" of the Askey--Wilson polynomials}.
	\newblock {\em Theor. Math. Phys.} {\bf 89} (1991), 1146--1157.
\end{thebibliography}
\end{document}